\documentclass{amsart}
\usepackage[utf8]{inputenc}
\usepackage{amssymb}
\usepackage{mathtools}
\usepackage[lite]{amsrefs}
\renewcommand*{\MR}[1]{ \href{http://www.ams.org/mathscinet-getitem?mr=#1}{MR \textbf{#1}}}

\renewcommand{\PrintDOI}[1]{\href{http://dx.doi.org/\detokenize{#1}}{doi: \detokenize{#1}}}

\BibSpec{book}{%
  +{}  {\PrintPrimary}                {transition}
  +{,} { \textit}                     {title}
  +{.} { }                            {part}
  +{:} { \textit}                     {subtitle}
  +{,} { \PrintEdition}               {edition}
  +{}  { \PrintEditorsB}              {editor}
  +{,} { \PrintTranslatorsC}          {translator}
  +{,} { \PrintContributions}         {contribution}
  +{,} { }                            {series}
  +{,} { \voltext}                    {volume}
  +{,} { }                            {publisher}
  +{,} { }                            {organization}
  +{,} { }                            {address}
  +{,} { \PrintDateB}                 {date}
  +{,} { }                            {status}
  +{}  { \parenthesize}               {language}
  +{}  { \PrintTranslation}           {translation}
  +{;} { \PrintReprint}               {reprint}
  +{.} { }                            {note}
  +{.} {}                             {transition}
  +{} { \PrintDOI}                   {doi}
  +{} { available at \url}            {eprint}
  +{}  {\SentenceSpace \PrintReviews} {review}
}

\usepackage{enumitem} 
\setlist[enumerate,1]{label=\textup{(\arabic*)}}
\usepackage{tikz}
\usetikzlibrary{matrix,calc,arrows,decorations.pathmorphing}
\tikzset{node distance=2cm, auto}

\tikzset{cd/.style=matrix of math nodes,row sep=2em,column sep=2em, text height=1.5ex, text depth=0.5ex}
\tikzset{cdar/.style=->,auto}
\tikzset{mid/.style={anchor=mid}} 
\tikzset{narrowfill/.style={inner sep=1pt, fill=white}}

\usepackage{tikz-cd}

\usepackage{microtype}
\usepackage[pdftitle={Stone duality and quasi-orbit spaces for C*-inclusions},   pdfauthor={Bartosz Kwa\'sniewski, Ralf Meyer},  pdfsubject={Mathematics}]{hyperref}

\numberwithin{equation}{section}
\newtheorem{theorem}[equation]{Theorem}
\newtheorem{lemma}[equation]{Lemma}
\newtheorem{proposition}[equation]{Proposition}
\newtheorem{corollary}[equation]{Corollary}
\theoremstyle{definition}
\newtheorem{definition}[equation]{Definition}
\theoremstyle{remark}
\newtheorem{remark}[equation]{Remark}
\newtheorem{example}[equation]{Example}

\DeclareMathOperator{\Aut}{Aut}
\DeclareMathOperator{\Prime}{Prime}
\DeclareMathOperator{\Pair}{Pair}

\DeclareMathOperator{\pt}{pt}

\newcommand{\OO}{\mathcal O}
\newcommand{\TT}{\mathcal T}

\newcommand*{\nb}{\nobreakdash}
\newcommand*{\Star}{\(^*\)\nobreakdash-}

\newcommand*{\C}{\mathbb C}
\newcommand*{\Z}{\mathbb Z}

\newcommand*{\N}{\mathbb N}
\newcommand*{\T}{\mathbb T}

\newcommand*{\Ideals}{\mathbb I}
\newcommand*{\Open}{\mathbb O}

\newcommand*{\Bound}{\mathbb B}
\newcommand*{\Comp}{\mathbb K}
\newcommand*{\Mat}{\mathbb M}
\newcommand*{\red}{\lambda}
\newcommand*{\Cred}{\mathrm{C}^*_\red} 

\renewcommand*{\SS}{\mathcal S}

\newcommand*{\blank}{\textup{\textvisiblespace}}

\newcommand*{\Cst}{\textup C^*}
\newcommand*{\Mult}{\mathcal M}

\newcommand*{\Cont}{\textup C}
\newcommand*{\Contc}{\Cont_\textup c} 

\newcommand*{\prid}[1][p]{\mathfrak{#1}} 

\newcommand*{\hot}{\mathbin{\widehat\otimes}}
\newcommand{\idealin}{\mathrel{\triangleleft}} 

\newcommand*{\Id}{\textup{Id}}
\newcommand*{\Ad}{\textup{Ad}}

\newcommand*{\Hils}{\mathcal H}

\newcommand*{\U}{\mathcal U}
\newcommand*{\A}{\mathcal A}
\newcommand*{\B}{\mathcal B}
\newcommand*{\Bgr}{\mathcal B}
\newcommand*{\defeq}{\mathrel{\vcentcolon=}}
\newcommand*{\congto}{\xrightarrow\sim}

\DeclarePairedDelimiter{\ket}{\lvert}{\rangle}
\DeclarePairedDelimiter{\bra}{\langle}{\rvert}
\DeclarePairedDelimiterX{\braket}[2]{\langle}{\rangle}{#1\,\delimsize\vert\,\mathopen{}#2}
\DeclarePairedDelimiterX{\BRAKET}[2]{\langle}{\rangle}{\!\delimsize\langle#1\,\delimsize\vert\,\mathopen{}#2\delimsize\rangle\!}
\DeclarePairedDelimiterX{\setgiven}[2]{\{}{\}}{#1\,{:}\,\mathopen{}#2}

\newcommand*{\into}{\rightarrowtail}
\newcommand*{\onto}{\twoheadrightarrow}

\begin{document}
\title[Stone duality and quasi-orbit spaces for C*-inclusions]{Stone duality and quasi-orbit spaces\\for generalised C*-inclusions}

\author{Bartosz Kosma Kwa\'sniewski}
\email{bartoszk@math.uwb.edu.pl}
\address{Faculty of Mathematics\\
  University  of Bia\l ystok\\
  ul.\@ K.~Cio\l kowskiego 1M\\
  15-245 Bia\l ystok\\
  Poland}
\author{Ralf Meyer}
\email{rmeyer2@uni-goettingen.de}
\address{Ralf Meyer\\
  Mathematisches Institut\\
  Georg-August-Universit\"at G\"ottingen\\
  Bunsenstra\ss e 3--5\\
  37073 G\"ottingen\\
  Germany}

\thanks{This work was partially supported by Polish Minister of
  Science and Higher Education, internal grant number BST-150,
  University of Bia\l ystok.}


\begin{abstract}
  Let~\(A\) be a \(\Cst\)\nb-subalgebra
  of the multiplier algebra~\(\Mult(B)\)
  of a \(\Cst\)\nb-algebra~\(B\).
  Exploiting the duality between sober spaces and spatial locales,
  and the adjunction between
  restriction and induction for ideals in \(A\)
  and~\(B\), we identify conditions that allow to define a
  quasi-orbit space
  and a quasi-orbit map for  \(A\subseteq \Mult(B)\).
  These objects generalise classical notions for group
  actions.  We characterise when the quasi-orbit space
  is an open quotient of the primitive ideal space
  of~\(A\)
  and when the quasi-orbit map is open and surjective.  We apply
  these results to cross section \(\Cst\)\nb-algebras
  of Fell bundles over locally compact groups, regular
  \(\Cst\)\nb-inclusions, tensor products,
  relative Cuntz--Pimsner algebras, and crossed products for actions
  of locally compact Hausdorff groupoids and quantum groups.
\end{abstract}
\maketitle

\section{Introduction}
\label{sec:introduction}

There are many different ways to build a \(\Cst\)\nb-algebra~\(B\)
as a crossed product for a \(\Cst\)\nb-algebra~\(A\)
with some kind of dynamics.  The dynamics may be, for instance, an
action of a locally compact group, groupoid, an inverse semigroup, a
semigroup or a quantum group.  The theory for each type of crossed
product aims at understanding the structure of~\(B\)
using the dynamics on~\(A\).
Here we are interested in the ideal lattice~\(\Ideals(B)\)
and the primitive ideal space~\(\check{B}\)
of~\(B\).
For a locally compact amenable group~\(G\) and a separable
\(\Cst\)\nb-dynamical
system \((A,G,\alpha)\),
the primitive ideal spaces of~\(A\)
and of the crossed product \(B\defeq A\rtimes_\alpha G\)
are linked by a quasi-orbit map
\(\varrho\colon \check{B}\to \check{A}/{\sim}\).
Its target is the quasi-orbit space~\(\check{A}/{\sim}\),
where \(\prid_1 \sim \prid_2\)
for \(\prid_1, \prid_2\in\check{A}\)
if and only if
\(\overline{G\cdot \prid_1} = \overline{G\cdot \prid_2}\).
Both \(\varrho\)
and the quotient map \(\check{A} \to \check{A}/{\sim}\)
are open, continuous and surjective.  These are well known results.
But even the existence of the quasi-orbit map is non-trivial.
Quasi-orbit spaces are a key ingredient in the Effros--Hahn
Conjecture (see \cites{Effros-Hahn:Transformation_groups,
  GootmanRosenberg.StructureOfCrossedProducts}, quasi-orbit spaces
seem somewhat implicit in the groupoid version of this conjecture
in~\cite{Renault:Ideal_structure}).  And they are interesting
objects in their own right
(compare~\cite{GootmanLazar.DualityCrossedProduct}).  Moreover, under
some freeness assumptions, the quasi-orbit map is a homeomorphism
\(\check{B}\cong \check{A}/{\sim}\) (see
\cites{Effros-Hahn:Transformation_groups, Green:Local_twisted,
  GootmanRosenberg.StructureOfCrossedProducts,
  Williams:Topology_primitive, Leung-Ng:Invariant_twisted,
  Takai:Quasi-orbit_continuous, Zeller-Meier:Produits_croises} for
the classical case of group actions,
\cite{Giordano-Sierakowski:Purely_infinite}*{Theorem~3.2} for
partial group actions,
\cite{Kwasniewski-Szymanski:Pure_infinite}*{Theorem~6.8} for Fell
bundles over discrete groups, or
\cite{Boenicke-Li:Ideal}*{Theorem~3.17} for a recent result for
groupoid \(\Cst\)\nb-algebras of étale groupoids).

In this paper, we provide a general framework for existence and
properties of quasi-orbit spaces and quasi-orbit maps that are
indispensable in the study of primitive ideal spaces of various
\(\Cst\)\nb-algebraic constructions.  Our main technical tool is
Theorem~\ref{the:open_surjective}, which characterises open and
surjective maps in terms of ideal lattices.
This is a result in pointfree topology.  It uses a version of Stone
duality or, more precisely, a pair of adjoint functors between the
categories of topological spaces and locales.  It
is interesting in its own right.
Its proof is inspired by the proof that all prime ideals in a
separable \(\Cst\)\nb-algebra~\(A\) are primitive.

Let \(\Ideals(A)\) be the lattice of ideals in \(A\).  For an ordinary
\(\Cst\)\nb-inclusion \(A\subseteq B\),
it is customary to call \(J\in \Ideals(B)\)
\emph{induced} from \(I\in \Ideals(A)\)
if \(J=BIB\)
and to call~\(I\)
the \emph{restriction} of~\(J\)
if \(I=J\cap A\).
We allow the more general situation of a \Star{}homomorphism
\(\varphi\colon A\to \Mult(B)\)
to the multiplier algebra of~\(B\).
We speak of a \emph{generalised \(\Cst\)\nb-inclusion}
if~\(\varphi\)
is injective.
Let \(I\in \Ideals(A)\), \(J\in \Ideals(B)\).
The induction map \(i\colon \Ideals(A)\to\Ideals(B)\)
is defined by \(i(I)=B\varphi(I)B\)
as expected.  The restriction
map \(r\colon \Ideals(B) \to \Ideals(A)\)
is defined so that \(r\) and~\(i\)
are \emph{adjoint}, that is, \(I\subseteq r(J)\)
holds if and only if \(i(I)\subseteq J\).
Adjunction has many useful consequences.
For instance, the partially ordered sets
\(\Ideals^B(A)\defeq r(\Ideals(B))\subseteq \Ideals(A)\)
and \(\Ideals^A(B)\defeq i(\Ideals(A))\subseteq \Ideals(B)\)
of restricted and induced ideals are complete lattices,
and the maps \(i\)
and~\(r\)
restrict to mutually inverse isomorphisms
\(\Ideals^B(A)\cong \Ideals^A(B)\).
Examples suggest that everything that can be
said about the ideal structure of \(A\)
and~\(B\)
in this generality follows from the adjunction (see
Section~\ref{sec:Galois}).

Our main results need some countability assumptions, which follow
if~\(\check{A}\) is second countable.  We assume this for the rest
of the introduction.  In Section~\ref{sec:quasi-orbit}, we
characterise when the quasi-orbit map and the quasi-orbit space
exist and when the quasi-orbit map is open and surjective.  This
involves the following lattice-theoretic conditions:
\begin{enumerate}
\item[\eqref{cond:C1}] joins of restricted ideals remain restricted;
\item[\eqref{cond:C2}] \(I\cap r\circ i(J) = r\circ i(I\cap J)\) for
  all \(I\in\Ideals^B(A)\) and \(J\in\Ideals(A)\);
\item[\eqref{cond:C3}] finite meets of induced ideals are again induced;
\item[\eqref{cond:C3'}] arbitrary meets of induced ideals are again induced;
\item[\eqref{cond:C4}] \(I\cap F(J)= F(I\cap J)\) for every
  \(I\in\Ideals^A(B)\) and \(J\in\Ideals(B)\),
  where \(F(J)\) is the meet of all induced ideals that
  contain~\(J\).
\end{enumerate}

Condition~\eqref{cond:C1} says that the inclusion
\(\Ideals^B(A) \hookrightarrow \Ideals(A)\)
is a frame homomorphism.  Then it induces a continuous map
\(\pi\colon \check{A} \to \Prime^B(A)\)
by Stone duality; here \(\Prime^B(A)\)
is the set of prime elements of \(\Ideals^B(A)\),
equipped with a canonical topology.
We define the \emph{quasi-orbit space} \(\check{A}/{\sim}\)
as the quotient space of~\(\check{A}\)
for the relation~\(\sim\)
defined by \(\prid \sim \prid[q]\)
iff \(\pi(\prid) = \pi(\prid[q])\).
Theorem~\ref{the:open_surjective}
shows that~\eqref{cond:C2} holds if and only if~\(\pi\)
is an open surjection, that is, if and only if
\(\check{A}/{\sim}\cong \Prime^B(A)\)
for an open equivalence relation on~\(\check{A}\).

We say that~\(A\) \emph{separates ideals} in~\(B\) if
\(r\colon \Ideals(B)\to \Ideals(A)\)
is injective or, equivalently, if all ideals in~\(B\)
are induced.  Then \(\Prime^B(A) \cong \Prime(B)\),
which is equal to~\(\check{B}\)
under our countability assumption.  Accordingly, \(r\)
induces a homeomorphism \(\check{B} \cong \check{A}/{\sim}\)
for an open equivalence relation~\(\sim\)
on~\(\check{A}\)
if and only if~\(A\)
separates ideals in~\(B\)
and \eqref{cond:C1} and~\eqref{cond:C2} hold.

Condition~\eqref{cond:C3} says that the
inclusion \(\Ideals^A(B) \hookrightarrow \Ideals(B)\)
is a frame homomorphism.  Then it induces a
continuous map
\(\check{B}\to \Prime^B(A)\)
by Stone duality.  If \(\check{A}/{\sim}\cong \Prime^B(A)\),
this becomes a continuous map
\(\varrho\colon \check{B}\to \check{A}/{\sim}\),
which we call the
\emph{quasi-orbit map}.
By Theorem~\ref{the:open_surjective},
conditions \eqref{cond:C3'} and~\eqref{cond:C4}
characterise when the quasi-orbit map~\(\varrho\)
is open and surjective.  We separate the conditions \eqref{cond:C3}
and~\eqref{cond:C3'} because \eqref{cond:C3} is far easier to check.
And there are crossed products for group actions for
which~\eqref{cond:C3'} fails, whereas~\eqref{cond:C3} holds.

In Section~\ref{sec:symmetric}, we identify an easily checkable
condition for \(\varphi\colon A\to \Mult(B)\) that
implies \eqref{cond:C2} and~\eqref{cond:C3}.  It also
implies~\eqref{cond:C1} for an ordinary inclusion \(A\subseteq B\).
Namely, \(\varphi\colon A\to \Mult(B)\) is called
\emph{symmetric} if every restricted ideal \(I\in\Ideals^B(A)\)
is \emph{symmetric} in the sense that
\(\varphi(I)\cdot B = B\cdot \varphi(I)\).

The remaining sections consider examples and applications.  We begin
with crossed
products for group actions, section \(\Cst\)\nb-algebras of Fell
bundles over groups, and regular
inclusions in Section~\ref{sec:group_crossed}.  These are the
prototypical examples for our theory.  All these cases lead to a
symmetric inclusion where conditions \eqref{cond:C1}, \eqref{cond:C2}
and \eqref{cond:C3} hold.  That is, the quasi-orbit space and the
quasi-orbit map exist automatically.  The conditions \eqref{cond:C3'}
and~\eqref{cond:C4} needed for the quasi-orbit map to be open and
surjective are much more subtle.  In fact, \eqref{cond:C3'} may fail
for full crossed products.
We check \eqref{cond:C3'}
and~\eqref{cond:C4} for reduced crossed
products and reduced section algebras of Fell bundles, assuming an
exactness property.  The key idea, taken
from~\cite{GootmanLazar.DualityCrossedProduct}, is to use
Imai--Takai duality~\cite{Imai-Takai:Duality} and to identify the
induced ideals in the crossed product with the restricted ideals for
the crossed product of the dual coaction.  We generalise this result
from crossed products to Fell bundles, using Morita globalisations,
which are studied in \cites{Abadie-Ferraro:Equivalence_of_Fell_Bundles,
  Abadie-Buss-Ferraro:Morita_Fell}.

In Section~\ref{sec:regular_inc}, we exploit the relationship
between regular \(\Cst\)\nb-inclusions
and gradings on \(\Cst\)\nb-algebras
by inverse semigroups (this is somewhat implicit
in~\cite{Exel:noncomm.cartan}).  This allows us to define a
\emph{groupoid dual to a regular \(\Cst\)\nb-inclusion}
\(A\subseteq B\),
which generalises the Weyl groupoid introduced by Renault
in~\cite{Renault:Cartan.Subalgebras} for a Cartan subalgebra.  The
quasi-orbit space of \(A\subseteq B\)
may be identified with the quasi-orbit space of the dual groupoid.

Section~\ref{sec:applications} treats more examples that illustrate
our theory.
In Section~\ref{sec:induction_locale_morphism}, we treat
\(\Cont_0(X)\)-\(\Cst\)-algebras,
tensor products and skew-commutative tensor products.
In Section~\ref{sec:product_systems}, we consider  relative
Cuntz--Pimsner algebras \(\OO(J,X)\)
for a \(\Cst\)\nb-corres\-pondence~\(X\) over a \(\Cst\)\nb-algebra~\(A\).
Condition~\eqref{cond:C1} usually fails for the standard homomorphism
\(A\to \OO(J,X)\).
Thus there is no quasi-orbit space for this inclusion.  There is,
however, a quasi-orbit space for the gauge action of~\(\T\)
on \(\OO(J,X)\).
We build a continuous open surjection from the primitive ideal space
of~\(\OO(J,X)\)
to the prime ideal space of the lattice of \(J\)\nb-pairs of ideals
in~\(A\).
This is a good substitute for the missing quasi-orbit map.
In Section~\ref{sec:groupoid}, we describe  induced ideals
and  quasi-orbit spaces and quasi-orbit maps for crossed
products of actions by locally compact groupoids.  Here the main
tool is Renault's Disintegration Theorem
\cites{Muhly-Williams:Renaults_equivalence,
  Renault:Representations}.  In Section~\ref{sec:quantum_group_crossed},
we consider reduced crossed products for quantum group coactions.
Here our results are incomplete because we need some technical extra
conditions to check our lattice-theoretic conditions.

We thank Pedro Resende for spotting an error in a previous version
of Lemma~\ref{lem:surjective_implies_spatial}.  We also thank the
referee for valuable remarks concerning locale-theoretical notation.

\section{Adjunction between induction and restriction of ideals}
\label{sec:Galois}

Let \(A\) be a \(\Cst\)-algebra.  Let~\(\Ideals(A)\) be the set of
all (closed, two-sided) ideals in~\(A\), equipped with the partial
order given by inclusion.  This makes~\(\Ideals(A)\) a complete
lattice.  For a family of ideals
\((I_s)_{s \in S}\subseteq \Ideals(A)\), their join
\(\bigvee_{s\in S} I_s\) in~\(\Ideals(A)\) is the closed linear span
\(\overline{\mathrm{span}}\setgiven{a}{a\in I_s,\ s\in S}\), and
their meet \(\bigwedge_{s\in S} I_s\) in~\(\Ideals(A)\) is the
intersection \(\bigcap_{s\in S} I_s\).  The complete lattice
\(\Ideals(A)\) is \emph{infinitely distributive}, that is,
\begin{equation}
  \label{eq:infinitely_distributive}
  J \wedge \bigvee_{s\in S} I_s = \bigvee_{s\in S} J \wedge I_s
\end{equation}
for any set~\(S\) and \(J, I_s\in\Ideals(A)\) for \(s \in S\).  Any
subset~\(K\) of~\(\Ideals(A)\) inherits a partial order.  It can
happen that this makes~\(K\) a lattice in such a way that joins or
meets in \(K\) and~\(\Ideals(A)\) differ.

Let~\(B\)
be a \(\Cst\)\nb-algebra
and \(\Mult(B)\)
its multiplier algebra.
Let \(\varphi\colon A\to \Mult(B)\)
be a \Star{}homomorphism.  If~\(\varphi\)
is injective, we call it a \emph{generalised inclusion}.
The following construction
from~\cite{Baaj-Skandalis:Hopf_KK} restricts ideals
along~\(\varphi\).  For any \(J\in\Ideals(B)\), let
\[
\Mult(B,J) \defeq
\setgiven{m\in \Mult(B)}{m\cdot B + B\cdot m\subseteq J}.
\]

\begin{lemma}
  \label{lem:Mult_BJ}
  Let \(J\in\Ideals(B)\).
  \begin{enumerate}
  \item \label{lem:Mult_BJ1}%
    \(\Mult(B,J)\) is the largest ideal~\(I\)
    in \(\Mult(B)\) with \(I\cap B\subseteq J\).
  \item \label{lem:Mult_BJ2}%
    A multiplier \(m\in\Mult(B)\) belongs to~\(\Mult(B,J)\) if and
    only if \(m\cdot B\subseteq J\), if and only if \(B\cdot
    m\subseteq J\), if and only if \(B\cdot m\cdot B\subseteq J\).

  \item \label{lem:Mult_BJ3}%
    \(\Mult(B,J)\) is the kernel of the canonical \Star{}homomorphism
    \(\Mult(B)\to \Mult(B/J)\).
  \end{enumerate}
\end{lemma}

\begin{proof}
  The subset \(\Mult(B,J)\)
  is a closed, two-sided ideal in~\(\Mult(B)\)
  because~\(J\)
  is a closed, two-sided ideal in~\(B\)
  and \(m\cdot B+B\cdot m\subseteq B\)
  for all \(m\in\Mult(B)\).
  We have \(\Mult(B,J)\cap B = \Mult(B,J)\cdot B \subseteq J\).
  Let \(I\in\Ideals(\Mult(B))\)
  satisfy \(I\cap B\subseteq J\).
  Then \(m\cdot B+B\cdot m \subseteq I\cap B\subseteq J\)
  for all \(m\in I\),
  that is, \(I\subseteq \Mult(B,J)\).
  Thus~\(\Mult(B,J)\)
  is the largest ideal in~\(\Mult(B)\) that intersects~\(B\) in~\(J\).

  Let \(m\in\Mult(B)\).
  Let~\(N\)
  be a directed set and let \((e_n)_{n\in N}\)
  be an approximate unit for~\(B\).
  If \(b\in B\),
  then \(b\cdot m = \lim b\cdot m\cdot e_n\)
  because \(B\cdot m\subseteq B\).
  Thus \(B\cdot m\cdot B \subseteq J\)
  implies \(B\cdot m\subseteq J\).
  Similarly, it implies \(m\cdot B\subseteq J\)
  and then \(m\in\Mult(B,J)\).
  Conversely, \(m\in\Mult(B,J)\)
  implies both \(m\cdot B\subseteq J\)
  and \(B\cdot m\subseteq J\),
  and these imply \(B\cdot m\cdot B\subseteq J\)
  because~\(J\) is a two-sided ideal in~\(B\).
  Conversely, \(B\cdot m\cdot B\subseteq J\)
  implies \(m\cdot B\subseteq J\)
  because \(m\cdot B\subseteq B\)
  and \(\Mult(B,J) \cap B = J\).

  If \(m\in\Mult(B)\),
  then \(m\cdot J+J\cdot m\subseteq J\),
  so that~\(m\)
  descends to a multiplier of~\(B/J\).
  This is the canonical \Star{}homomorphism \(\Mult(B)\to\Mult(B/J)\)
  in~\ref{lem:Mult_BJ3}.  Its kernel consists of those
  \(m\in\Mult(B)\)
  with \(m\cdot B\subseteq J\).
  This is \(\Mult(B,J)\) by~\ref{lem:Mult_BJ2}.
\end{proof}

\begin{definition}
  \label{def:restrict}
  The \emph{restriction} of \(J\in\Ideals(B)\) is
  \[
  r(J) \defeq \varphi^{-1}(\Mult(B,J)).
  \]
  The \emph{induced ideal} of \(I\in\Ideals(A)\) is
  \[
  i(I)\defeq B \varphi(I)  B
  = \overline{\operatorname{span}}
  \setgiven{b_1\varphi(x) b_2}{b_1,b_2\in B,\ x\in I}.
  \]
\end{definition}

\begin{remark}
  The maps \(r\)
  and~\(i\)
  also appear in \cite{Green:Local_twisted}*{Proposition~9.(i)}.  They
  coincide with the maps denoted by \(\operatorname{Res}_\varphi\)
  and \(\operatorname{Ex}_\varphi\)
  in \cite{Nilsen:Bundles}*{Lemma 1.1}.
\end{remark}

By construction,
 \(i(I)\in\Ideals(B)\)
for all \(I\in\Ideals(A)\).
Let \(J\in\Ideals(B)\).
Lemma~\ref{lem:Mult_BJ}.\ref{lem:Mult_BJ3} implies that~\(r(J)\)
is the kernel of the composite \Star{}homomorphism
\(A\to\Mult(B)\to\Mult(B/J)\).
So \(r(J)\in\Ideals(A)\),
and~\(\varphi\) induces an injective \Star{}homomorphism
\begin{equation}
  \label{eq:induced_hom_on_quotient}
  \varphi_J\colon A/r(J) \hookrightarrow \Mult(B/J).
\end{equation}
In particular, \(r(0)=\ker \varphi\),
which is~\(0\)
if and only if~\(\varphi\) is injective.

\begin{definition}
  \label{def:restricted_induced_ideals}
  Let \(\varphi\colon A\to \Mult(B)\)
  be a \Star{}homomorphism.  We call \(I\in\Ideals(A)\)
  \emph{restricted} if \(I=r(J)\)
  for some \(J\in\Ideals(B)\).
  We call \(J\in\Ideals(B)\)
  \emph{induced} if \(J=i(I)\)
  for some \(I\in\Ideals(A)\).
  Let \(\Ideals^B(A)\subseteq \Ideals(A)\)
  and \(\Ideals^A(B)\subseteq \Ideals(B)\)
  be the subsets of restricted and induced ideals.
\end{definition}

\begin{remark}
  \label{rem:pass_to_injective}
  The \Star{}homomorphism
  \(\varphi_{r(0)}\colon A/r(0)\to \Mult(B)\) as
  in~\eqref{eq:induced_hom_on_quotient}
  is injective and there are natural order isomorphisms
  \(\Ideals^B(A)\cong \Ideals^B(A/r(0))\) and \(\Ideals^A(B)\cong
  \Ideals^{A/r(0)}(B)\).
  Thus in all statements concerning only induced and restricted
  ideals one may assume that \(\varphi\colon A\to \Mult(B)\) is
  injective, that is, a generalised \(\Cst\)\nb-inclusion.
\end{remark}

\begin{lemma}
  \label{lem:Galois_connection}
  Let \(I\in\Ideals(A)\),
  \(J\in\Ideals(B)\).
  Then \(I\subseteq r(J)\) if and only if \(i(I)\subseteq J\).
\end{lemma}

\begin{proof}
  Since \(I\subseteq A\),
  the inclusion \(I\subseteq r(J)\)
  holds if and only if \(\varphi(I)\subseteq \Mult(B,J)\).
  By Lemma~\ref{lem:Mult_BJ}, this is equivalent to
  \(b_1\varphi(x) b_2\in J\)
  for all \(x\in I\),
  \(b_1,b_2\in B\).  And this is equivalent to \(i(I)\subseteq J\).
\end{proof}

The relationship between the two maps
\[
r\colon \Ideals(B)\to\Ideals(A),\qquad
i\colon \Ideals(A)\to\Ideals(B)
\]
in Lemma~\ref{lem:Galois_connection} says that they are
(contravariantly) \emph{adjoint}
(see~\cite{Johnstone:Stone_spaces}).  In other notation, they form a
(monotone) \emph{Galois connection} (see
\cite{Davey-Priestley:Lattices_order}*{Definition~7.23}).
This was already noticed by Green in
\cite{Green:Local_twisted}*{Proposition~9.(i)}.  It has several useful
consequences, which we will list in
Proposition~\ref{pro:restrict_induce_ideals}.
In particular,
the adjoint maps \(i\)
and~\(r\)
 determine each other; \(r\)
is the \emph{right adjoint} of~\(i\)
and~\(i\)
is the \emph{left adjoint} of~\(r\).
More precisely, the adjunction dictates that
\(r(J)\in\Ideals(A)\)
is the join of all \(I\in\Ideals(A)\)
with \(i(I)\subseteq J\),
whereas \(i(I)\in\Ideals(B)\)
is the meet of all \(J\in\Ideals(B)\)
with \(I\subseteq r(J)\).

\begin{proposition}
  \label{pro:restrict_induce_ideals}
  Let \(\varphi\colon A\to \Mult(B)\) be a \Star{}homomorphism.  Then
  \begin{enumerate}
  \item \label{pro:restrict_induce_ideals1}%
    the maps \(r\colon \Ideals(B) \to \Ideals(A)\)
    and \(i\colon \Ideals(A) \to \Ideals(B)\) are monotone;
  \item \label{pro:restrict_induce_ideals2}%
    if \(I\in\Ideals(A)\),
    then \(r\circ i(I)\supseteq I\)
    and \(i\circ r\circ i(I)=i(I)\);
  \item \label{pro:restrict_induce_ideals2b}%
    if \(J\in\Ideals(B)\),
    then \(i\circ r(J)\subseteq J\) and \(r\circ i\circ r(J)=r(J)\);
  \item \label{pro:restrict_induce_ideals3}%
    the maps \(i\)
    and~\(r\)
    restrict to isomorphisms of partially ordered sets
    between \(\Ideals^A(B)\) and \(\Ideals^B(A)\)
    that are inverse to each other;
  \item \label{pro:restrict_induce_ideals4}%
    the map~\(i\) preserves joins and~\(r\) preserves meets;
  \item \label{pro:restrict_induce_ideals5}%
    \(i(0)=0\) and \(r(B)=A\);
  \item \label{pro:restrict_induce_ideals6}%
    meets of restricted ideals in~\(A\)
    remain restricted and joins of induced ideals in~\(B\)
    remain induced;

  \item \label{pro:restrict_induce_ideals6.5}%
    the isomorphic partially ordered sets
    \(\Ideals^A(B)\cong \Ideals^B(A)\) are complete lattices;
  \item \label{pro:restrict_induce_ideals7}%
    the inclusion \(\Ideals^B(A) \hookrightarrow \Ideals(A)\) is right
    adjoint to
    \(r\circ i\colon \Ideals(A) \to \Ideals^B(A)\), that is,
    \(I\in\Ideals(A)\) and \(J\in\Ideals^B(A)\) satisfy
    \(I\subseteq J\) if and only if \(r \circ i(I)\subseteq J\);
    equivalently, \(r\circ i(I)\) is the smallest restricted ideal
    that contains~\(I\);
  \item \label{pro:restrict_induce_ideals8}%
    \(i\circ r\colon \Ideals(B) \to \Ideals^A(B)\) is right adjoint to
    the inclusion \(\Ideals^A(B) \hookrightarrow \Ideals(B)\), that
    is, \(I\in\Ideals^A(B)\) and \(J\in\Ideals(B)\) satisfy
    \(I\subseteq J\) if and only if \(I\subseteq i \circ r(J)\);
    equivalently, \(i\circ r(J)\) for \(J\in\Ideals(B)\) is the
    largest induced ideal contained in~\(J\).
  \end{enumerate}
\end{proposition}

\begin{proof}
  Statements
  \ref{pro:restrict_induce_ideals1}--\ref{pro:restrict_induce_ideals2b}
  are \cite{Davey-Priestley:Lattices_order}*{Lemma 7.26}
  and~\ref{pro:restrict_induce_ideals4} is
  \cite{Davey-Priestley:Lattices_order}*{Proposition 7.31}.
  The rest now follows easily (and is well known).
\end{proof}

\begin{remark}
  \label{rem:joins_restricted}
  The subset \(\Ideals^B(A) \subseteq\Ideals(A)\) need not be closed
  under joins and \(\Ideals^A(B) \subseteq\Ideals(B)\) need not be
  closed under meets.  Nevertheless, the map
  \(r\circ i\colon \Ideals(A)\to\Ideals^B(A)\) preserves joins and
  \(i\circ r\colon \Ideals(B)\to\Ideals^A(B)\) preserves meets by
  \cite{Davey-Priestley:Lattices_order}*{Proposition 7.31}.  This is
  no contradiction because here joins and meets are taken in the
  respective sublattices.  By definition, the join of
  \((I_\alpha)_{\alpha\in S}\) in \(\Ideals^B(A)\) is the smallest
  element of \(\Ideals^B(A)\) that contains~\(I_\alpha\) for all
  \(\alpha\in S\).  This is equal to
  \(r\circ i\bigl(\bigvee I_\alpha\bigr)\)
  by~\ref{pro:restrict_induce_ideals7} in
  Proposition~\ref{pro:restrict_induce_ideals}.
\end{remark}

\begin{definition}
  Let \(\varphi\colon A\to \Mult(B)\) be a \Star{}homomorphism.
  We say that~\(A\)
  \emph{detects} ideals in~\(B\)
  if \(r(J)=r(0)\)
  for \(J\in\Ideals(B)\)
  implies \(J=0\),
  and \(A\) \emph{separates} ideals in~\(B\)
  if \(r(J_1)=r(J_2)\)
  for \(J_1,J_2\in\Ideals(B)\) implies \(J_1=J_2\).
\end{definition}

\begin{proposition}
  \label{pro:separate_induced}
  Let \(\varphi\colon A\to\Mult(B)\)
  be a \Star{}homomorphism.  Then~\(A\)
  detects ideals in~\(B\)
  if and only if any non-zero ideal in~\(B\)
  contains a non-zero induced ideal.  And~\(A\)
  separates ideals in~\(B\)
  if and only if all ideals in~\(B\)
  are induced, if and only if \(A/r(J)\)
  embedded into \(\Mult(B/J)\)
  using the induced injective \Star{}homomorphism
  \(A/r(J) \hookrightarrow \Mult(B/J)\)
  in~\eqref{eq:induced_hom_on_quotient} detects ideals for all
  induced ideals \(J\in\Ideals^A(B)\).
\end{proposition}

\begin{proof}
  By definition, \(A\)
  detects ideals in~\(B\)
  if \(J\neq0\)
  implies \(r(J)\neq r(0)\).
  The statements \ref{pro:restrict_induce_ideals2}
  and~\ref{pro:restrict_induce_ideals2b} in
  Proposition~\ref{pro:restrict_induce_ideals} say that
  \(i\circ r(J)\subseteq J\)
  and \(r\circ i\circ r = r\).
  Assume first that~\(A\)
  detects ideals in~\(B\).
  For each \(J\neq0\),
  the induced ideal \(i\circ r(J)\)
  is contained in~\(J\),
  and it is non-zero because \(r(i \circ r(J)) = r(J) \neq r(0)\).
  Conversely, assume that any \(J\in\Ideals(B)\)
  contains a non-zero induced ideal \(i(I)\).
  Then \(r(J) \supseteq r\circ i(I)\),
  and \(r\circ i(I)\neq r(0)\)
  because \(i(r\circ i(I)) = i(I)\)
  differs from \(i\circ r(0)=0\).
  This proves the first statement about detection of ideals.

  By definition, \(A\)
  separates ideals in~\(B\)
  if and only if~\(r\)
  is injective.  If~\(r\)
  is not injective, then neither is \(i\circ r\).
  Conversely, if \(i\circ r\)
  is not injective, then neither is \(r\circ i\circ r\),
  which is equal to~\(r\)
  by
  Proposition~\ref{pro:restrict_induce_ideals}.\ref{pro:restrict_induce_ideals2b}.
  Thus~\(A\)
  separates ideals in~\(B\)
  if and only \(i\circ r\)
  is injective.  Since \(i\circ r\)
  is a retraction from \(\Ideals(B)\)
  onto \(\Ideals^A(B)\),
  it is injective if and only if \(\Ideals^A(B)=\Ideals(B)\).

  As a result, there is \(J\in\Ideals(B)\)
  with \(i\circ r(J)\neq J\)
  if (and only if)~\(A\)
  does not separate ideals in~\(B\).
  Proposition~\ref{pro:restrict_induce_ideals}.\ref{pro:restrict_induce_ideals2b}
  implies
  \(i\circ r(J)\subseteq J\)
  and \(r(J) = r(i\circ r(J))\).
  So~\(J\)
  is mapped to a non-zero ideal in \(B/i\circ r(J)\),
  which is not detected by the map
  \(A/r(J) \to \Mult(B/r\circ i(I))\).
  Conversely, if~\(A/r(J)\)
  does not detect ideals in~\(B/J\)
  for some induced ideal \(J\in\Ideals^A(B)\),
  then there is an ideal \(J'\in\Ideals(B)\)
  with \(J\subsetneq J'\)
  and \(r(J') = r(J)\).  So~\(A\) does not separate ideals in~\(B\).
\end{proof}

The following examples show that, in general, restriction
and induction of ideals do not have more properties than those in
Proposition~\ref{pro:restrict_induce_ideals}.  Examples
\ref{exa:Z2_crossed_product_i_r_1}
and~\ref{exa:Z2_crossed_product_i_r_2} are very easy cases of a
crossed product by the finite group~\(\Z/2\).
Example~\ref{exa:join_not_restricted}, which shows
a join of two restricted ideals that is no longer restricted,
is related to graph
\(\Cst\)\nb-algebras (see also Example~\ref{exa:Katsura_graph}).
Example~\ref{ex:skew-commutative} will show an example of a meet of
two induced ideals that is no longer induced.

\begin{example}
  \label{exa:Z2_crossed_product_i_r_1}
  Let \(A=\C\)
  be embedded unitally into \(B=\C\oplus \C\).
  Here \(r(J) = A\cap J\),
  and \(r(\C\oplus0) = 0 = r(0\oplus\C)\),
  but \(r(\C\oplus0 + 0\oplus\C) = r(\C\oplus\C) = \C\).
  So~\(r\) does not commute with finite joins.
\end{example}

\begin{example}
  \label{exa:Z2_crossed_product_i_r_2}
  Let \(A=\C\oplus\C\) be
  embedded diagonally into \(B=\Mat_2(\C)\).
  Here \(i(\C\oplus0) = B = i(0\oplus\C)\),
  but \(i(\C\oplus0 \cap 0\oplus\C) = i(0\oplus0) = 0\).
  So~\(i\) does not commute with finite meets.
\end{example}

\begin{example}
  \label{exa:preserve_min_max_i_r}
  Since \(r(0)=\ker \varphi\),
  the map~\(r\)
  preserves the minimal elements if and only if~\(\varphi\)
  is injective.  Similarly, \(i\)
  preserves the maximal elements if and only if \(i(A) = B\),
  that is, \(B\cdot \varphi(A)\cdot B = B\).
  This happens, for instance, if~\(\varphi\)
  is non-degenerate.
\end{example}

\begin{example}
  \label{exa:join_not_restricted}
  Let \(B\defeq \Mat_2(\C)\oplus\Mat_2(\C)\)
  and let \(A\subseteq B\)
  be the commutative \(\Cst\)\nb-subalgebra
  spanned by the orthogonal diagonal projections \((E_{00},0)\),
  \((0,E_{00})\),
  and \((E_{11},E_{11})\).
  Let \(J_1=\Mat_2(\C)\oplus0\)
  and \(J_2=0\oplus\Mat_2(\C)\).
  Then \(I_1\defeq J_1\cap A = \C\cdot (E_{00},0)\)
  and \(I_2\defeq J_2\cap A = \C\cdot (0,E_{00})\)
  are restricted ideals in \(A\cong\C^3\).
  The only other restricted ideals are \(\{0\}\)
  and~\(A\).
  So the join \(I_1+I_2\neq A\)
  of the two restricted ideals \(I_1\) and~\(I_2\) is not restricted.
\end{example}

\section{Stone duality and open surjective maps}
\label{sec:Stone}

Stone duality refers to several dualities between certain categories
of topological spaces and partially ordered sets.  We briefly recall
the duality between sober topological spaces and spatial locales
(see \cites{Johnstone:Stone_spaces, Vickers:Topology_logic}).  This
is the key to turn information about ideal lattices into information
about prime and primitive ideal spaces.

A \emph{frame} is a complete, infinitely distributive lattice
(see~\eqref{eq:infinitely_distributive}), and a \emph{frame
  homomorphism} is a map that preserves arbitrary joins and finite
meets.  The category of \emph{locales} is the opposite of the category
of frames.  So a locale is the same as a frame and a \emph{locale map}
\(g\colon L\to R\) is a frame homomorphism \(g^*\colon R\to L\).
If~\(X\) is a topological space, then the partially ordered set of
open subsets~\(\Open(X)\) is a locale.  And if \(f\colon X\to Y\) is a
continuous map, then \(f^*=f^{-1}\colon \Open(Y)\to\Open(X)\) is a
frame homomorphism.  Equivalently, it is\ a locale map
\(f_*\colon \Open(X)\to\Open(Y)\).  Thus~\(\Open\) is a covariant
functor from topological spaces to locales.  It is well known
that~\(\Open\) has an adjoint functor~\(\pt\).  That is, \(\pt\) maps
a locale~\(L\) to a topological space~\(\pt(L)\) such that continuous
maps \(X\to \pt(L)\) for a topological space~\(X\) are in natural
bijection with locale maps \(\Open(X)\to L\).

Points in~\(\pt(L)\)
are \emph{characters}, that is, unital frame homomorphisms
\(L\to \{0,1\}\);
these are equivalent to \emph{prime elements}, that is, elements
\(p\in L\setminus\{1\}\) with the property
that \(I_1\wedge I_2\le p\) for \(I_1,I_2\in L\)
implies  \(I_1\le p\)
or \(I_2\le p\).
Indeed, for a character \(\chi\colon L\to \{0,1\}\),
there is a largest element \(p_\chi\in L\)
with \(\chi(p_\chi)=0\).
Thus \(\chi(I) = 0\)
for \(I\le p_\chi\)
and \(\chi(I) = 1\)
otherwise by maximality of~\(p_\chi\).
So~\(p_\chi\)
is prime.
Conversely, if \(p\in L\)
is prime, then \(\chi_p(I) \defeq 0\)
for \(I\le p\)
and \(\chi_p(I) \defeq 1\)
otherwise defines a character~\(\chi_p\).
Thus characters are equivalent to prime elements. The topology on~\(\pt(L)\) is defined as follows.  For \(I\in L\),
let
\[
U_I\defeq \setgiven{\chi\colon L\to\{0,1\} \text{ character}}{\chi(I)=1}
\cong \setgiven{p\in L \text{ prime}}{I\not\le p}.
\]
The map \(I\mapsto U_I\)
preserves joins and finite meets by the definition of a character.
Hence \(\setgiven{U_I}{I\in L}\)
is a topology on~\(\pt(L)\).
Any locale map \(g\colon L\to R\)
between locales \(L\)
and~\(R\)
induces a continuous map \(g_*\colon \pt(L)\to \pt(R)\)
by \(g_*(\chi)\defeq \chi \circ g^*\).
Equivalently, if \(p\in L\)
is prime, then~\(g_*(p)\)
is the unique element of~\(R\)
with
\begin{equation}
  \label{eq:induced_in_prime}
  I\le g_*(p) \text{ if and only if } g^*(I)\le p.
\end{equation}
In other words, \(g_*\) is right adjoint to~\(g^*\).
Equation~\eqref{eq:induced_in_prime} says that~\(g_*(p)\)
is the largest element in~\(R\)
with \(g^*(g_*(p)) \le p\).

The functor~\(\pt\) is adjoint to~\(\Open\).  The counit of the
adjunction is the family of natural locale maps defined by the natural
frame homomorphisms \(U^L\colon L \to \Open(\pt(L))\),
\(I\mapsto U_I\), and the counit is the family of continuous maps
\(\delta^X\colon X\to \pt(\Open(X))\), \(x\mapsto \delta_x\),
where~\(\delta_x\) for \(x\in X\) is the character
\[
\delta_x\colon \Open(X) \to \{0,1\},\qquad
U\mapsto
\begin{cases}
  1&\text{if }x\in U,\\
  0&\text{if }x\notin U.
\end{cases}
\]
The character~\(\delta_x\)
corresponds to the prime element \(X\setminus\overline{\{x\}}\)
of~\(\Open(X)\).
The induced continuous map \((U^L)_*\colon \pt(\Open(\pt(L))) \to \pt(L)\)
is a homeomorphism with inverse~\(\delta^{\pt(L)}\),
and the induced frame homomorphism
\((\delta^X)^{-1}\colon \Open(\pt(\Open(X))) \to \Open(X)\)
is an isomorphism with inverse~\(U^{\Open(X)}\).
A locale map \(g\colon\Open(X) \to L\)
and its adjunct \(\pi\colon X\to \pt(L)\)
are related by \(g^*(I) = \pi^{-1}(U_I)\)
for all \(I\in L\)
and \(\pi(x) = \delta^X_x\circ g^*\) for all \(x\in X\).

The map~\(U^L\)
is always surjective.  A locale is called \emph{spatial} if~\(U^L\)
is injective.  Equivalently, the characters
separate elements of~\(L\).
Let~\(X\)
be a topological space.  The locale~\(\Open(X)\)
is spatial.  The space~\(X\)
is~\(T_0\)
if and only if~\(\delta^X\) is injective.
Since \((\delta^X)^{-1}\colon \Open(\pt(\Open(X)) \to \Open(X)\)
is a lattice isomorphism, \(\delta^X\)
is a homeomorphism once it is bijective.  The space~\(X\)
is called \emph{sober} in this case.
This holds if and only if every irreducible closed
set in~\(X\)
is the closure of a singleton.  Spaces of points of locales are always sober.

Our next goal is to characterise when the continuous map
\(\pi\colon X \to \pt(R)\) that is the adjunct of a locale map
\(g\colon L \to \Open(X)\) is surjective and open.  We will show
that the necessary conditions identified
in~\cite{Joyal-Tierney:Galois_Grothendieck} are sufficient under
some countability hypotheses.  To motivate these conditions, let
\(f\colon X\to Y\) be a continuous map and let
\(f^*=f^{-1}\colon \Open(Y)\to\Open(X)\) be the associated frame
homomorphism.  If \(f\colon X\to Y\) is open, then
\(f_!\colon \Open(X)\to\Open(Y)\), \(U\mapsto f(U)\), is a left
adjoint of~\(f^*\), that is,
\(f(U) \subseteq V\)
is equivalent to \(U\subseteq f^{-1}(V)\) for all
\(U\in \Open(X)\), \(V\in \Open(Y)\).  In addition,
\begin{equation}
  \label{eq:Frobenius_concrete}
  f(U \cap f^{-1}(V)) = f(U) \cap V.  
\end{equation}
If~\(f\) is open, then each of the
following properties are equivalent to~\(f\) being surjective:
\(f_!(X)=Y\); \(f_!\circ f^{*}(V)=V\) for all \(V\in \Open(Y)\);
\(f^{*}\) being injective.

\begin{definition}[\cite{Joyal-Tierney:Galois_Grothendieck}]
  A locale map \(g\colon L \to R\) is called \emph{open} if
  \(g^*\colon R \to L\) has a left adjoint \(g_!\colon L \to R\) such
  that the \emph{Frobenius reciprocity condition} holds:
  \begin{equation}
    \label{eq:open_surjective}%
    g_!(J \wedge g^*(I)) = g_!(J) \wedge I
    \qquad\text{for all } I\in R,\ J\in L.
  \end{equation}
  An open locale map~\(g\) is called \emph{surjective} if~\(g^*\)
  is an injective map.
\end{definition}

\begin{remark}
  \label{rem:stupid_remark}
  The discussion above shows that a continuous, open, surjective map
  \(f\colon X\to Y\) induces an open and surjective locale map
  \(\Open(X)\to\Open(Y)\).  The Frobenius reciprocity condition
  is~\eqref{eq:Frobenius_concrete}.
\end{remark}

\begin{remark}
  \label{rem:Joyal_Tierney}
  Let \(g\colon L\to R\) be an open locale map.  Then~\(g\) is
  surjective if and only if \(g_!(1)=1\), if and only if
  \(g_!\circ g^* = \Id_R\). Indeed, the adjunction between \(g_!\)
  and~\(g^*\) implies \(1 \le g^*\circ g_!(1)\).  Then
  \(1 = g^*\circ g_!(1)\), and \(g^*(1)=1\) because~\(g^*\) is
  monotone.  Therefore, if~\(g^*\) is injective, then \(g_!(1) = 1\).
  Equation~\eqref{eq:open_surjective} with \(J=1\) shows that
  \(g_!(1)=1\) is equivalent to \(g_!\circ g^* = \Id_R\), and the
  latter implies that~\(g^*\) is injective.  Thus our definition of an
  open surjection agrees with that
  in~\cite{Joyal-Tierney:Galois_Grothendieck}.
\end{remark}

Let~\(A\)
be a \(\Cst\)\nb-algebra.
Let~\(\check{A}\)
be the primitive ideal space of~\(A\).
Then \(\Ideals(A) \cong \Open(\check{A})\)
(see \cite{Pedersen:Cstar_automorphisms}*{Theorem~4.1.3}).
Hence~\(\Ideals(A)\)
is a spatial locale.  By definition, \(\pt(\Ideals(A))\)
is the set \(\Prime(A)\)
of \emph{prime ideals} in~\(A\)
with the hull-kernel topology.  The space~\(\check{A}\)
is \(T_0\)
(see \cite{Dixmier:Cstar-algebras}*{3.3.8}).  Hence
\(\check{A}\subseteq \Prime(A)\).
The inclusion \(\check{A} \subseteq \Prime(A)\) is a
continuous map whose associated locale map is the identity on
\(\Ideals(A)\cong \Open(\check{A})\cong \Open(\Prime(A))\).
Thus \(\check{A} = \Prime(A)\) if and only if~\(\check{A}\)
is sober.  A \(\Cst\)\nb-algebra
where this fails is built
in~\cite{Weaver:Prime_not_primitive}.  In this example, the
inclusion map \(\check{A} \hookrightarrow \Prime(A)\) is not
surjective although the induced map
\(\Open(\Prime(A)) \to \Open(\check{A})\)
is even a frame isomorphism.
If~\(A\)
is separable, however, then \(\check{A}=\Prime(A)\).
This was first proved by Dixmier (see
\cite{Dixmier:Sur_Cstar-algebres}*{Corollaire~1} and also
\cite{Pedersen:Cstar_automorphisms}*{Proposition~4.3.6}).
In fact, since~\(\check{A}\)
is always locally quasi-compact, \(T_0\)
and Baire (see \cite{Dixmier:Cstar-algebras}*{3.1.3, 3.4.13}), a
result of Hofmann
\cite{Hofmann:Baire_lattices}*{Proposition~1} implies
that~\(\check{A}\)
is sober if~\(\check{A}\)
is second countable or, more generally, if \(\Prime(A)\)
is first countable.  We prove the following far-reaching
generalisation of these results:

\begin{theorem}
  \label{the:open_surjective}
  Let~\(X\) be a topological space and~\(L\) a locale.  Let
  \(g\colon \Open(X) \to L\) be a locale map and
  \(\pi\colon X\to \pt(L)\) a continuous map that are adjuncts of each
  other.  If~\(\pi\) is an open surjection and~\(L\) is spatial,
  then~\(g\) is open and surjective.  Conversely, assume~\(g\)
  to be open and surjective and
  \begin{enumerate}
  \item \label{the:open_surjective_1}%
    \(\pt(L)\) to be first countable --~which holds when~\(X\) is second
    countable~-- and
  \item \label{the:open_surjective_2}%
    all closed subspaces of~\(X\) to be Baire spaces.
  \end{enumerate}
  Then~\(\pi\) is open and surjective and~\(L\) is spatial.
\end{theorem}

\begin{corollary}
  Let \(X\) and~\(Y\) be topological spaces and let
  \(\pi\colon X\to Y\) be a continuous map.  Assume~\(Y\) to be sober
  and the locale map \(\Open(X)\to \Open(Y)\) defined by~\(\pi\) to be
  open and surjective.  Also assume~\(Y\) to be first countable
  or~\(X\) to be second countable, and assume all closed subspaces
  of~\(X\) to be Baire spaces.  Then~\(\pi\) is an open surjection.
\end{corollary}

\begin{proof}
  Apply Theorem~\ref{the:open_surjective} with \(L\defeq \Open(Y)\);
  here \(Y\cong \pt(L)\) because~\(Y\) is sober.
\end{proof}

\begin{corollary}
  \label{cor:first_countable_implies_sober}
  Let~\(A\)
  be a \(\Cst\)\nb-algebra.
  If~\(\check{A}\)
  is second countable or, more generally, if \(\Prime(A)\)
  is first countable, then \(\Prime(A)=\check{A}\).
\end{corollary}

\begin{proof}
  Let
  \(X\defeq \check{A}\), \(L\defeq \Ideals(A)\), and \(g\defeq
  \Id_{\Ideals(A)}\).
  A closed subset of~\(X\)
  is the primitive ideal space of a quotient~\(A/I\)
  for some ideal~\(I\).
  Hence it is a Baire space.  So Theorem~\ref{the:open_surjective}
  implies the statement.
\end{proof}

The proof of  Theorem~\ref{the:open_surjective} will occupy the rest of this section.
We start with some preliminary results.

\begin{lemma}
  \label{lem:surjective_implies_spatial}
  Let \(g\colon\Open(X) \to L\) be a locale map and let
  \(\pi\colon X\to \pt(L)\) be a surjective, continuous map, which are
  adjuncts of each other.  Then \(g^*\colon L \to \Open(X)\) is
  injective if and only if~\(L\) is spatial.
\end{lemma}

\begin{proof}
  Let \(I,I'\in L\).  Then \(g^*(I)=g^*(I')\) is equivalent to
  \(U_I=U_{I'}\) because \(\pi\colon X\to \pt(L)\) is surjective and
  \(g^*(I) = \pi^{-1}(U_I)\).  Hence~\(g^*\) is
  injective if and only if \(L\to \Open(\pt(L))\), \(I\mapsto U_I\), is
  injective.  This is the definition of~\(L\) being spatial.
\end{proof}

\begin{lemma}
  \label{lem:open_surjective_properties}
  Let \(g\colon \Open(X)  \to L \)
  be an open and surjective locale map.
  \begin{enumerate}
  \item \label{enu:second_PL}%
    If~\(X\) is second countable, then so is~\(\pt(L)\).
  \item \label{enu:dense_equivalence}%
    For any \(I\in L\),
    \(g^*(I)\subseteq X\)
    is dense for the topology~\(g^*(L)\)
    if and only if it is dense for the topology~\(\Open(X)\).
  \item\label{enu:Baire_GL}
    If~\(X\) with the topology~\(\Open(X)\) is a Baire space,
    then~\(X\) with the topology~\(g^*(L)\) is Baire.
  \end{enumerate}
\end{lemma}

\begin{proof}
  \ref{enu:second_PL}.  Let \((V_n)_{n\in\N}\)
  be a countable basis for the topology on~\(X\).
  Let \(I\in L\).
  There is \(S\subseteq\N\)
  with \(g^*(I) = \bigvee_{k\in S} V_k\).
  Then
  \(g^*(I) = g^*\circ g_! \circ g^*(I) = g^*\bigl(\bigvee_{k\in S}
  g_!(V_k)\bigr)\)
  because~\(g_!\)
  preserves joins as it is a left adjoint.  Since~\(g^*\)
  is injective, this implies \(I = \bigvee_{k\in S} g_!(V_k)\).
  Thus the open subsets \(U_{g_!(V_k)}\)
  form a countable basis for~\(\pt(L)\).

  \ref{enu:dense_equivalence}.  The subset \(g^*(I)\)
  is dense for the topology~\(\Open(X)\)
  if and only if \(g^*(I)\cap V\neq\emptyset\)
  for all \(V\in\Open(X)\)
  with \(V\neq\emptyset\).
  This is clearly stronger than being dense for the
  topology~\(g^*(L)\).
  Now assume that~\(g^*(I)\)
  is not dense for~\(\Open(X)\).
  We claim that~\(g^*(I)\)
  is not dense for~\(g^*(L)\)
  either.  By assumption, there is \(V\in \Open(X)\)
  with \(V\neq \emptyset\)
  and \(g^*(I)\cap V=\emptyset\).
  Then~\eqref{eq:open_surjective} implies
  \(I\wedge g_!(V) = g_!(g^*(I) \cap V)=g_!(\emptyset)\).
  Since~\(g_!\)
  is a left adjoint, it commutes with arbitrary joins.  So
  does~\(g^*\)
  as a frame homomorphism.  Therefore, \(g^*(g_!(\emptyset)) = \emptyset\).
  Since~\(g^*\) preserves finite meets,
  \[
  g^*(I)\cap g^*\circ g_!(V)
  = g^*(I\wedge g_!(V))
  = g^*(g_!(\emptyset))
  =  \emptyset.
  \]
  Since \(g^*\circ g_!(V)\supseteq V\neq \emptyset\),
  this shows that~\(g^*(I)\) is not dense for the topology~\(g^*(L)\).

  \ref{enu:Baire_GL}.  Let \(I_n \in L\)
  for \(n\in\N\)
  be such that~\(g^*(I_n)\)
  is dense for the topology~\(g^*(L)\)
  on~\(X\).
  By the previous lemma, \(g^*(I_n)\)
  is dense for~\(\Open(X)\).
  Since~\(X\)
  is a Baire space, the
  intersection \(\bigcap g^*(I_n)\)
  is dense for~\(\Open(X)\).
  Then it is also dense for~\(g^*(L)\).
\end{proof}

The following proposition generalises Dixmier's
\cite{Dixmier:Sur_Cstar-algebres}*{Th\'eor\`eme~2}, which states
that every separable prime \(\Cst\)\nb-algebra is primitive.  It is
a key step in the proof of Theorem~\ref{the:open_surjective}.

\begin{proposition}
  \label{prop:Dixmier's}
  Let~\(X\)   be a topological space and~\(L\)
  a locale.  Let \(g\colon \Open(X)  \to L\)
  be a locale map and \(\pi\colon X\to \pt(L)\)
  a continuous map that are adjuncts of each other.
  Suppose~\(X\)
  is a Baire space, \(\pt(L)\)
  is first countable and the minimal element~\(0\)
  of~\(L\)
  is prime.  If \(g\)
  is open surjective, then
  \(0\in \pt(L)\) belongs to the image \(\pi(X)\).
\end{proposition}

\begin{proof}
  The singleton~\(\{0\}\)
  is dense in~\(\pt(L)\).
  Indeed, \(0\in U_I\)
  for all \(I\in L\setminus\{0\}\),
  that is, \(0\in U\)
  for all non-empty open subsets \(U\subseteq \pt(L)\).
  In particular, every pair of non-empty open sets in~\(\pt(L)\)
  has non-empty intersection.  Thus any non-empty open set is dense
  in~\(\pt(L)\).
  Since~\(\pt(L)\)
  is first countable, \(0\)
  has a countable neighbourhood basis.  Equivalently, there is a
  sequence \((J_n)_{n\in\N}\)
  in \(L\setminus\{0\}\)
  such that for every \(I\in L\setminus\{0\}\)
  there is some \(n\in\N\)
  with \(J_n \le I\).
  Since~\(g^*\)
  is injective and~\(U_{J_n}\)
  is dense in \(\pt(L)\), it follows that~\(g^*(J_n)\)
  is dense in~\(X\)
  for the topology~\(g^*(L)\).
  By Lemma~\ref{lem:open_surjective_properties}\ref{enu:Baire_GL}, \(X\)
  with this topology is a Baire space.  Then
  the intersection
  \[
  \bigcap g^*(J_n) = \bigcap \pi^{-1}(U_{J_n})
  = \pi^{-1}\left(\bigcap U_{J_n}\right)
  \]
  is dense and hence non-empty.  Let~\(x\)
  be an element of it.  Then \(\pi(x) \in U_{J_n}\)
  for all \(n\in\N\).
  So \(\pi(x)\in U_I\)
  for all \(I\in L\setminus\{0\}\).
  Then \(\pi(x)=0\)
  because~\(\pt(L)\)
  is~\(T_0\)
  and no open subset of~\(\pt(L)\) separates \(0\) and~\(\pi(x)\).
\end{proof}

\begin{proof}[of Theorem~\textup{\ref{the:open_surjective}}]
  The first assertion in Theorem~\ref{the:open_surjective} is
  Remark~\ref{rem:stupid_remark}.  It remains to prove the converse
  assertion.  So assume \(g\colon \Open(X)\to L\) to be open and
  surjective.
  By Lemma~\ref{lem:open_surjective_properties}.\ref{enu:second_PL},
  condition \ref{the:open_surjective_1} in the assertion is
  satisfied whenever~\(X\)
  is second countable.  Let us assume conditions
  \ref{the:open_surjective_1} and \ref{the:open_surjective_2}.  We
  first prove that \(\pi\colon X\to \pt(L)\)
  is surjective and then use this to show that~\(\pi\) is open.

  Let \(p\in L\).  We are going to
  find \(x\in X\) with \(\pi(x)=p\) by reducing the situation to
  that in Proposition~\ref{prop:Dixmier's}.  Let
  \[
  L_{\ge p} \defeq \setgiven{J\in L}{J\ge p}.
  \]
  This subset of~\(L\)
  has the minimal element~\(p\)
  and is closed under arbitrary joins and non-empty meets.
  So~\(L_{\ge p}\)
  is a locale.  An element of~\(L_{\ge p}\)
  is prime if and only if it is prime in~\(L\).
  And the prime elements in~\(L\) that do not belong to~\(L_{\ge p}\)
  are precisely those in~\(U_p\).
  Thus \(\pt(L_{\ge p})=\pt(L)\setminus U_p\).
  The topology on \(\pt(L_{\ge p})\)
  is the subspace topology from~\(\pt(L)\),
  and \(I\in L_{\ge p}\)
  corresponds to the (relatively) open subset
  \(U_I\setminus U_p \subseteq \pt(L) \setminus U_p\).
  Since \(\pi^{-1}(U_p) = g^*(p)\subseteq X\),
  the map~\(\pi\) restricts to a continuous map
  \[
  \pi_p\colon X\setminus g^*(p) \to \pt(L_{\ge p}),
  \]
  where \(X\setminus g^*(p)\subseteq X\)
  also carries the subspace topology.  This map is the adjunct of
  \[
  g_p^*\colon L_{\ge p} \to \Open(X\setminus g^*(p)),\qquad
  I\mapsto g^*(I) \setminus g^*(p),
  \]
  because
  \(\pi_p^{-1}(U_I\setminus U_p) = \pi^{-1}(U_I) \setminus g^*(p) =
  g^*(I)\setminus g^*(p)\).

  We claim that the locale map defined by~\(g^*_p\)
  is open and surjective.  Since \(g^*(p)\subseteq g^*(I)\)
  for every \(I\in  L_{\ge p}\)
  and the map \(g^*\colon L \to \Open(X)\)
  is injective, the map \(g^*_p\colon L_{\ge p} \to \Open(X\setminus g^*(p))\)
  is injective.  Recall that~\(g^*\) has the left adjoint
  \(g_!\colon \Open(X) \to L\)
  (that  satisfies~\eqref{eq:open_surjective}).  Since~\(g^*(p)\)
  is open, a subset \(V\subseteq X\setminus g^*(p)\)
  is relatively open if and only if \(V\cup g^*(p)\in\Open(X)\).
  If \(V\in \Open(X\setminus g^*(p))\),
  then \(g_!(V\cup g^*(p)) \ge g_!(g^*(p)) = p\),
  that is, \(g_!(V\cup g^*(p)) \in L_{\ge p}\).  This allows us to define
  \[
  g_{p!}\colon \Open(X\setminus g^*(p)) \to L_{\ge p},\qquad
  V\mapsto g_!(V\cup g^*(p)).
  \]
  Let \(I\in L_{\ge p}\) and \(V\in \Open(X\setminus g^*(p))\).  Then
  \[
 g_{p!}(V)=g_!(V\cup g^*(p))  \le I
  \iff
  V \cup g^*(p) \subseteq g^*(I)
  \iff
  V \subseteq g^*(I) \setminus g^*(p) = g^*_p(I).
  \]
  Hence~\(g_{p!}\)
  is a left adjoint for~\(g^*_p\).
  Moreover, \eqref{eq:open_surjective} for \(g_!\) and~\(g^*\) imply
  \begin{multline*}
    g_{p!}(V\cap g^*_p(I))
    = g_!\bigl( (V\cap g^*_p(I))\cup g^*(p) \bigr)
    = g_!\bigl((V \cup g^*(p))\cap  g^*(I) \bigr)
    \\= g_!(V \cup g^*(p)) \wedge I
    = g_{p!}(V)\wedge I,
  \end{multline*}
  which is~\eqref{eq:open_surjective} for~\(g_{p!}\)
  and \(g^*_p\).  Hence~\(g^*_p\) is open and surjective as claimed.

  Since~\(\pt(L)\) is first countable,  so is the subspace~\(\pt(L_{\ge p})\).
  And the closed subset \(X\setminus g^*(p) \subseteq X\)
  is a Baire space by the assumption~\ref{the:open_surjective_2}.  So
  Proposition~\ref{prop:Dixmier's} applies to
  \(X\setminus g^*(p)\),
  \(L_{\ge p}\),
  \(g^*_p\) and~\(\pi_p\),
  giving \(x\in X\setminus g^*(p)\)
  with \(\pi(x) = p\).
  This shows that \(\pi\colon X\to \pt(L)\) is surjective.
  Then~\(L\) is spatial by
  Lemma~\ref{lem:surjective_implies_spatial}.

  Now we show that~\(\pi\) is open.
  Let \(V\in \Open(X)\)
  and let \(g_!\colon \Open(X) \to L\)
  be the left adjoint of~\(g^*\).
  The proof will be finished once we show that
  \begin{equation}
    \label{eq:proof_open_map}
    \pi(V) = U_{g_!(V)},
  \end{equation}
  as~\(U_{g_!(V)}\)
  is an open subset of~\(\pt(L)\).
  The adjunctions between \(g_!\)
  and~\(g^*\)
  and between \(g^*\)
  and~\(\pi\)
  imply \(V \subseteq g^*(g_!(V)) = \pi^{-1}(U_{g_!(V)})\).
  Then \(\pi(V) \subseteq U_{g_!(V)}\).
  Thus~\eqref{eq:proof_open_map} follows if the
  restriction \(\pi|_V\colon V\to U_{g_!(V)}\) of~\(\pi\) is surjective.

  The subset \(L_{\le g_!(V)}\defeq \setgiven{J\in L}{J\le g_!(V)} \subseteq L\)
  has the maximal element~\(g_!(V)\)
  and is closed under joins and non-empty meets.  So it is a locale.
  If~\(p\)
  is a character on~\(L\),
  then~\(p\)
  restricted to~\(L_{\le g_!(V)}\)
  is either a character or the constant function~\(0\).
  Thus the prime elements in~\(L_{\le g_!(V)}\)
  are exactly those of the form \(p\wedge g_!(V)\)
  for a prime~\(p\) in~\(L\) with \(g_!(V) \not\le p\).  That is,
  \[
  \pt(L_{\le g_!(V)})
  \cong \setgiven{p\in \pt(L)}{g_!(V)\not\le p}
  = U_{g_!(V)}.
  \]
  We identify \(\pt(L_{\le g_!(V)})\)
  with~\(U_{g_!(V)}\).
  Let \(g^*_V\colon L_{\le g_!(V)} \to \Open(V)\)
  be the adjunct of
  \(\pi|_V\colon V\to U_{g_!(V)}=\pt(L_{\le g_!(V)})\).
  If \(I\in L_{\le g_!(V)}\), then
  \[
  g^*_V(I)=(\pi|_V)^{-1}(U_I)= \pi^{-1}(U_I) \cap V= g^*(I) \cap V.
  \]
  The left adjoint~\(g_!\)
  to~\(g^*\) restricts to \(g_{V!}\colon \Open(V) \to L_{\le g_!(V)}\).
  The latter is left adjoint to~\(g^*_V\)
  because if \(I\in L_{\le g_!(V)}\)
  and \(W\in \Open(V)\), then
  \[
  g_{V!}(W)=g_!(W)  \le I
  \iff
  W  \subseteq g^*(I)
  \iff
  W \subseteq g^*(I) \cap V =  g^*_V(I).
  \]
  The adjoint maps \(g^*_V\) and~\(g_{V!}\)
  satisfy~\eqref{eq:open_surjective}: if \(I\in L_{\le g_!(V)}\),
  \(W\in \Open(V)\), then
  \[
  g_{V!}( W \cap g^*_V(I))=
  g_!(W\cap g^*(I) \cap V)
  = g_!(W \cap g^*(I))
  = g_!(W) \wedge I= g_{V!}(W) \wedge I.
  \]
  Hence~\(g_V\) is  an open  locale map.  It is surjective
  because \(g_{V!}(V)=g_!(V)\),
  that is, \(g_{V!}\) preserves units (see
  Remark~\ref{rem:Joyal_Tierney}).

  The open subspace \(\pt(L_{\le g_!(V)}) \subseteq \pt(L)\)
  inherits first countability and the Baire property because both
  properties are hereditary for open subsets.   Thus \(V\),
  \(L_{\le g_!(V)}\),
  \(g^*_V\) and~\(\pi|_V\)
  have all the properties required in the theorem we are proving.
  Hence \(\pi|_V\colon V \to \pt(L_{\le g_!(V)}) = U_{g_!(V)}\)
  is surjective by what we have already proved.  This
  shows~\eqref{eq:proof_open_map} and finishes the proof.
\end{proof}

\section{The quasi-orbit space}
\label{sec:quasi-orbit}

Throughout this section, we fix a \Star{}homomorphism
\(\varphi\colon A\to \Mult(B)\).  Then \(\Ideals(A)\), \(\Ideals(B)\),
\(\Ideals^B(A)\) and \(\Ideals^A(B)\) are complete lattices, and
\(\Ideals^B(A)\cong\Ideals^A(B)\) are isomorphic (see
Proposition~\ref{pro:restrict_induce_ideals}).  Here \(\Ideals(A)\)
and~\(\Ideals(B)\) are spatial locales.  It is unclear whether
\(\Ideals^A(B)\cong \Ideals^B(A)\) is infinitely distributive in
general.  Extra assumptions in our theorems will ensure this, however.
And even if this failed, we may define topological spaces
\(\Prime^A(B)\cong\Prime^B(A)\) of prime elements in
\(\Ideals^A(B)\cong\Ideals^B(A)\) by copying the definition for
locales.  This gives a candidate for the ``quasi-orbit space''
of~\(\varphi\).  We shall, however, define the quasi-orbit space as a
quotient of~\(\check{A}\), following the usual practice (see
\cites{Boenicke-Li:Ideal, Effros-Hahn:Transformation_groups,
  Giordano-Sierakowski:Purely_infinite, Green:Local_twisted,
  GootmanRosenberg.StructureOfCrossedProducts,
  GootmanLazar.DualityCrossedProduct,
  Kwasniewski-Szymanski:Pure_infinite, Williams:Topology_primitive,
  Leung-Ng:Invariant_twisted, Takai:Quasi-orbit_continuous,
  Zeller-Meier:Produits_croises}).  We need the following assumption:
\begin{equation}
  \tag{\textup{JR}}
  \label{cond:C1}%
  \textit{joins of restricted ideals remain restricted.}
\end{equation}

\begin{lemma}
  \label{lem:induced_map_pi}
  Let \(\varphi\colon A\to \Mult(B)\).
  The following statements are equivalent:
  \begin{enumerate}
  \item \label{enu:induced_map_pi1}%
    condition \eqref{cond:C1} holds;
  \item \label{enu:induced_map_pi2}%
    the inclusion map \(\Ideals^B(A) \hookrightarrow \Ideals(A)\)
    is a frame homomorphism;
  \item \label{enu:induced_map_pi3}%
    \(\Ideals^B(A)\) is a spatial locale, and  there is a well defined
    continuous map \(\pi\colon \Prime(A) \to \Prime^B(A)\), where
    \begin{equation}
      \label{eq:definition_of_pi}
      \pi(\prid) \text{ is the largest restricted ideal in }
      A \text{ that is contained in }\prid.
    \end{equation}
  \end{enumerate}
  If these equivalent conditions hold, then
  \(\pi\colon \check{A} \subseteq \Prime(A) \to \Prime^B(A)\)
  is the map induced by the frame homomorphism
  \(\Ideals^B(A) \hookrightarrow \Ideals(A)\).
\end{lemma}

\begin{proof}
  By Proposition~\ref{pro:restrict_induce_ideals}.\ref{pro:restrict_induce_ideals6},  \(\Ideals^B(A) \subseteq \Ideals(A)\)
  is always closed under meets (intersections).  Hence
  \(\Ideals^B(A) \hookrightarrow \Ideals(A)\)
  is a frame homomorphism if and only if joins of restricted ideals
  remain restricted.
  Thus \ref{enu:induced_map_pi1}\(\iff\)\ref{enu:induced_map_pi2}.

  Assume that \(\Ideals^B(A) \hookrightarrow \Ideals(A)\) is a frame
  homomorphism.  Then infinite joins and meets in \(\Ideals^B(A)\)
  are the same as in \(\Ideals(A)\), so that \(\Ideals^B(A)\) is
  infinitely distributive.  And \(\Ideals^B(A)\) is complete by
  Proposition~\ref{pro:restrict_induce_ideals}.  So \(\Ideals^B(A)\)
  is a locale.  It is spatial because it embeds
  into the spatial locale~\(\Ideals(A)\).
  The induced continuous map
  \(\pi\colon \Prime(A) \to \Prime^B(A)\)
  verifies~\eqref{eq:induced_in_prime}.  That is,
  \(I\subseteq \pi(\prid)\)
  if and only if \(I\subseteq \prid\)
  for \(I\in\Ideals^B(A)\),
  \(\prid\in\Prime(A)\).
  This and \(\pi(\prid) \in \Prime^B(A) \subseteq \Ideals^B(A)\)
  imply~\eqref{eq:definition_of_pi}.  This proves
  \ref{enu:induced_map_pi2}\(\Longrightarrow\)\ref{enu:induced_map_pi3}.

  Now assume that \(\Ideals^B(A)\)
  is a spatial locale and that~\eqref{eq:definition_of_pi} defines a
  continuous map \(\pi\colon  \Prime(A) \to \Prime^B(A)\).
  Since both \(\Ideals^B(A)\) and \(\Ideals(A)\)
  are spatial locales, we may identify them with \(\Open(\Prime^B(A))\)
  and \(\Open(\Prime(A))\), respectively.  Then the frame  homomorphism
  \[
    \pi^{-1}\colon  \Open(\Prime^B(A))\to \Open(\Prime(A))
  \]
  becomes
  a map \(\pi^{-1}\colon \Ideals^B(A) \to \Ideals(A)\).
  Let \(g^*\colon \Ideals^B(A) \hookrightarrow \Ideals(A)\)
  be the inclusion.  If \(I\in\Ideals^B(A)\) and
  \(\prid\in\check{A}\), then~\eqref{eq:definition_of_pi} implies
  \begin{equation}
    \label{eq:adjunct_pi_G}
    \prid \in \pi^{-1}(U_I)
    \iff I \not\subseteq \pi(\prid)
    \iff I \not\subseteq \prid
    \iff \prid \in U_{g^*(I)}.
  \end{equation}
  Under our identifications, this means that \(\pi^{-1}=g^*\).
  Hence~\(g\) is a locale map.
  Thus
  \ref{enu:induced_map_pi3}\(\Longrightarrow\)\ref{enu:induced_map_pi2},
  and this finishes the proof.
\end{proof}

\begin{definition}
  Let \(\varphi\colon A\to \Mult(B)\) be such that~\eqref{cond:C1}
  holds.  Then the map
  \(\pi\colon\check{A} \subseteq \Prime(A)\to \Prime^B(A)\)
  in~\eqref{eq:definition_of_pi} is defined.  Let
  \(\prid\sim \prid[q]\) if and only if
  \(\pi(\prid)=\pi(\prid[q])\).  The \emph{quasi-orbit space}
  of~\(\varphi\) is~\(\check{A}/{\sim}\) with the quotient topology.
\end{definition}

Our first application of Theorem \ref{the:open_surjective} gives
necessary and sufficient conditions under which we may identify
the quasi-orbit space~\(\check{A}/{\sim}\) with~\(\Prime^A(B)\).

\begin{theorem}
  \label{the:prime_primitive}
  Let \(\varphi\colon A\to \Mult(B)\)
  be such that \eqref{cond:C1} holds.  If the continuous map
  \(\pi\colon \check{A} \to \Prime^B(A)\)
  defined in~\eqref{eq:definition_of_pi} is open and surjective, then
  \begin{equation}
    \tag{\textup{FR1}}
    \label{cond:C2}%
    r\circ i(I\cap J)=I\cap r\circ i(J)
    \qquad\text{for all }I\in\Ideals^B(A)
    \text{ and }J\in\Ideals(A).
  \end{equation}
  Conversely, if~\eqref{cond:C2} holds and \(\Prime^B(A)\)
  is first countable -- which follows if~\(\check{A}\)
  is second countable -- then the continuous map \(\pi\)
  is surjective and open and induces a homeomorphism
  \(\check{A}/{\sim} \cong \Prime^B(A)\).
\end{theorem}

\begin{proof}
  Proposition~\ref{pro:restrict_induce_ideals}.\ref{pro:restrict_induce_ideals7}
  says that the inclusion
  \(g^*\colon \Ideals^B(A) \hookrightarrow \Ideals(A) \cong
  \Open(\check{A})\)
  has a left adjoint
  \(g_!\defeq r\circ i\colon \Open(\check{A}) \cong \Ideals(A) \to
  \Ideals^B(A)\).  Condition~\eqref{cond:C2} is the
  Frobenius reciprocity condition~\eqref{eq:open_surjective} for
  these maps.  Thus~\(g\)
  is an open surjective locale map if and only
  if~\eqref{cond:C2} holds.
  Equation~\eqref{eq:adjunct_pi_G} implies \(\pi^{-1}(U_I)=U_{g^*(I)}\)
  for all \(I\in\Ideals^B(A)\).
  Thus~\(\pi\)
  is the adjunct of~\(g\).
  So Theorem~\ref{the:open_surjective} implies all assertions except
  \(\check{A}/{\sim} \cong \Prime^B(A)\).
  This follows from the definition of the equivalence
  relation~\(\sim\) whenever~\(\pi\) is an open surjection.
\end{proof}

\begin{lemma}
  \label{lem:separate_ideals_into_primes}
  Let \(\varphi\colon A\to \Mult(B)\).
  The following are equivalent:
  \begin{enumerate}
  \item \label{enu:separate_ideals_into_primes1}%
    \(A\) separates ideals in~\(B\);
  \item \label{enu:separate_ideals_into_primes2}%
    \(\Ideals^B(A)\) is a spatial locale, and the
    map \(r \colon \Ideals(B) \to \Ideals^B(A)\) restricts to
    a homeomorphism \(r\colon \Prime(B) \to \Prime^B(A)\).
  \end{enumerate}
  Thus if \(\Prime^B(A)\)
  is first countable and the above equivalent conditions hold, then
  \(\check{B} = \Prime(B)\cong\Prime^B(A)\).
\end{lemma}

\begin{proof}
  If~\(A\) separates ideals in~\(B\), then
  \(r\colon \Ideals(B) \to \Ideals^B(A)\) is a lattice isomorphism.
  Then \(\Ideals^B(A)\) is a spatial locale because \(\Ideals(B)\) is,
  and~\(r\) restricts to a homeomorphism
  \(r\colon \Prime(B) \to \Prime^B(A)\).  Conversely, if
  \(\Ideals^B(A)\) is a spatial locale and~\(r\) is a homeomorphism
  \(\Prime(B) \congto \Prime^B(A)\), then
  \(r\colon \Ideals(B) \to \Ideals^B(A)\) is a frame isomorphism
  because \(\Ideals(B) \cong \Open(\Prime(B))\) and
  \(\Ideals^B(A) \cong \Open(\Prime^B(A))\).  Then~\(A\) separates
  ideals in~\(B\).  This proves the equivalence between
  \ref{enu:separate_ideals_into_primes1} and
  \ref{enu:separate_ideals_into_primes2}.  These conditions imply
  \(\Prime(B)\cong \Prime^B(A)\).  If this is first countable, then
  \(\check{B} = \Prime(B)\) by
  Corollary~\ref{cor:first_countable_implies_sober}.
\end{proof}

\begin{corollary}
  \label{cor:primitive_ideal_space_description}
  Let \(\varphi\colon A\to \Mult(B)\).
  Suppose that \(\Prime^B(A)\)
  is first countable.  Identify \(\Ideals(A)\)
  and~\(\Ideals(B)\)
  with \(\Open(\check{A})\)
  and~\(\Open(\check{B})\),
  respectively.  There is a continuous open surjection
  \(\pi\colon \check{A} \to \check{B}\)
  with \(r=\pi^{-1}\)
  if and only if~\(A\)
  separates ideals in~\(B\)
  and the conditions \eqref{cond:C1} and~\eqref{cond:C2} hold.  Then
  \(\check{B}\cong \Prime^B(A) \cong \check{A}/{\sim}\).
\end{corollary}

\begin{proof}
  If \(r=\pi^{-1}\)
  for a continuous open surjection
  \(\pi\colon \check{A} \to \check{B}\),
  then~\(r\)
  is injective, that is, \(A\)
  separates ideals in~\(B\).
  Then Lemma~\ref{lem:separate_ideals_into_primes} implies
  \(\check{B}\cong \Prime^B(A)\).
  Now all assertions follow from Lemma~\ref{lem:induced_map_pi} and
  Theorem~\ref{the:prime_primitive}.
\end{proof}

In order for~\(r\)
to restrict to a continuous map \(r\colon \Prime(B) \to \Prime^B(A)\),
we need the following condition:
\begin{equation}
  \tag{\textup{MI$_\mathrm{f}$}}
  \label{cond:C3}%
  \textit{finite meets \textup{(}intersections\textup{)} of induced
    ideals are again induced.}
\end{equation}

\begin{lemma}
  \label{lem:induced_map_rho}
  Let \(\varphi\colon A\to \Mult(B)\).  The following are equivalent:
  \begin{enumerate}
  \item \label{enu:induced_map_rho1}%
    condition \eqref{cond:C3} holds;
  \item\label{enu:induced_map_rho2}%
    the inclusion map \(\Ideals^A(B) \hookrightarrow \Ideals(B)\)
    is a frame homomorphism;
  \item \label{enu:induced_map_rho3}%
    \(\Ideals^B(A)\) is a spatial locale, and the map
    \(r\colon \Ideals(B) \to \Ideals^B(A)\)
    restricts to a continuous map
    \(r\colon \Prime(B) \to \Prime^B(A)\).
  \end{enumerate}
  If these equivalent conditions hold, then
  the continuous map \(\Prime(B) \to \Prime^A(B)\)
  induced by \(\Ideals^A(B) \hookrightarrow \Ideals(B)\)
  is equal to the restriction of
  \(i\circ r\colon \Ideals(B)\to \Ideals^A(B)\).
\end{lemma}

\begin{proof}
  The inclusion
  \(\Ideals^A(B) \hookrightarrow \Ideals(B)\)
  preserves joins by
  Proposition~\ref{pro:restrict_induce_ideals}.\ref{pro:restrict_induce_ideals6}.
  So it is a frame homomorphism if and only if it preserves finite
  meets.  Hence
  \ref{enu:induced_map_rho1}\(\iff\)\ref{enu:induced_map_rho2}.

  Assume that \(\Ideals^A(B) \hookrightarrow \Ideals(B)\) is a frame
  homomorphism.  Then \(\Ideals^A(B)\cong \Ideals^B(A)\) is a spatial
  locale as it is embedded in the spatial locale~\(\Ideals(B)\).  The
  inclusion induces a continuous map \(\Prime(B) \to \Prime^A(B)\).
  By~\eqref{eq:induced_in_prime}, it maps \(\prid\in \Prime(B)\) to
  the largest induced ideal contained in~\(\prid\).  This is exactly
  \(i\circ r(\prid)\) (see Proposition
  \ref{pro:restrict_induce_ideals}.\ref{pro:restrict_induce_ideals8}).
  Hence \(i\circ r\colon \Prime(B) \to \Prime^A(B)\)
  is a well defined continuous map.  The same holds for
  \(r\colon \Prime(B) \to \Prime^B(A)\),
  as it is a composite \(r=r\circ i\circ r\)
  of the continuous map \(i\circ r\colon \Prime(B) \to \Prime^A(B)\)
  and the homeomorphism \(r\colon \Prime^A(B)\cong \Prime^B(A)\).
  This proves
  \ref{enu:induced_map_pi2}\(\Longrightarrow\)\ref{enu:induced_map_pi3}.

  Finally, assume \ref{enu:induced_map_pi3}.  Then
  \(i\circ r\colon \Prime(B) \to \Prime^A(B)\)
  is a continuous map.  Since \(\Ideals^A(B)\cong \Ideals^B(A)\)
  is a spatial locale, we may identify \(\Ideals^A(B)\)
  and~\(\Ideals(A)\)
  with \(\Open(\Prime^B(A))\) and \(\Open(\Prime(A))\), respectively.
  Let \(J\in\Ideals^A(B)\) and \(\prid\in\check{B}\),
  and write~\(g^*\)
  for the inclusion \(\Ideals^A(B) \hookrightarrow \Ideals(B)\).
  Since \(i\circ r(\prid)\)
  is the largest induced ideal contained in~\(\prid\),
  \[
  \prid \in (i\circ r)^{-1}(U_J)
  \iff J\not\subseteq i\circ r(\prid)
  \iff J \not\subseteq \prid
  \iff \prid \in U_{g^*(J)}.
  \]
  When we identify
  \(\Open(\Prime^B(A))\) with \(\Ideals^A(B)\) and
  \(\Open(\Prime(A))\) with \(\Ideals(B)\),
  then the frame  homomorphism \((i\circ r)^{-1}\colon
  \Open(\Prime^B(A))\to \Open(\Prime(A))\)
  becomes the inclusion
  \(g^*\colon \Ideals^A(B) \hookrightarrow \Ideals(B)\).
  Thus
  \ref{enu:induced_map_pi3}\(\Longrightarrow\)\ref{enu:induced_map_pi2}.
  This finishes the proof.
\end{proof}

In summary, if conditions \eqref{cond:C1} and~\eqref{cond:C3} hold,
then there is a commutative diagram of continuous maps
\begin{equation}
  \label{eq:continuous_maps_diagram}
  \begin{tikzcd}
    \check{B} \subseteq \Prime(B) \arrow[r,"i\circ r"] \arrow[rd,   "r"  ]&
    \Prime^A(B)  \arrow[d,  "r"' ,  "\cong", shift right=.53em]   \\
    \Prime(A)   \arrow[r,"\pi"]  &
    \Prime^B(A)  \arrow[u, "i"',  shift right=.53em] \\
    \check{A} \arrow[u, hookrightarrow] \arrow[r] &
    \check{A}/{\sim} \arrow[u, hookrightarrow, "\pi_*"']
  \end{tikzcd}
\end{equation}
Here \(i\circ \pi =i\)
if and only if \(\pi=r\circ i\)
if and only if \(\Ideals(A) = \Ideals^B(A)\).
So typically there is no natural map between \(\Prime(B)\)
and \(\Prime(A)\).
The best substitute seems to be the map
\(\check{B} \to \check{A}/{\sim}\)
that exists when the map
\(\tilde{\pi}\colon \check{A}/{\sim} \to \Prime^B(A)\) is a homeomorphism:

\begin{definition}
  Let \(\varphi\colon A\to \Mult(B)\) be such that \eqref{cond:C1},
  \eqref{cond:C2}, \eqref{cond:C3} are satisfied and \(\Prime^B(A)\) is
  first countable.  Then
  \(\tilde{\pi}\colon \check{A}/{\sim} \to \Prime^B(A)\) is a
  homeomorphism by Theorem~\ref{the:prime_primitive} and
  \(r\colon \check{B} \to \Prime^B(A)\) is continuous by
  Lemma~\ref{lem:induced_map_rho}.  The continuous map
  \(\varrho\defeq \tilde{\pi}^{-1}\circ r\colon \check{B} \to
  \check{A}/{\sim}\) is called the \emph{quasi-orbit map}
  for~\(\varphi\).
\end{definition}

Theorem~\ref{the:open_surjective} implies necessary and sufficient
conditions for the quasi-orbit map to be open surjective:

\begin{lemma}
  \label{lem:galois_insertion}
  Let \(\varphi\colon A\to \Mult(B)\)
  be such that~\eqref{cond:C3} holds.  The inclusion map
  \(\Ideals^A(B) \hookrightarrow \Ideals(B)\)
  has a left adjoint \(F\colon \Ideals(B) \to \Ideals^A(B)\)
  if and only if
  \begin{equation}
    \tag{\textup{MI}}
    \label{cond:C3'}%
    \textit{arbitrary meets of induced ideals are again induced.}
  \end{equation}
  If this holds and \(J\in \Ideals(B)\), then
  \begin{equation}
    \label{eq:definition_of_F}
    F(J) \text{ is the meet of all induced ideals that contain  } J.
  \end{equation}
\end{lemma}

\begin{proof}
  The inclusion \(\Ideals^A(B) \hookrightarrow \Ideals(B)\)
  has a left adjoint if and only if for every
  \(J\in \Ideals(B)\)
  there is a least induced ideal \(F(J)\)
  with \(J\subseteq F(J)\).
  If~\eqref{cond:C3'} holds, then \(F(J)\)
  is the meet of all induced ideals that contain \(J\in \Ideals(B)\).
  Conversely, assume that the inclusion
  \(\Ideals^A(B) \hookrightarrow \Ideals(B)\)
  has a left adjoint
  \(F\colon \Ideals(B) \to\Ideals^A(B)\).
  Then the inclusion preserves all meets by
  \cite{Davey-Priestley:Lattices_order}*{Proposition 7.31}.
  Equivalently, the meet of a family of induced ideals is induced.
\end{proof}

\begin{theorem}
  \label{the:prime_primitive2}
  Let \(\varphi\colon A\to \Mult(B)\)
  be a \Star{}homomorphism. If
    \(r\colon \Ideals(B) \to \Ideals^B(A)\)
    restricts to a well defined continuous, open and surjective map
    \(r\colon \Prime(B) \to \Prime^B(A)\),
  then~\eqref{cond:C3'} holds and the map
  \(F\colon \Ideals(B) \to \Ideals^A(B)\)
  given by~\eqref{eq:definition_of_F} satisfies
  \begin{equation}
    \tag{\textup{FR2}}
    \label{cond:C4}%
    F(I\cap J)=I\cap F(J)
    \qquad \text{for every }I\in\Ideals^A(B)
    \text{ and }J\in\Ideals(B).
  \end{equation}
  Conversely, if \eqref{cond:C3'} and~\eqref{cond:C4} hold and
  \(\Prime^B(A)\)
  is first countable -- which follows if~\(\check{B}\)
  is second countable -- then
  \(r\colon \check{B} \subseteq \Prime(B) \to \Prime^B(A)\)
  is a well defined continuous, open surjection.
\end{theorem}

\begin{proof}
  Assume first that \(r\colon \Prime(B) \to \Prime^B(A)\) is a well
  defined continuous open surjection.  Then
  \(i\circ r\colon \Prime(B) \to \Prime^A(B)\) is a continuous open
  surjection.  By Lemma~\ref{lem:induced_map_rho}, the inclusion
  \(\Ideals^A(B) \hookrightarrow \Ideals(B)\) is a frame homomorphism
  whose adjunct continuous map may be identified with
  \(i\circ r\colon \Prime(B) \to \Prime^A(B)\).  Hence by the first
  part of Theorem~\ref{the:open_surjective}, the inclusion
  \(\Ideals^A(B) \hookrightarrow \Ideals(B)\) has a left adjoint
  \(F\colon \Ideals(B)\to \Ideals^A(B)\) which
  satisfies~\eqref{cond:C4}.  Lemma~\ref{lem:galois_insertion} shows
  that~\(F\) is given by~\eqref{eq:definition_of_F} and
  that~\eqref{cond:C3'} holds.  Conversely, assume \eqref{cond:C3'}
  and~\eqref{cond:C4} and that \(\Prime^B(A)\) is first countable
  or~\(\check{B}\) is second countable.  By
  Lemma~\ref{lem:galois_insertion},
  \(F\colon \Ideals(B) \to \Ideals^A(B)\) is left adjoint to the
  inclusion \(\Ideals^A(B) \hookrightarrow \Ideals(B)\).  The second
  part of Theorem~\ref{the:open_surjective} applied to the inclusion
  \(\Ideals^A(B) \hookrightarrow \Ideals(B)=\Open(\check{B})\) shows
  that the induced continuous map \(\check{B} \to \Prime^A(B)\) is
  surjective and open.  This gives the assertion by
  Lemma~\ref{lem:induced_map_rho}.
\end{proof}

\begin{corollary}
  \label{cor:prime_primitive2}
  Suppose that  \eqref{cond:C1}, \eqref{cond:C2} and~\eqref{cond:C3} are
  satisfied and \(\Prime^B(A)\)
  is first countable.  The quasi-orbit map
  \(\varrho\colon \check{B} \to \check{A}/{\sim}\)
  is open and surjective if and only if \eqref{cond:C3'}
  and~\eqref{cond:C4} are satisfied.
\end{corollary}

\section{Symmetric ideals}
\label{sec:symmetric}

Our next task is to verify the assumptions \eqref{cond:C1},
\eqref{cond:C2} and~\eqref{cond:C3} in interesting cases.  Then the
quasi-orbit map exists.  As we shall see, restricted ideals often
satisfy the equivalent conditions in the following lemma:

\begin{lemma}
  \label{lem:symmetric_induced_ideal}
  Let \(\varphi\colon A\to \Mult(B)\)
  and \(I\in\Ideals(A)\).  The following are equivalent:
  \begin{enumerate}
  \item \label{lem:symmetric_induced_ideal1}%
    the map \(I\to\Mult(i(I))\)
    induced by \(\varphi\colon A\to \Mult(B)\) is non-degenerate;
   \item \label{lem:symmetric_induced_ideal2}%
    \(i(I)=\varphi(I)B\varphi(I)\);
  \item \label{lem:symmetric_induced_ideal3}%
    \(i(I)=\varphi(I)B\);
  \item \label{lem:symmetric_induced_ideal4}%
    \(i(I)=B\varphi(I)\);
  \item \label{lem:symmetric_induced_ideal5}%
    \(\varphi(I) B = B \varphi(I)\).
  \end{enumerate}
  If \(\varphi\colon A\to B\)
  is injective and the above equivalent conditions hold, then
  \(I\in\Ideals^B(A)\).
\end{lemma}

\begin{proof}
  If \(I\to\Mult(i(I))\)
  is non-degenerate, then
  \(i(I)= \varphi(I) i(I) \varphi(I)\subseteq \varphi(I) B
  \varphi(I)\).
  We claim that the reverse inclusion
  \(\varphi(I) B \varphi(I)\subseteq B \varphi(I) B = i(I)\)
  always holds.  Let \((e_n)_{n\in N}\)
  be an approximate unit in~\(B\).
  Then \(\lim e_n\cdot x\cdot e_n = x\)
  for all \(x\in \varphi(I) B \varphi(I)\)
  because \(\varphi(I) B \varphi(I) \subseteq B\).
  Thus~\(x\)
  belongs to the norm closure of
  \(B \varphi(I) B \varphi(I) B \subseteq B \varphi(I) B\)
  as asserted.  Thus~\ref{lem:symmetric_induced_ideal1}
  implies~\ref{lem:symmetric_induced_ideal2}.

  If \(i(I)=\varphi(I) B \varphi(I)\),
  then \(i(I)\subseteq \varphi(I) B\).
  The reverse inclusion
  \(\varphi(I) B\subseteq B \varphi(I) B = i(I)\)
  always holds, the proof is the same approximate unit argument as
  above.  Thus~\ref{lem:symmetric_induced_ideal2}
  implies~\ref{lem:symmetric_induced_ideal3}.  Since \(i(I)\)
  is self-adjoint, \ref{lem:symmetric_induced_ideal3}
  and~\ref{lem:symmetric_induced_ideal4} are equivalent.  Clearly, the
  equivalent conditions \ref{lem:symmetric_induced_ideal3}
  and~\ref{lem:symmetric_induced_ideal4}
  imply~\ref{lem:symmetric_induced_ideal5}.  If
  \(\varphi(I) B = B \varphi(I)\),
  then
  \(i(I) = B \varphi(I) B = B \varphi(I) \varphi(I) B = \varphi(I) B
  \varphi(I) B = \varphi(I) i(I)\).
  Hence~\ref{lem:symmetric_induced_ideal5}
  implies~\ref{lem:symmetric_induced_ideal1}.  This proves the first
  part of the assertion.

  Now assume that \(\varphi\colon A\to B\)
  is injective and let \(I\in\Ideals(A)\)
  satisfy the equivalent conditions
  \ref{lem:symmetric_induced_ideal1}--\ref{lem:symmetric_induced_ideal5}.
  The inclusion \(I\subseteq r (i(I))\)
  always holds by
  Proposition~\ref{pro:restrict_induce_ideals}.\ref{pro:restrict_induce_ideals2}.
  We show the reverse inclusion.  Let \((e_n)_{n\in N}\)
  be an approximate unit for~\(I\).
  By assumption, \(i(I) = \varphi(I) B\).
  So \(\lim \varphi(e_n) x = x\)
  holds for all \(x\in i(I)\).
  Since~\(\varphi\)
  is injective and \(\varphi(r\circ i(I)) \subseteq i(I)\),
  this implies \(\lim e_n x = x\)
  for all \(x\in r (i(I))\).
  Thus \(r (i(I))\subseteq I\).
\end{proof}

\begin{definition}
  \label{def:symmetric_ideal_inclusion}
  We call an ideal \(I\in\Ideals(A)\)
  \emph{symmetric} for \(\varphi\colon A\to \Mult(B)\)
  if it satisfies the equivalent conditions in
  Lemma~\ref{lem:symmetric_induced_ideal}.  We call the
  \Star{}homomorphism \(\varphi\colon A\to \Mult(B)\)
  \emph{symmetric} if all restricted ideals are symmetric.
\end{definition}

The name ``symmetric'' is suggested by
condition~\ref{lem:symmetric_induced_ideal5} in
Lemma~\ref{lem:symmetric_induced_ideal}.

\begin{remark}
  \label{rem:generalised_symmetric_inclusions}
  For a usual inclusion \(A\subseteq B\),
  all symmetric ideals are restricted, by
  Lemma~\ref{lem:symmetric_induced_ideal}.  There are, however,
  important symmetric generalised \(\Cst\)\nb-inclusions
  \(A\subseteq \Mult(B)\)
  that admit symmetric ideals that are not restricted (see
  Example~\ref{exa:symmetric_not_invariant}).
\end{remark}

\begin{lemma}
  \label{lem:symmetric_restricted_intersection}
  If \(I,J\in\Ideals(A)\)
  and~\(I\)
  is symmetric, then \(i(I\cap J) = i(I) \cap i(J)\)
  and \(r\circ i(I\cap J) = r\circ i(I) \cap r\circ i(J)\).
  This is equal to \(I\cap r\circ i(J)\) if \(I\in\Ideals^B(A)\).

  The set of symmetric ideals is closed under joins.
\end{lemma}

\begin{proof}
  If \(\varphi(I)B = B\varphi(I)\),
  then
  \begin{align*}
    i(I) \cap i(J)
    &= i(I) i(J)
    = (\varphi(I) B) (B \varphi(J) B)
    = \varphi(I) B \varphi(J) B
    \\
    &= B \varphi(I) \varphi(J) B
    = B\varphi(I\cap J)B
    = i(I\cap J).
  \end{align*}
  Then \(r\circ i(I\cap J) = r\circ i(I) \cap r\circ i(J)\)
  because~\(r\)
  preserves meets by
  Proposition~\ref{pro:restrict_induce_ideals}.\ref{pro:restrict_induce_ideals4}.

  Let \((I_\alpha)_{\alpha \in S}\)
  be symmetric ideals.  Then \(\varphi(\bigvee I_\alpha)B\)
  is the closed linear span of \(\varphi(I_\alpha)B = B\varphi(I_\alpha)\),
  and this is equal to \(B\varphi(\bigvee I_\alpha)\).
  So \(\bigvee I_\alpha\) is symmetric as well.
\end{proof}

\begin{corollary}
  \label{cor:symmetric_nice}
  If \(A\subseteq B\)
  is a symmetric inclusion, then conditions \eqref{cond:C1},
  \eqref{cond:C2} and~\eqref{cond:C3} hold.  If
  \(\varphi\colon A\to \Mult(B)\)
  is a symmetric \Star{}homomorphism, then conditions \eqref{cond:C2}
  and~\eqref{cond:C3} hold.
\end{corollary}

\begin{proof}
  Any induced ideal is induced by a restricted ideal by
  Proposition~\ref{pro:restrict_induce_ideals}.\ref{pro:restrict_induce_ideals2}.
		Therefore, if all restricted ideals are symmetric, then
  Lemma~\ref{lem:symmetric_restricted_intersection} implies
  conditions \eqref{cond:C2} and~\eqref{cond:C3}.  If the sets of
  restricted and symmetric ideals are equal, then the last part of
  Lemma~\ref{lem:symmetric_restricted_intersection} also
  yields~\eqref{cond:C1}.  The latter is satisfied when \(A\subseteq B\)
  by the last part of Lemma~\ref{lem:symmetric_induced_ideal}.
\end{proof}

It is unclear whether all symmetric \emph{generalised}
\(\Cst\)\nb-inclusions
satisfy~\eqref{cond:C1} because there may be symmetric ideals that are
not restricted (see
Remark~\ref{rem:generalised_symmetric_inclusions}).
Condition~\eqref{cond:C3'} may fail for symmetric inclusions (see
Example~\ref{exa:Cstar_over_X} below).

\section{Group actions and regular inclusions}
\label{sec:group_crossed}

The theory above suggests the following programme to study the
ideal structure for any type of
\(\Cst\)\nb-inclusion
\(A\hookrightarrow \Mult(B)\).
First verify the locale-theoretic conditions \eqref{cond:C1},
\eqref{cond:C2} and~\eqref{cond:C3}; a good Ansatz for this is showing
that the inclusion is symmetric.  This requires a good characteristion
of the restricted ideals.  Usually, they are ``invariant'' in a
suitable sense.  Then the assumptions needed for
Theorem~\ref{the:prime_primitive} and the existence of the quasi-orbit
map \(\varrho\colon \check{B} \to \check{A}/{\sim}\)
are in place, except for the first or second countability assumptions,
which remain assumptions in all the following theorems.  Secondly,
one may try
to verify the conditions \eqref{cond:C3'} and~\eqref{cond:C4}, which are
then equivalent to openness and surjectivity of the quasi-orbit map
(see Corollary~\ref{cor:prime_primitive2}).  This in turn requires a good
understanding of induced ideals, which are usually ``invariant'' in
some dual sense.

In this section, we apply the above programme to some prototypical
examples.  We begin with crossed products for actions of locally
compact groups by automorphisms.  Here we recover classical
results that played a crucial role in the study of the Effros--Hahn
Conjecture and related problems.  We generalise these results to Fell
bundles over locally compact groups.  We finish this section with the
case of regular inclusions, which we treat by relating them to Fell
bundles over inverse semigroups.

\subsection{Crossed products for actions of locally compact groups}
\label{sec:crossed_groups}

Let~\(G\)
be a locally compact group and let \(\alpha\colon G\to\Aut(A)\)
be a continuous group action.  A \emph{crossed product} is a
\(\Cst\)\nb-algebra~\(B\)
with surjective maps
\(A\rtimes_\alpha G \onto B \onto A\rtimes_{\alpha,\red} G\)
whose composition is the regular representation
\(\lambda\colon A\rtimes_\alpha G \onto A\rtimes_{\alpha,\red} G\).
The canonical \Star{}homomorphism
\(A\to \Mult(A\rtimes_\alpha G)\)
gives a generalised \(\Cst\)\nb-inclusion
\(\varphi\colon A\to \Mult(B)\)
because the canonical \Star{}homomorphism
\(A\to \Mult(A\rtimes_{\alpha,\red} G)\)
is injective.

\begin{proposition}
  \label{pro:restricted_group_crossed}
  Let \(\varphi\colon A\to \Mult(B)\)
  be the canonical generalised \(\Cst\)\nb-inclusion
  of~\(A\) into a crossed product~\(B\).
  \begin{enumerate}
  \item \label{pro:restricted_group_crossed1}%
    \(\Ideals^B(A)=\Ideals^\alpha(A)\)
    is the lattice of \(\alpha\)\nb-invariant ideals in~\(A\);
  \item \label{pro:restricted_group_crossed2}%
    an ideal \(J\idealin B\)
    is induced if and only if it is the image of
    \(I\rtimes_{\alpha|_I} G\)
    in~\(B\) for some \(\alpha\)\nb-invariant ideal \(I\idealin A\).
  \end{enumerate}
   The \(\Cst\)\nb-inclusion
  \(\varphi\colon A\to \Mult(B)\)
  is symmetric and conditions \eqref{cond:C1}, \eqref{cond:C2}
  and~\eqref{cond:C3} hold.
 \end{proposition}

\begin{proof}
  Statements \ref{pro:restricted_group_crossed1}
  and~\ref{pro:restricted_group_crossed2} are well known to experts,
  and we prove them in greater generality in the proof of
  Proposition~\ref{pro:restricted_group_Fell} below.  Since the
  canonical homomorphisms \(I\to \Mult(I\rtimes_\alpha G)\)
  are non-degenerate, statements~\ref{pro:restricted_group_crossed1}
  and~\ref{pro:restricted_group_crossed2} imply that~\(\varphi\)
  is symmetric.  Thus conditions \eqref{cond:C2} and~\eqref{cond:C3} follow
  from Corollary~\ref{cor:symmetric_nice}.
  The closed linear span of a family of
  \(\alpha\)\nb-invariant
  ideals is again \(\alpha\)\nb-invariant.
  So~\ref{pro:restricted_group_crossed1} implies
  condition~\eqref{cond:C1}.
\end{proof}

If~\(G\) is discrete, then all symmetric ideals are restricted by
Lemma~\ref{lem:symmetric_induced_ideal}.  This fails for locally
compact~\(G\):

\begin{example}
  \label{exa:symmetric_not_invariant}
  Let~\(G\)
  be a locally compact group that is not discrete.  Let
  \(A=\Cont_0(G)\)
  and let~\(\alpha\)
  be the translation action.  So
  \(B\defeq A\rtimes G \cong \Comp(L^2 (G))\).
  Let \(e\in G\)
  be the unit element.  The ideal
  \(I\defeq \Cont_0(G\setminus\{e\})\)
  is not invariant and not restricted.  The closed right ideal
  \(I\cdot B\)
  in \(B=\Comp(L^2 (G))\)
  consists of those compact operators whose image is contained in the
  closure of \(I\cdot L^2(G)\).
  Now \(I\cdot L^2 (G)\)
  is dense in~\(L^2 (G)\)
  because \(\{e\}\)
  is a set of measure zero.  Thus \(I\cdot B= B\).
  This implies \(B\cdot I = B\)
  by taking adjoints.  So~\(I\)
  is a symmetric ideal that is not restricted.
\end{example}

Let \(\varphi\colon A \to \Mult(B)\)
be the canonical generalised inclusion into a crossed
product~\(B\).
Proposition~\ref{pro:restricted_group_crossed}
and Lemma~\ref{lem:induced_map_pi} yield a continuous map
\[
\pi\colon \Prime(A) \to \Prime(\Ideals^\alpha(A))
\]
that maps \(\prid\in \Prime(A)\)
to the largest invariant ideal contained in~\(\prid\), that is,
\[
\pi(\prid)=\bigcap_{g\in G} \alpha_g(\prid).
\]
The open subset of~\(\Prime(A)\)
corresponding to this intersection is the complement of the closure of
the orbit \(G\cdot \prid=\setgiven{\alpha_g(\prid)}{g\in G}\).
Thus \(\prid_1 \sim \prid_2\)
for \(\prid_1, \prid_2\in\Prime(A)\)
if and only if
\(\overline{G\cdot \prid_1} = \overline{G\cdot \prid_2}\).
This is the usual definition of the quasi-orbit space for a group
action (compare \cite{Takai:Quasi-orbit_continuous}); when~\(A\)
is commutative or separable, we may replace \(\Prime(A)\)
by~\(\check{A}\) by
Corollary~\ref{cor:first_countable_implies_sober}.
So our definition of the quasi-orbit space for a generalised
\(\Cst\)\nb-inclusion
extends the usual definition for group actions.  In particular, our
results yield the canonical quasi-orbit map
\(\check{B} \to \check{A}/{\sim}\)
(see, for instance, \cite{Takai:Quasi-orbit_continuous}*{(3.2)},
\cite{Green:Local_twisted}*{pp.~221--223},
\cite{GootmanRosenberg.StructureOfCrossedProducts}*{p.~290}, or
\cite{GootmanLazar.DualityCrossedProduct}*{p.~620}).

\begin{theorem}
  \label{thm:group_restricted_ideals_nice}
  Let \(\varphi\colon A\to \Mult(B)\)
  be the canonical generalised \(\Cst\)\nb-inclusion
  of~\(A\)
  into a crossed product~\(B\).
  The quasi-orbit space for~\(\varphi\)
  exists and coincides with the usual quasi-orbit space for the
  action~\(\alpha\).
  If\/ \(\Prime(\Ideals^\alpha(A))\)
  is first countable or~\(\check{A}\)
  is second countable, then the quasi-orbit map exists:
  \begin{enumerate}
  \item The map \(\pi\colon \check{A}\to\Prime(\Ideals^\alpha(A))\),
    \(\prid\mapsto \bigcap_{g\in G} \alpha_g(\prid)\),
    is continuous, open and surjective.  It descends to a
    homeomorphism
    \(\widetilde{\pi}\colon
    \check{A}/{\sim}\to\Prime(\Ideals^\alpha(A))\).
    The quasi-orbit space~\(\check{A}/{\sim}\)
    is a quotient of~\(\check{A}\) by an open equivalence relation.
  \item There is a continuous map
    \(\varrho\colon \check{B} \to \check{A}/{\sim}\),
    \(\prid\mapsto \widetilde{\pi}^{-1}(r(\prid))\).
    It identifies \(\check{B}\cong \check{A}/{\sim}\)
    if and only if~\(A\)
    separates ideals in~\(B\).
    And then \(B=A\rtimes_{\alpha,\red} G\).
  \end{enumerate}
\end{theorem}

\begin{proof}
  The assertions follow from Theorem~\ref{the:prime_primitive},
  Lemma \ref{lem:induced_map_rho} and the discussion above.  For
  the last part, compare with
  Corollary~\ref{cor:primitive_ideal_space_description}, and use that
  if~\(A\)
  detects ideals in~\(B\),
  then the kernel of the surjection
  \(B \onto A\rtimes_{\alpha,\red} G\) has to be trivial.
\end{proof}

The existence of the quasi-orbit map
\(\varrho\colon \check{B} \to \check{A}/{\sim}\)
is established in \cites{Green:Local_twisted,
  GootmanRosenberg.StructureOfCrossedProducts,
  GootmanLazar.DualityCrossedProduct, Takai:Quasi-orbit_continuous}
if~\(A\)
is separable and~\(G\)
is second countable and amenable.  In fact, under these
assumptions Gootman and Lazar prove that~\(\varrho\)
is open and surjective (see
\cite{GootmanLazar.DualityCrossedProduct}*{Theorem~4.8}).
As we show in Example~\ref{exa:MI_fails_tensor} below (see also
Corollary~\ref{cor:amenable_group_characterization}),
the map~\(\varrho\) may fail to be open when~\(G\) is not amenable
and~\(B\) is the full crossed product.
Nevertheless, we manage to improve \cite{GootmanLazar.DualityCrossedProduct}*{Theorem~4.8}
 by applying  Corollary~\ref{cor:prime_primitive2}
to the reduced
crossed product \(B=A\rtimes_{\alpha,\red} G\).
Thus our next goal is to prove  conditions \eqref{cond:C3'}
and~\eqref{cond:C4} for the inclusion \(A\to \Mult(B)\).
Here we use Imai--Takai
Duality~\cite{Imai-Takai:Duality}.
This allows us to translate conditions \eqref{cond:C3'}
and~\eqref{cond:C4} for \(A\to \Mult(B)\) to conditions
\eqref{cond:C1} and~\eqref{cond:C2} for
\(B\to \Mult(B\rtimes_{\widehat{\alpha}} \hat{G})\),
which are easy to check.

Recall that \(B=A\rtimes_{\alpha,\red} G\)
is equipped with a (reduced) coaction
\(\widehat{\alpha}\colon B \to B \otimes \Cred(G)\),
where~\(\otimes\)
denotes the minimal \(\Cst\)\nb-tensor
product.  This coaction generates a crossed product
\(B\rtimes_{\widehat{\alpha}} \hat{G}\),
which comes with a morphism
\(B\to \Mult(B\rtimes_{\widehat{\alpha}} \hat{G})\),
and Imai--Takai Duality~\cite{Imai-Takai:Duality} identifies
\[
B\rtimes_{\widehat{\alpha}}  \hat{G} \cong A\otimes \Comp(L^2 (G)).
\]
Following
\cite{Sierakowski:IdealStructureCrossedProducts}*{Definition~1.5} we
call the action~\(\alpha\)
\emph{exact} if every \(I\in \Ideals^\alpha(A)\)
induces a short exact sequence
\[
0\to I\rtimes_{\alpha|_I,\red} G \to A\rtimes_{\alpha,\red} G
\to A/I\rtimes_{\alpha|_{A/I},\red} G \to 0.
\]
If~\(G\)
is an exact group, then all actions of~\(G\) are exact.  Any
action~\(\alpha\)
with \(A\rtimes_{\alpha,\red} G=A\rtimes_{\alpha} G\)
is exact.  So all actions of amenable groups are exact.

\begin{theorem}
  \label{thm:Imai-Takai_duality}
  Let~\(\alpha\) be an exact action.  Let
  \(B\defeq A\rtimes_{\alpha,\red} G\) and
  \(\Comp\defeq \Comp(L^2(G))\).  Build induction and restriction
  maps for the canonical \Star{}homomorphisms
  \[
    \varphi\colon A\to \Mult(B),\qquad
    \psi\colon B\to \Mult(B\rtimes_{\widehat{\alpha}} \hat{G})
    \cong \Mult(A\otimes \Comp).
  \]
  Then \(\Ideals^A(B) = \Ideals^{A\otimes \Comp}(B)\) and the
  following diagram commutes:
  \begin{equation}
    \label{eq:diagram_imai_takai}
    \begin{tikzcd}
      \Ideals^A(B) \arrow[d, "r"', "\cong", shift right=.53em]
      \arrow[r, equal] &
      \Ideals^{A\otimes \Comp}(B) \arrow[d, "i"', "\cong", shift right=.53em] \\
      \Ideals^B(A) \arrow[r, "\otimes \Comp", "\cong"' ]
      \arrow[u, "i"',  shift right=.53em] &
      \Ideals^B(A\otimes \Comp). \arrow[u, "r"', shift right=.53em]
    \end{tikzcd}
  \end{equation}
  And \(A\to \Mult(B)\)
  satisfies \eqref{cond:C3'} and~\eqref{cond:C4}.  Thus, if
  \(\check{A}\)
  or~\(\check{B}\)
  is second countable or, more generally, \(\Prime(\Ideals^\alpha(A))\)
  is first countable, then the quasi-orbit map
  \(\varrho\colon \check{B} \to \check{A}/{\sim}\)
  is open and surjective,  the quasi-orbit space
  \(\check{B}/{\sim}\)
  for \(B\to \Mult(B\rtimes_{\widehat{\alpha}}  \hat{G})\)
  exists, and there is a homeomorphism
  \(\check{A}/{\sim} \cong \check{B}/{\sim}\)
  such that the map
  \(\check{B} \xrightarrow{\varrho} \check{A}/{\sim}\cong
  \check{B}/{\sim}\)
  is the quotient map \(\check{B} \to \check{B}/{\sim}\).
\end{theorem}

\begin{remark}
  It is known that a diagram similar to~\eqref{eq:diagram_imai_takai}
  commutes when~\(i\)
  is replaced by the map \(\operatorname{Ind}\)
  given by kernels of the corresponding induced representations (see
  \cite{GootmanLazar.DualityCrossedProduct}*{Remarks 2.8} or
  \cite{Nilsen:DualityCrossedProducts}*{Propositions 2.7, 2.8}).
  Thus we need to show that
  \(\operatorname{Ind}\) and~\(i\) coincide on restricted ideals.
  It is readily seen that \(\operatorname{Ind}\)
  and~\(i\)
  coincide on \(\Ideals^B(A)=\Ideals^\alpha(A)\) if and only if the
  action~\(\alpha\) is exact.
  Hence this assumption is necessary for our proof to work. That
 \(\operatorname{Ind}\)
  restricted to \(\Ideals^{B\rtimes_{\widehat{\alpha}} \hat{G}}(B)\)
  coincides with~\(i\) is proved in
  \cite{GootmanLazar.DualityCrossedProduct}*{Proposition 3.14(iii)}
  if~\(G\) amenable (see also
  \cite{Nilsen:DualityCrossedProducts}*{Proposition 3.1(iii)}, where
  the full crossed products are considered).
\end{remark}

The proof of Theorem~\ref{thm:Imai-Takai_duality} is based on the
following two lemmas.

\begin{lemma}
  \label{lem:induced_in_double_crossed}
  An ideal in \(A \otimes \Comp\)
  is induced from~\(B\)
  if and only if it is of the form \(I\otimes \Comp\)
  for an invariant ideal \(I\in \Ideals^\alpha(A)\),
  if and only if it is induced from an ideal in~\(A\)
  along \(\psi\circ\varphi\colon A\to \Mult(A \otimes \Comp)\).
  \textup{(}This holds for any action \(\alpha\).\textup{)}
\end{lemma}

\begin{proof}
  Let \(I\in \Ideals^\alpha(A)\).
  The ideal in~\(B\)
  induced by~\(I\)
  is \(I\rtimes_{\lambda} G\)
  (see Proposition~\ref{pro:restricted_group_crossed}).
  Inducing further to an ideal in \(A \otimes \Comp\)
  gives the double crossed product ideal
  \((I\rtimes_\lambda G) \rtimes \hat{G} \cong I\otimes \Comp\).
  So ideals of this form are induced.
  We claim that any ideal in \(A \otimes \Comp\)
  induced from~\(A\)
  along \(\psi\circ\varphi\)
  is of this form.  Indeed, if \(I\idealin A\)
  is arbitrary, then induction along \(\psi\circ\varphi\)
  has the same effect as first inducing along~\(\varphi\)
  and then along~\(\psi\).
  When we induce along~\(\varphi\),
  we get \(i(I) = i(\bar{I})\),
  where \(\bar{I}\idealin A\)
  is the \(\alpha\)\nb-invariant
  ideal generated by~\(I\).
  So \(I\)
  and~\(\bar{I}\)
  induce the same ideal also along \(\psi\circ\varphi\).

  Now let \(J\idealin B\)
  be any ideal.  It remains to prove that the induced ideal in
  \(A \otimes \Comp\)
  is of the form \(I\otimes \Comp\)
  for an \(\alpha\)\nb-invariant ideal \(I\idealin A\).
  The \Star{}homomorphism~\(\psi\)
  first maps~\(J\)
  to its image in~\(A\rtimes_\lambda G\).
  The coaction crossed product
  \(C \defeq B \rtimes \hat{G}\)
  comes with non-degenerate, injective \Star{}homomorphisms
  \(B \to \Mult(C) \leftarrow \Cont_0(G)\)
  such that
  \[
  B \cdot \Cont_0(G)
  = \Cont_0(G) \cdot  B
  = C.
  \]
  And there is a canonical dual action \(\gamma\colon G\to\Aut(C)\).
  The Imai--Takai isomorphism is \(G\)\nb-equivariant,
  that is, it intertwines the action~\(\gamma\)
  and the action \(\alpha\otimes \Ad(\lambda_g)\)
  on \(A\otimes \Comp=A\otimes \Comp(L^2 (G))\).
  So the ideal in~\(C\)
  induced by \(J\idealin A\rtimes_\lambda G\)
  has the desired form if and only if it is \(\gamma\)\nb-invariant.
  The dual action~\(\gamma\)
  is built as follows.  If \(x\in A\rtimes_\lambda G\),
  \(f\in \Cont_0(G)\),
  then \(\gamma_g(x\cdot f) = x\cdot \lambda_g(f)\)
  with the automorphism \(\lambda_g\in\Aut(G)\)
  defined by the left regular representation,
  \((\lambda_g f)(x) = f(g^{-1} x)\).
  The ideal in~\(C\) induced by~\(J\) is
  \[
  C \cdot J \cdot C
  = \Cont_0(G) \cdot  B\cdot J\cdot B \cdot \Cont_0(G)
  = \Cont_0(G) \cdot J \cdot \Cont_0(G).
  \]
  If \(f_1,f_2\in\Cont_0(G)\),
  \(x\in J\),
  then
  \(\gamma_g(f_1 \cdot x \cdot f_2) = \lambda_g(f_1) \cdot x \cdot
  \lambda_g(f_2)\)
  again belongs to \(C\cdot J \cdot C\).
  Thus \(C \cdot J \cdot C\) is invariant.
\end{proof}

\begin{lemma}
  \label{pro:induced_ideals_crossed}
  Suppose that the action~\(\alpha\) is exact.  An ideal
  in \(B= A\rtimes_\lambda G\)
  is induced from~\(A\)
  if and only if it is restricted from \(A\otimes \Comp\).
  And the generalised inclusion
  \(\psi\colon B\to \Mult(A\otimes \Comp)\)
  is symmetric.
\end{lemma}

\begin{proof}
  Let \(J\idealin B\)
  be restricted from \(A\otimes \Comp\).
  Then it is the restriction of an induced ideal by
  Proposition~\ref{pro:restrict_induce_ideals}.\ref{pro:restrict_induce_ideals2b}.
  By Lemma~\ref{lem:induced_in_double_crossed}, it is the restriction of
  \(I\otimes \Comp\)
  for an invariant ideal \(I\in\Ideals^\alpha(A)\).
  This is equal to the kernel of the canonical \Star{}homomorphism
  \[
  B=A\rtimes_\lambda G
  \to \Mult(A\otimes \Comp)
  \to \Mult\bigl(A\otimes \Comp\bigm/ I\otimes \Comp\bigr).
  \]
  This factors through the canonical injective \Star{}homomorphism
  \[
  (A/I)\rtimes_\lambda G \hookrightarrow
  \Mult\bigl((A/I)\otimes \Comp\bigr)
  = \Mult\bigl(A\otimes \Comp\bigm/ I\otimes \Comp\bigr).
  \]
  So the restricted ideal from \(I\otimes \Comp\)
  is the kernel of the \Star{}homomorphism
  \(A\rtimes_\lambda G \to (A/I)\rtimes_\lambda G\).
  Since we assume that the action~is exact, this is equal to
  \(I\rtimes_\lambda G\).
  The ideals of this form for \(I\in\Ideals^\alpha(A)\)
  are exactly the induced ideals by
  Proposition~\ref{pro:restricted_group_crossed}.  Thus an ideal
  in~\(A\rtimes_\lambda G\)
  is induced from~\(A\)
  if and only if it is restricted from \(A\otimes \Comp\).
  It follows that the generalised inclusion~\(\psi\)
  is symmetric because
  \((I\rtimes_\lambda G) \cdot (A\otimes \Comp)
    = I \otimes \Comp
    = (A\otimes \Comp) \cdot (I\rtimes_\lambda G)\).
\end{proof}

\begin{proof}[of Theorem~\textup{\ref{thm:Imai-Takai_duality}}]
  Lemmas \ref{lem:induced_in_double_crossed}
  and~\ref{pro:induced_ideals_crossed} show that the
  diagram~\eqref{eq:diagram_imai_takai} commutes.
  Hence the restricted ideals for~\(\psi\)
  are the same as the induced ideals for~\(\varphi\).
  So conditions \eqref{cond:C3'} and~\eqref{cond:C4} for~\(\varphi\)
  are equivalent to conditions \eqref{cond:C1} and~\eqref{cond:C2}
  for~\(\psi\).
  Moreover, \(\psi\) satisfies condition~\eqref{cond:C1} because
  joins of induced ideals remain induced by
  Proposition~\ref{pro:restrict_induce_ideals}.\ref{pro:restrict_induce_ideals6},
  and~\(\psi\)
  satisfies condition~\eqref{cond:C2} because~\(\psi\)
  is symmetric (see Lemma~\ref{pro:induced_ideals_crossed} and
  Corollary~\ref{cor:symmetric_nice}).

  Accordingly, the second parts of both Theorems
  \ref{the:prime_primitive} and~\ref{the:prime_primitive2}
  apply to~\(\varphi\).
  So~\(\varrho\)
  is open and surjective
  if \(\Prime(\Ideals^\alpha(A))=\Prime^B(A)\)
  is first countable.  And this follows if \(\check{A}\)
  or~\(\check{B}\)
  is second countable.
  By Lemma~\ref{lem:induced_map_pi}, the quasi-orbit
  space~\(\check{B}/{\sim}\)
  for~\(\psi\)
  exists, and \(\prid_1\sim \prid_2\)
  for \(\prid_1,\prid_2\in \check{B}\)
  if and only if the largest restricted-along-\(\psi\) ideals
  contained in \(\prid_1\) and~\(\prid_2\) coincide.
  By~\eqref{eq:diagram_imai_takai},
  this holds if and only if the largest induced-along-\(\varphi\)
  ideals contained in \(\prid_1\) and~\(\prid_2\) coincide.
  Thus
  \[
  \varrho(\prid_1) = \varrho( \prid_2)
  \iff
  r(\prid_1)=r(\prid_2)
  \iff
  (i\circ r)(\prid_1)=(i\circ r)(\prid_2)
  \iff
  \prid_1 \sim \prid_2.
  \]
  So the continuous open surjection
  \(\varrho\colon \check{B}\to \check{A}/{\sim}\)
  induces a homeomorphism
 \(\check{B}/{\sim}\cong \check{A}/{\sim}\).
\end{proof}

\subsection{Fell bundles over locally compact groups}
\label{sec:Fell_group}

Fell bundles over a locally compact group~\(G\),
introduced in~\cite{Fell:Extension_Mackey}, are the most general kinds
of ``actions'' of~\(G\)
on \(\Cst\)\nb-algebras.
These contain twisted partial actions
(see~\cite{Exel:TwistedPartialActions}) as a special case.  We may
even allow measurable twists for global actions
by~\cite{Exel-Laca:Continuous_Fell}, but this result seems not to have
been extended to partial actions with a measurable twist yet.  Twisted
partial actions contain both partial actions and twisted actions and
thus ordinary group actions by automorphisms.  The full and reduced
crossed products for a (twisted, partial) action are naturally
isomorphic to the full and reduced section \(\Cst\)\nb-algebras
of the corresponding Fell bundle.

We now  explain how to generalise our results above to a Fell
bundle \(\A = (A_g)_{g\in G}\)
over~\(G\);
it comes with multiplication maps \(A_g \times A_h \to A_{g h}\)
for \(g,h\in G\),
involutions \(A_g \to A_{g^{-1}}\)
for \(g\in G\)
with certain properties, and a topology on \(\bigsqcup_{g\in G} A_g\)
(see \cite{Exel:TwistedPartialActions}*{Definitions 2.2 and~3.9} or
\cites{Doran-Fell:Representations, Doran-Fell:Representations_2}).  In
particular, \(A\defeq A_e\)
is a \(\Cst\)\nb-algebra
and each~\(A_g\)
becomes a Hilbert \(A\)\nb-bimodule.
The set \(\Contc(G,\A)\) of continuous, compactly supported
sections of the bundle \(\A=(A_g)_{g\in G}\)
over~\(G\)
carries a \Star{}algebra structure.  The full section
\(\Cst\)\nb-algebra \(\Cst(\A)\)
of~\(\A\)
is defined as a completion of \(\Contc(G,\A)\)
in the maximal continuous \(\Cst\)\nb-norm.  For locally compact~\(G\),
the reduced section \(\Cst\)\nb-algebra
\(\Cred(\A)\)
is defined in~\cite{ExelNg:ApproximationProperty}
as the range of a regular representation
\(\lambda\colon \Cst(\A) \to
\Bound(L^2_{e}(\A))\),
where \(L^2_{e}(\A)\)
is a canonical Hilbert \(A_e\)\nb-module.
A \emph{cross section \(\Cst\)\nb-algebra} is any
\(\Cst\)\nb-algebra~\(B\)
with surjective \Star{}homomorphisms
\(\Cst(\A) \onto B \onto \Cred(\A)\)
whose composition is the regular representation
\(\lambda\colon \Cst(\A) \onto \Cred(\A)\).

An ideal \(I\idealin A\)
is called \emph{\(\A\)\nb-invariant}
if it is invariant for every Hilbert bimodule~\(A_g\),
\(g\in G\),
that is, if \(I\cdot A_g = A_g\cdot I\)
for all \(g\in G\).
This is equivalent to \(A_g \cdot I\cdot A_{g^{-1}} \subseteq I\)
for all \(g\in G\).
Let~\(I\)
be \(\A\)\nb-invariant.
The Fell bundle structure
on~\((A_g)_{g\in G}\)
restricts to one on \(I_g \defeq I\cdot A_g = A_g\cdot I\).
And it induces a Fell bundle structure on the quotients
\((A_g/I_g)_{g\in G}\)
as well.  We denote these induced Fell bundles  by \(\A|_I\)
and~\(\A|_{A/I}\), respectively.

\begin{proposition}
  \label{pro:restricted_group_Fell}
  Let \(\varphi\colon A\to \Mult(B)\)
  be the canonical generalised \(\Cst\)\nb-inclusion
  of~\(A\)
  into a cross section algebra~\(B\)
  of a Fell bundle \(\A = (A_g)_{g\in G}\).
  \begin{enumerate}
  \item \label{pro:restricted_group_bundle1}%
    \(\Ideals^B(A)\)
    is the subset \(\Ideals^\A(A)\) of
    \(\A\)\nb-invariant ideals in~\(A\);
  \item \label{pro:restricted_group_bundle2}%
    an ideal \(J\idealin B\)
    is induced if and only if it is the image of
    \(\Cst(\A|_I)\)
    in~\(B\) for some \(\alpha\)\nb-invariant ideal \(I\idealin A\).
  \end{enumerate}
  The \(\Cst\)\nb-inclusion
  \(\varphi\colon A\to \Mult(B)\)
  is symmetric, and \eqref{cond:C1}, \eqref{cond:C2} and \eqref{cond:C3} hold.
\end{proposition}

\begin{proof}
  First we show that restricted ideals are \(\A\)\nb-invariant.
  To this end, let \(J\in\Ideals(B)\) and put \(I\defeq r(J)\).
  We use the canonical maps from the spaces~\(A_g\)
  to the multiplier algebra of~\(\Cst(\A)\).
  First, \(x\in A_g\)
  defines a multiplier on the \Star{}algebra \(\Contc(G,\A)\)
  by \((x\cdot f)(h) = x\cdot (f(g^{-1} h))\)
  and \((f\cdot x)(h) = (f(h g^{-1}))\cdot x\)
  for \(h\in G\) and \(f\in \Contc(G,\A)\).
  This extends to a multiplier on the \(\Cst\)\nb-completion (see
  \cite{ExelNg:ApproximationProperty}*{Lemma 1.1}).
  The maps \(A_g \to \Mult(\Cst(\A)) \to \Mult(B)\)
  form a Fell bundle representation.  Since \(\Mult(B,J)\)
  is an ideal in~\(\Mult(B)\),
  an element \(x\in A_g\)
  belongs to \(\Mult(B,J)\)
  if and only if \(x^* x \in \Mult(B,J)\)
  if and only if \(x x^* \in \Mult(B,J)\).
  Now \(x^* x\)
  and~\(x x^*\)
  belong to \(A_g^* \cdot A_g = A_{g^{-1}} \cdot A_g \subseteq A_e\)
  and \(A_g \cdot A_g^* = A_g \cdot A_{g^{-1}} \subseteq A_e\).
  And an element of \(A_e=A\)
  is in \(\Mult(B,J)\)
  if and only if it belongs to~\(r(J)\).
  So \(x^* x\in r(J)\)
  if and only if \(x x^* \in r(J)\)
  for all \(x\in A_g\).
  During the proof of the Rieffel correspondence for the Hilbert
  bimodule~\(A_g\),
  it is shown that \(x x^* \in I\)
  if and only if \(x\in I\cdot A_g\)
  and that \(x^* x \in I\)
  if and only if \(x\in A_g\cdot I\).
  So \(I\cdot A_g = A_g\cdot I\)
  for all \(g\in G\).
  This finishes the proof that restricted ideals in~\(A\)
  are invariant.

  Now let~\(I\idealin A_e\)
  be \(\A\)\nb-invariant.
  We claim that~\(I\)
  is symmetric.  Indeed, the standard formula for the convolution
  in \(\Contc(G,\A)\)
  shows that \(I\cdot \Contc(G,\A)\)
  and \(\Contc(G,\A)\cdot I\)
  are contained in \(\Contc(G,\A|_I)\),
  the \Star{}algebra
  of continuous, compactly supported sections of~\(\A|_I\).
  Moreover, \(\Contc(G,\A|_I)\)
  is non-degenerate as a left or right \(I\)\nb-module.
  Since \(\Contc(G,\A)\) is a dense \Star{}subalgebra
  in \(\Cst(\A)\),
  we may view it also as a dense \Star{}subalgebra of~\(B\).
  Thus taking the closures in the crossed product~\(B\),
  we see that both \(I\cdot B\)
  and \(B\cdot I\)
  are equal to the closure of \(\Contc(G,\A|_I)\)
  in~\(B\).
  Thus \(I\cdot B = B\cdot I\),
  that is, \(I\)
  is symmetric.

  Next we show that any \(I\in\Ideals^\A(A)\)
  is restricted.  We need to prove that \(r\circ i(I) = I\).
  The inclusion \(r\circ i(I) \supseteq I\)
  is already contained in
  Proposition~\ref{pro:restrict_induce_ideals}.\ref{pro:restrict_induce_ideals2}.
  To see the reverse
  inclusion, note that the canonical map from \(\Cst(\A)\)
  to \(\Cred(\A|_{A/I})\)
  annihilates \(\Contc(G,\A|_I)\),
  which is dense in~\(i(I)\)
  by the proof above that~\(I\)
  is symmetric.  Therefore, every element \(a\in r\circ i(I)\)
  induces the zero multiplier on \(\Cred(\A|_{A/I})\).
  The canonical map from~\(A/I\)
  to the multiplier algebra of \(\Cred(\A|_{A/I})\)
  is injective.  So \(a\in r\circ i(I)\)
  is mapped to~\(0\) in~\(A/I\), that is, \(a\in I\).
  Hence \(r\circ i(I) \subseteq I\), and~\(I\) is restricted.

  This finishes the proof of~\ref{pro:restricted_group_bundle1}
  and shows that~\(\varphi\)
  is symmetric.  In particular, conditions \eqref{cond:C2}
  and~\eqref{cond:C3} follow from Corollary~\ref{cor:symmetric_nice}.
  It is readily seen that the closed linear span of a family of
  \(\A\)\nb-invariant
  ideals is again \(\A\)\nb-invariant.
  So \ref{pro:restricted_group_crossed1} implies
  condition~\eqref{cond:C1}.

  Any induced ideal in~\(B\)
  is obtained by inducing a restricted ideal because
  \(i(I) = i(r\circ i(I))\)
  for all \(I\in\Ideals(A)\).
  Restricted ideals are invariant.  The proof above that invariant
  ideals are symmetric also shows that the ideal in~\(B\)
  induced by an invariant ideal \(I\in\Ideals^\A(A)\)
  is the closure of the image of \(\Contc(G,\A|_I)\)
  in~\(B\).
  This closure of \(\Contc(G,\A|_I)\)
  in the full cross section algebra~\(\Cst(\A)\)
  is isomorphic to \(\Cst(\A|_I)\).
  Thus~\(i(I)\)
  is the image of
  \(\Cst(\A|_I)\idealin \Cst(\A)\)
  in~\(B\).  This proves~\ref{pro:restricted_group_crossed2}.
\end{proof}

Let \(\A=(A_g)_{g\in G}\)
be a Fell bundle and recall that each~\(A_g\),
\(g\in G\), is a Hilbert \(A\)\nb-bimodule over \(A\defeq A_e\).
By the Rieffel correspondence, \(A_g\)
induces a homeomorphism~\(\check{A}_g\)
between two open subsets of~\(\Prime(A)\),
which we view as a partial homeomorphism of~\(\check{A}\).
The associativity of the multiplication in~\(\A\)
implies that \((\check{A}_g)_{g\in G}\)
is a partial action of~\(G\)
on \(\check{A}\) or \(\Prime(A)\)
(see \cites{Kwasniewski-Szymanski:Pure_infinite,
  Abadie-Abadie:Ideals}).
This action is continuous by
\cite{Abadie-Buss-Ferraro:Morita_Fell}*{Proposition 5.5}.
The orbit of \(\prid \in \check{A}\)
under this action is the set~\(G\cdot\prid\)
of those \(\prid' \in \check{A}\)
which lie in the domain of~\(\check{A}_g\) for some \(g\in G\)
with \(\check{A}_g(\prid')=\prid\).

\begin{theorem}
  \label{the:Fell_bundle_group}
  Let \(\A = (A_g)_{g\in G}\)
  be a Fell bundle over a locally compact group~\(G\).
  Let \(\varphi\colon A\to \Mult(B)\)
  be the canonical generalised \(\Cst\)\nb-inclusion
  of \(A\defeq A_e\)
  into a cross section algebra~\(B\) of~\(\A\).
  The relation~\(\sim\)
  that defines the quasi-orbit space~\(\check{A}/{\sim}\)
  for~\(\varphi\)
  is \(\prid_1\sim \prid_2\)
  if and only if
  \(\overline{G\cdot \prid_1} = \overline{G\cdot \prid_2}\)
  as closed subsets of~\(\check{A}\).
  Assume \(\Prime(\Ideals^\A(A))\)
  to be first countable or~\(\check{A}\)
  to be second countable.  Then the quasi-orbit map exists.  More
  precisely:
  \begin{enumerate}
  \item The map \(\pi\colon \check{A}\to\Prime(\Ideals^\A(A))\)
    that sends~\(\prid\)
    to the largest \(\A\)\nb-invariant
    ideal contained in~\(\prid\)
    is continuous, open and surjective.  It descends to a
    homeomorphism
    \(\widetilde{\pi}\colon
    \check{A}/{\sim}\to\Prime(\Ideals^\A(A))\).
    The quasi-orbit space~\(\check{A}/{\sim}\)
    is a quotient of~\(\check{A}\) by an open equivalence relation.
  \item There is a continuous map
    \(\varrho\colon \check{B} \to \check{A}/{\sim}\),
    \(\prid\mapsto \widetilde{\pi}^{-1}(r(\prid))\).
    It identifies \(\check{B}\cong \check{A}/{\sim}\)
    if and only if~\(A\)
    separates ideals in~\(B\).
    And then \(B=A\rtimes_{\alpha,\red} G\).
  \end{enumerate}
\end{theorem}

\begin{proof}
  Let \(\prid \in \check{A}\).
  Let~\(\pi(\prid)\)
  be the largest \(\A\)\nb-invariant ideal contained in~\(\prid\).
  The open subset of~\(\Prime(A)\)
  corresponding to~\(\pi(\prid)\)
  is the largest invariant open subset that does not
  contain~\(\prid\).
  So it is the complement of the closure of the orbit of~\(\prid\).
  Thus \(\prid_1 \sim \prid_2\)
  for \(\prid_1, \prid_2\in\Prime(A)\)
  if and only if
  \(\overline{G\cdot \prid_1} = \overline{G\cdot \prid_2}\).
  The remaining assertions follow from
  Proposition~\ref{pro:restricted_group_Fell},
  Theorem~\ref{the:prime_primitive}, and
  Corollary~\ref{cor:primitive_ideal_space_description}.
\end{proof}

A Morita enveloping action for a Fell bundle is an ordinary group
action that is equivalent to the Fell bundle in a suitable sense.
Then the full and reduced \(\Cst\)\nb-algebras are Morita equivalent,
respectively.  For partial actions, Morita enveloping actions have
been studied by Abadie~\cite{Abadie:Enveloping}.  For Fell bundles,
they are built by Abadie, Buss and Ferraro
\cites{Abadie-Ferraro:Equivalence_of_Fell_Bundles,
  Abadie-Buss-Ferraro:Morita_Fell}.  They may be used for another
proof of Theorem~\ref{the:Fell_bundle_group}, by showing that the
properties that we are interested in are preserved by Morita
globalisations.  We use this technique to generalise
Theorem~\ref{thm:Imai-Takai_duality}.  We call a Fell bundle
\(\A=(A_g)_{g\in G}\) \emph{exact} if for every \(I\in \Ideals^\A(A)\)
the induced sequence
\begin{equation}
  \label{eq:exact_Fell_bundles}
  0\to \Cred(\A|_I) \to \Cred(\A)
  \to \Cred(\A|_{A/I})\to 0
\end{equation}
is exact.  For instance, Fell bundles with the approximation
property introduced in~\cite{ExelNg:ApproximationProperty} are
exact.

\begin{proposition}
  \label{pro:equivalent_Fell_bundles}
  Let \(\A = (A_g)_{g\in G}\)
  and \(\B = (B_g)_{g\in G}\)
  be Fell bundles that are \textup{(}weakly\textup{)} equivalent as
  in \cite{Abadie-Ferraro:Equivalence_of_Fell_Bundles}*{Definition 2.6}
  or \cite{Abadie-Buss-Ferraro:Morita_Fell}*{Definition 2.6}.
  \begin{enumerate}
  \item \label{pro:equivalent_Fell_bundles1}%
    The \(\Cst\)\nb-algebras \(\Cred(\A)\) and~\(\Cred(\B)\) are Morita
    equivalent.
  \item \label{pro:equivalent_Fell_bundles2}%
    The induced Rieffel correspondence
    \(\Ideals(\Cred(\A))\cong \Ideals(\Cred(\B))\)
    restricts to an isomorphism
    \(\Ideals^{A_e}(\Cred(\A) )\cong \Ideals^{B_e}(\Cred(\B))\)
    between the lattices of \(G\)\nb-graded ideals.

  \item \label{pro:equivalent_Fell_bundles3}%
    The inclusion \(A_e\subseteq \Mult(\Cred(\A))\)
    satisfies conditions \eqref{cond:C3'} and~\eqref{cond:C4} if and
    only if \(B_e\subseteq \Mult(\Cred(\B))\) satisfies them.
  \item \label{pro:equivalent_Fell_bundles4}%
    \(\A\) is exact if and only if~\(\B\) is exact.
  \end{enumerate}
\end{proposition}

\begin{proof}
  Let \(\mathcal{X}=(X_g)_{g\in G}\) be a Hilbert
  \(\A\)\nb-\(\B\)-bundle.
  By \cite{Abadie-Ferraro:Equivalence_of_Fell_Bundles}*{Proposition 4.13},
  it yields a \(\Cred(\A)\)-\(\Cred(\B)\)-equivalence bimodule
  \(\Cred(\mathcal{X})\).
  This gives~\ref{pro:equivalent_Fell_bundles1}.
  Statement~\ref{pro:equivalent_Fell_bundles2}  follows in essence
  from \cite{Abadie-Buss-Ferraro:Morita_Fell}*{Corollary 4.3}, as
  \(G\)\nb-graded ideals are invariant with respect to dual
  coactions.  In fact, \ref{pro:equivalent_Fell_bundles2}  can be
  checked directly as follows.
  \cite{Abadie-Ferraro:Equivalence_of_Fell_Bundles}*{Lemma 2.7(6)}
  implies
  \(B_g=\overline{\operatorname{span}}
  \setgiven{\braket{X_r}{X_{rg}}_{\B}}{r\in G}\)
  for all \(g\in G\).
  Let~\(I\) be an \(\A\)\nb-invariant ideal in~\(A_e\).  Let
  \begin{equation}
    \label{eq:define_J_from_I}
    J\defeq \overline{\operatorname{span}}
    \setgiven{\braket{X_r}{IX_{r}}_{\B}} {r\in G}
  \end{equation}
  Let \(r,g \in G\).
  The basic properties of Hilbert bundles imply
  \begin{align*}
    B_{g}\cdot \braket{X_r}{IX_{r}}_{\B} \cdot B_{g^{-1}}
    &= \braket{X_r B_{g^{-1}}}{IX_{r}B_{g^{-1}}}_{\B}
    \subseteq \braket{X_{rg^{-1}}}{IX_{rg^{-1}}}_{\B}
    \subseteq J.
  \end{align*}
  This implies that \(J\) is a \(\B\)\nb-invariant ideal in~\(B_e\).
  Since~\(I\)
  is \(\A\)\nb-invariant,
  we may recover~\(I\)
  from~\(J\).
  Indeed, for \(r,s \in G\),
  we get
  \begin{multline*}
    \braket{X_r}{IX_{r}}_{\B}\cdot  \braket{X_s }{X_{s}}_{\B}
    = \braket{X_r}{IX_{r} \braket{X_s }{X_{s}}_{\B} }_{\B}
    = \braket{X_r}{I \cdot {{}_\A\braket{X_{r} }{X_{s}}}X_{s}}_{\B}
    \\ \subseteq \braket{X_r }{I A_{rs^{-1}}X_{s}}_{\B}
    = \braket{X_r }{ A_{rs^{-1}}IX_{s}}_{\B}
    = \braket{A_{sr^{-1}}X_r }{ IX_{s}}_{\B}
    \subseteq \braket{X_s }{IX_{s}}_{\B}.
  \end{multline*}
  Hence \(J\cdot \braket{X_s}{X_s}_{\B} = \braket{X_s}{ IX_s}_{\B}\).
  By the symmetry of our assumptions, this implies
  \(I\cdot \braket{X_s}{X_s}_{\A} = \braket{X_s J}{X_s}_{\A}\).
  Then
  \(I=\overline{\operatorname{span}} \setgiven{{}_\A\braket{X_s
      J}{X_s}} {s\in G}\).  Let \(r,s\in G\).  Then
  \begin{multline*}
    \braket{X_r }{IX_{s}}_{\B}
    = \braket{X_r }{I \cdot {}_\A\braket{X_s}{X_{s}} X_{s}}_{\B}
    = \braket{X_r }{{}_\A\braket{X_s\cdot  J}{X_{s}} X_{s}}_{\B}
    \\= \braket{X_r }{X_s\cdot  J \braket{X_{s}}{X_{s}}_{\B} }_{\B}
    = \braket{X_r }{X_s }_{\B}  \cdot  J \braket{X_{s}}{X_{s}}_{\B}
    \subseteq B_{rs^{-1}} J.
  \end{multline*}
  The latter implies that the Rieffel correspondence
  \(R\colon\Ideals(\Cred(\A) )\congto \Ideals(\Cred(\B))\)
  maps \(\Cred(\A|_I)=i(I)\)
  to an ideal \(R(i(I))\)
  contained in  \(\Cred(\B|_J)=i(J)\).
  By symmetry, we get \(R^{-1}(i(J))\subseteq i(I)\).
  Hence \(R(i(I))=i(J)\).
  This proves~\ref{pro:equivalent_Fell_bundles2}.

  Statement~\ref{pro:equivalent_Fell_bundles3} follows immediately
  from~\ref{pro:equivalent_Fell_bundles2} because conditions
  \eqref{cond:C3'} and~\eqref{cond:C4} are phrased in terms of induced
  ideals.

  To see~\ref{pro:equivalent_Fell_bundles4}, let \(\Cred(\A|_I)\)
  be the \(G\)\nb-graded
  ideal in~\(\Cred(\A)\)
  corresponding to an \(\A\)\nb-invariant ideal~\(I\).
  By~\ref{pro:equivalent_Fell_bundles2}, the Rieffel correspondence
  maps \(\Cred(\A|_I)\) to \(\Cred(\B|_J)\) for the
  \(\B\)\nb-invariant ideal~\(J\)
  defined in~\eqref{eq:define_J_from_I}.
  The Hilbert bundle~\(\mathcal{X}\)
  between \(\A\) and~\(\B\)
  restricts to a Hilbert bundle~\(\mathcal{X}|_I\)
  between \(\A|_I\) and~\(\B|_J\)
  with \((\mathcal{X}|_I)_g = I\cdot X_g\).
  And it induces a Hilbert bundle~\(\mathcal{X}|_{A/I}\)
  between \(\A|_{A/I}\) and~\(\B|_{B/J}\)
  with \((\mathcal{X}|_{A/I})_g = X_g/I\cdot X_g\).
  The restricted Hilbert bundles induce Morita equivalences
  \(\Cred(\A|_I) \sim \Cred(\B|_J)\)
  and \(\Cred(\A|_{A/I}) \sim \Cred(\B|_{B/J})\).
  These are obtained by restricting the imprimitivity bimodule
  between \(\Cred(\A)\) and~\(\Cred(\B)\).
  The sequence~\eqref{eq:exact_Fell_bundles} is exact if and only if
  the primitive ideal space of \(\Cred(\A)\)
  is the union of the primitive ideal spaces of \(\Cred(\A|_I)\)
  and \(\Cred(\A|_{A/I})\).
  Since Morita equivalence implies an isomorphism between the ideal
  lattices, the Morita equivalence transfers this property
  from~\(\A\) to~\(\B\).
  This proves~\ref{pro:equivalent_Fell_bundles4}.
\end{proof}

\begin{theorem}
  \label{thm:open_orbit_map_bundles}
  Suppose that \(\A = (A_g)_{g\in G}\)
  is an exact Fell bundle over a locally compact group~\(G\).
  Let \(\varphi\colon A\to \Mult(B)\)
  be the canonical generalised \(\Cst\)\nb-inclusion
  of \(A\defeq A_e\)
  into the reduced cross section algebra \(B\defeq \Cred(\A)\) of~\(\A\).
  If \(\check{A}\) or~\(\check{B}\) is second countable or, more
  generally, \(\Prime(\Ideals^\A(A))\)
  is first countable, then the quasi-orbit map
  \(\varrho\colon \check{B} \to \check{A}/{\sim}\)
  is open and surjective.
\end{theorem}

\begin{proof}
  By \cite{Abadie-Buss-Ferraro:Morita_Fell}*{Theorem 3.4}, \(\A\)
  is equivalent to a Fell bundle~\(\B\) associated to an
  action \(\gamma\colon C\to \Aut(C)\)
  on a \(\Cst\)\nb-algebra~\(C\).  The Fell bundle~\(\B\) is exact
  by
  Proposition~\ref{pro:equivalent_Fell_bundles}.\ref{pro:equivalent_Fell_bundles4}.
  Hence the inclusion \(C\to \Mult(C\rtimes_{\alpha,\red} G)\)
  satisfies \eqref{cond:C3'} and~\eqref{cond:C4} by
  Theorem~\ref{thm:Imai-Takai_duality}.
  Thus \(A_e\to \Mult(\Cred(\A))\) satisfies \eqref{cond:C3'}
  and~\eqref{cond:C4} by
  Proposition~\ref{pro:equivalent_Fell_bundles}.\ref{pro:equivalent_Fell_bundles3}.
  Hence the claims follow from Corollary~\ref{cor:prime_primitive2}.
  Here we may still describe the relation~\(\sim\)
  as in Theorem~\ref{the:Fell_bundle_group}.
\end{proof}

\subsection{Regular inclusions and
  \texorpdfstring{$\Cst$}{C*}-algebras graded by inverse semigroups}
\label{sec:regular_inc}

The interest in regular \(\Cst\)\nb-inclusions
started with the study of Cartan \(\Cst\)\nb-subalgebras
(see \cites{Renault:Groupoid_Cstar, Kumjian:Diagonals,
  Renault:Cartan.Subalgebras, Exel:noncomm.cartan}).  They model a large
class of examples, including crossed products for topologically free
actions of discrete groups, inverse semigroups, or étale groupoids.
In order to apply our programme to these inclusions, we first
introduce and discuss \(\Cst\)\nb-algebras
graded by inverse semigroups.  Then we translate the corresponding
results to regular inclusions by showing that they are naturally
graded by certain inverse semigroups.

\begin{definition}
  \label{exa:regular_inclusion}
  Let~\(S\)
  be an inverse semigroup with unit element \(e\in S\).
  An \emph{\(S\)\nb-graded
    \(\Cst\)\nb-algebra}
  is a \(\Cst\)\nb-algebra~\(B\)
  with a family of closed linear subspaces \((B_t)_{t\in S}\)
  such that \(B_g^* = B_{g^*}\),
  \(B_g\cdot B_h \subseteq B_{g h}\)
  for all \(g,h\in S\)
  and \(B_g \subseteq B_h\)
  if \(g \le h\)
  in~\(S\),
  and \(\sum B_t\)
  is dense in~\(B\).
  We call \(A\defeq B_e\subseteq B\)
  the \emph{unit fibre} of the \(S\)\nb-grading.
\end{definition}

\begin{example}
  \label{exa:regular_group}
  Any discrete group \(G\)
  may be viewed as an inverse semigroup with \(g^* \defeq g^{-1}\)
  for all \(g\in S\).
  Then \(g\le h\)
  for \(g,h\in G\)
  only happens for \(g=h\).
  Thus in this case our notion of a \(G\)\nb-graded
  \(\Cst\)\nb-algebra
  reduces to the standard one (see
  \cite{Exel:Partial_dynamical}*{Definition~16.2}).  Crossed products
  for (partial or twisted) \(G\)\nb-actions
  have an obvious \(G\)\nb-grading, and so do the section
  \(\Cst\)\nb-algebras
  of Fell bundles over~\(G\).
  Here we may complete the \Star{}algebra of sections of a Fell bundle
  in \emph{any} \(\Cst\)\nb-seminorm
  for which the map from the unit fibre to the Hausdorff completion
  remains injective.
\end{example}

\begin{example}
  \label{exa:regular_groupoid}
  Let~\(G\)
  be an \'etale groupoid with locally compact Hausdorff space of
  units~\(G^0\).
  A subset \(U\subseteq G\)
  is called a \emph{bisection} if it is open and the restrictions of
  the source and range maps to~\(U\)
  are injective.  Bisections of \(G\),
  with multiplication and inverse inherited from~\(G\),
  form an inverse semigroup~\(\SS(G)\).
  Let \(\A=(A_\gamma)_{\gamma\in G}\)
  be a Fell bundle over~\(G\)
  as in
  \cite{BussExel:Fell.Bundle.and.Twisted.Groupoids}*{Definition~2.6}.
  The section \(\Cst\)\nb-algebra
  \(B=\Cst(\A)\)
  of~\(\A\)
  is the completion of a certain convolution \Star{}algebra~\(B_0\).
  For \(U\in \SS(G)\),
  let~\(A_U\)
  be the space of \(\Cont_0\)\nb-sections of~\(\A\)
  on~\(U\).
  Extend functions in~\(A_U\)
  by~\(0\) outside~\(0\) and let~\(B_0\)
  be the linear span of~\(A_U\) for all
  \(U\in \SS(G)\).
  This is a dense \Star{}subalgebra of~\(B\).
  The subspaces~\(A_U\) are closed in~\(B\) because the
  \(\Cst\)\nb-norm
  on~\(B\)
  restricts to the supremum norm on~\(A_U\).
  And they satisfy \(A_U^* = A_{U^*}\),
  \(A_U\cdot A_V \subseteq A_{UV}\)
  for all \(U,V\in \SS(G)\)
  and \(A_U \subseteq A_V\)
  for \(U \subseteq V\)
  (see
  \cite{BussExel:Fell.Bundle.and.Twisted.Groupoids}*{Example~2.11}).
  Thus \((A_U)_{U\in \SS(G)}\)
  forms a grading of~\(B\)
  whose unit fibre~\(A\)
  is the \(\Cont_0(G^0)\)\nb-algebra~\(A_{G^0}\)
  of sections of the bundle \((A_x)_{x\in G^0}\).
  In fact, \(B\)
  is graded by \((A_U)_{U\in S}\)
  for any (unital) inverse subsemigroup \(S\subseteq \SS(G)\)
  such that the sets in~\(S\)
  form a basis of the topology in~\(G\).
  As in Example~\ref{exa:regular_group}, the choice of the
  \(\Cst\)\nb-seminorm
  on~\(B_0\)
  is immaterial: if~\(B\)
  is any quotient of the full section \(\Cst\)\nb-algebra,
  such that the unit fibre \(A=A_{G^0}\hookrightarrow \Cst(\A)\)
  embeds into~\(B\)
  under the quotient homomorphism, then~\(B\) is
  graded by the images of~\(A_U\)
  for \(U\in S\).
\end{example}

The conditions \(B_g^* = B_{g^*}\),
\(B_g\cdot B_h \subseteq B_{g h}\)
and \(B_g \subseteq B_h\)
for \(g \le h\)
in Definition~\ref{exa:regular_inclusion} say that the spaces~\(B_g\)
with the multiplication, involution, and norm inherited from~\(B\)
form a \emph{Fell bundle over}~\(S\)
as in \cite{Exel:noncomm.cartan}*{Definition~2.1}.  These conditions
imply \(B_g B_g^* = B_g B_{g^*} = B_{g g^*} \subseteq B_e=A\)
because \(g g^* \le e\).
Since \(A B_g+ B_g A\subseteq B_g\),
we see that~\(B_g\)
for \(g\in S\)
is naturally a Hilbert \(A\)\nb-bimodule with the inner products
\[
\BRAKET{a}{b} \defeq a\cdot b^* \in A
\quad \text{ and  } \quad
\braket{a}{b}\defeq a^* \cdot b\in A\quad
\text{ for } a,b\in B_g.
\]
The Fell bundle \((B_t)_{t\in S}\)
is \emph{saturated} if \(B_g\cdot B_h = B_{g h}\)
for every \(g,h\in S\).
Then \((B_t)_{t\in S}\)
forms an action of~\(S\)
on~\(A\)
by Hilbert bimodules as in
\cite{Buss-Meyer:Actions_groupoids}*{Definition~4.7}.

\begin{definition}
  Let \(\Bgr=(B_t)_{t\in S}\)
  be an \(S\)\nb-grading
  of a \(\Cst\)\nb-algebra~\(B\).
  We say that an ideal \(I\in \Ideals(A)\)
  is \emph{\(\Bgr\)\nb-invariant}
  if \(B_g I B_g^*\subseteq I\)
  for all \(g\in S\)
  (see \cites{Kwasniewski-Szymanski:Pure_infinite,
    Abadie-Abadie:Ideals}).  Let \(\Ideals^\B(A)\)
  denote the set of \(\Bgr\)\nb-invariant ideals.
\end{definition}

\begin{proposition}
  \label{prop:invariant_ideals_in_graded}
  Let \(A\defeq B_e\subseteq B\)
  be the unit fibre of an \(S\)\nb-grading
  \((B_t)_{t\in S}\).
  For every \(I\in \Ideals(A)\), the following are equivalent:
  \begin{enumerate}
  \item \label{enu:invariant_ideals_in_graded_invariant}%
    \(I\) is restricted;
  \item \label{enu:invariant_ideals_in_graded1}%
    \(I\) is \(\Bgr\)-invariant;
  \item \label{enu:invariant_ideals_in_graded3}%
    \(I\)
    is \(B_g\)-invariant for all \(g\in S\), that is, \(I B_g=B_g I\)
    for all \(g\in S\);
  \item \label{enu:invariant_ideals_in_graded4}%
    \(\overline{\sum_{g\in S} IB_g}\) is a two-sided ideal in~\(B\);
  \item \label{enu:invariant_ideals_in_graded_symmetric}%
    \(I\) is symmetric.
  \end{enumerate}
  Thus \(\Ideals^B(A)=\Ideals^\B(A)\),
  \(A\subseteq B\)
  is symmetric and conditions \eqref{cond:C1}, \eqref{cond:C2}
  and~\eqref{cond:C3} hold.
\end{proposition}

\begin{proof}
  Assume \ref{enu:invariant_ideals_in_graded_invariant}, so that
  \(I = r\circ i(I) = B I B\cap A\).
  If \(g\in S\), then
  \(B_g I B_g^*\subseteq B I B\cap A=I\).
  So~\ref{enu:invariant_ideals_in_graded_invariant}
  implies~\ref{enu:invariant_ideals_in_graded1}.
  Assume~\ref{enu:invariant_ideals_in_graded1}.  Since
  \(B_g^*B_g= \braket{B_g}{B_g}\)
  is a closed two-sided ideal in~\(A\), we get
  \[
  B_g \cdot I
  = B_g B_g^* B_g I
  = B_g (B_g^*B_g\cap I)
  = B_g I B_g^*B_g
  \subseteq I \cdot B_g.
  \]
  The same computation for~\(B_g^*\) gives
  \(I\cdot B_g \subseteq B_g\cdot I\).  Thus \(I\) is
  \(B_g\)\nb-invariant.  Hence~\ref{enu:invariant_ideals_in_graded1}
  implies~\ref{enu:invariant_ideals_in_graded3}.  To see
  that~\ref{enu:invariant_ideals_in_graded3}
  implies~\ref{enu:invariant_ideals_in_graded4}, note that
  \(\overline{\sum_{g\in S} I B_g}\) is a right ideal in~\(B\).
  In the presence of~\ref{enu:invariant_ideals_in_graded3} it is
  also a left ideal.  If~\ref{enu:invariant_ideals_in_graded4}
  holds, then \(I\subseteq i(I)=\overline{\sum_{g\in S} I B_g}\) is
  non-degenerate.  Thus~\ref{enu:invariant_ideals_in_graded4}
  implies~\ref{enu:invariant_ideals_in_graded_symmetric}.
  That~\ref{enu:invariant_ideals_in_graded_symmetric}
  implies~\ref{enu:invariant_ideals_in_graded_invariant} follows
  from Lemma~\ref{lem:symmetric_induced_ideal}.  Conditions
  \eqref{cond:C1}, \eqref{cond:C2} and~\eqref{cond:C3} follow from
  Corollary~\ref{cor:symmetric_nice}.
\end{proof}

\begin{proposition}
  \label{pro:graded_induced}
  Let \(A\defeq B_e\subseteq B\)
  be the unit fibre of an \(S\)\nb-grading
  \((B_t)_{t\in S}\).
  Let \(J\in \Ideals(B)\)
  and put \(J_g\defeq J\cap B_g\)
  for \(g\in S\).
  The spaces \((J_t)_{t\in S}\)
  form an \(S\)\nb-grading
  of \(i\circ r(J)=\overline{\sum_{g\in S} J_g}\).
  In particular, \(J\)
  is induced if and only if \(J=\overline{\sum_{g\in S} J_g}\)
  is \(S\)\nb-graded.
\end{proposition}

\begin{proof}
  Since \(B_h J_{g}\subseteq J_{hg}\)
  and \(J_h B_{g}\subseteq J_{hg}\)
  for \(h,g\in S\),
  we conclude that \(\overline{\sum_{g\in S} J_g}\)
  is an ideal in~\(B\).
  It is easy to see that \((J_t)_{t\in S}\)
  is an \(S\)\nb-grading
  of \(\overline{\sum_{g\in S} J_g}\).
  In particular, every~\(J_g\)
  is a Hilbert bimodule over \(I\defeq J_e=J\cap A\).
  Then \(J_g= I J_g=I B_g\).
  Thus
  \(\overline{\sum_{g\in S} J_g}=\overline{\sum_{g\in S} IB_g}=i\circ
  r(J)\) (see Proposition~\ref{prop:invariant_ideals_in_graded}).
\end{proof}

Let \(\Bgr=(B_t)_{t\in S}\)
be a saturated \(S\)\nb-grading
on~\(B\).
Each~\(B_t\) is a Hilbert \(B\)\nb-module
and thus induces a partial homeomorphism of~\(\check{A}\)
by the Rieffel correspondence.  More precisely, for \(t\in S\),
\(D_t\defeq B_t^* B_t\)
is an ideal in~\(A\),
which corresponds to an open subset~\(\check{D}_t\)
of~\(\check{A}\),
and the Rieffel correspondence gives a homeomorphism
\(\check{B}_t\colon \check{D}_t\congto \check{D}_{t^*}\).
These partial homeomorphisms for \(t\in S\)
form an action of~\(S\)
on~\(\check{A}\)
by partial homeomorphisms
(see \cite{Buss-Meyer:Actions_groupoids}*{Lemma~6.12}).
As in~\cite{Buss-Exel-Meyer:Reduced}, this inverse semigroup
action has a transformation groupoid~\(\check{A}\rtimes S\)
(compare \cite{Exel:Inverse_combinatorial}*{Section~4}, where this
groupoid is called the groupoid of germs).  This is an \'etale
topological groupoid with object space~\(\check{A}\).
Its arrows are equivalence classes of pairs~\((t,\prid)\)
for \(t\in S\)
and \(\prid \in \check{D}_{t}\subseteq \check{A}\),
where two pairs \((t,\prid)\)
and~\((t',\prid')\)
are equivalent if \(\prid =\prid'\)
and there is \(v\in S\)
with \(v\le t, t'\)
and \(\prid\in \check{D}_{v}\).
There is a unique topology on \(\check{A}\rtimes S\)
for which the source map \([t,\prid]\mapsto \prid\)
is a partial homeomorphism onto~\(\check{D}_{t}\)
for each \(t\in S\).
The subsets
\(U_t \defeq \setgiven{[t,\prid]}{\prid \in \check{D}_t}\)
form an open covering of \(\check{A}\rtimes S\)
by bisections (see
\cite{Exel:Inverse_combinatorial}*{Corollary~4.16}).  Then the range
map \([t,\prid]\mapsto \check{B}_t(\prid)\)
is also a local homeomorphism.

\begin{remark}
  \label{rem:dual_groupoid}
  An ideal \(I\in\Ideals(A)\)
  is \(\B\)\nb-invariant
  if and only if the corresponding open subset
  \(\check{I} \subseteq \check{A}\)
  is invariant for the action of~\(S\),
  if and only if~\(\check{I}\)
  is invariant for the groupoid \(\check{A}\rtimes S\); this follows
  from \cite{Kwasniewski-Szymanski:Pure_infinite}*{Lemma~3.22}
  and the equivalence between \ref{enu:invariant_ideals_in_graded1}
  and~\ref{enu:invariant_ideals_in_graded3} in
  Proposition~\ref{prop:invariant_ideals_in_graded}.
\end{remark}

\begin{theorem}
  \label{thm:inverse_semigroup_graded}
  Let \(A\defeq B_e\subseteq B\)
  be the unit fibre of an \(S\)\nb-grading
  \(\Bgr=(B_t)_{t\in S}\).
  The quasi-orbit space~\(\check{A}/{\sim}\)
  of \(A\subseteq B\)
  exists, and \(\prid_1, \prid_2\in\check{A}\)
  satisfy \(\prid_1 \sim \prid_2\) if and only if
  \(\overline{(\check{A}\rtimes S)\cdot \prid_1}
  = \overline{(\check{A}\rtimes S)\cdot \prid_2}\).
  So~\(\check{A}/{\sim}\)
  is the quasi-orbit space of the transformation groupoid
  \(\check{A}\rtimes S\).
  If \(\Prime(\Ideals^\Bgr(A))\)
  is first countable or~\(\check{A}\) is second countable, then
  \begin{enumerate}
  \item there is a continuous, open and surjective map
    \(\pi\colon \check{A}\to\Prime(\Ideals^\Bgr(A))\),
    which maps~\(\prid\) to the largest \(\Bgr\)\nb-invariant
    ideal contained in it.
    It descends to a homeomorphism
    \(\widetilde{\pi}\colon
    \check{A}/{\sim}\to\Prime(\Ideals^{\Bgr}(A))\).
  \item There is a continuous map
    \(\varrho\colon \check{B} \to \check{A}/{\sim}\),
    \(\prid\mapsto \widetilde{\pi}^{-1}(r(\prid))\).
    It identifies \(\check{B}\cong \check{A}/{\sim}\)
    if and only if~\(A\)
    separates ideals in~\(B\).
  \end{enumerate}
\end{theorem}

\begin{proof}
  Most of the claims follow from our abstract results in
  Theorem~\ref{the:prime_primitive} and
  Lemma~\ref{lem:induced_map_rho}; these apply here because of
  Proposition~\ref{prop:invariant_ideals_in_graded}.  As in the
  proof of Theorem~\ref{the:Fell_bundle_group}, restricted ideals
  in~\(A\) correspond to those open subsets of~\(\check{A}\) that
  are invariant for~\(\check{A}\rtimes S\) (see
  Remark~\ref{rem:dual_groupoid}); and this implies that the
  quasi-orbit space of the inclusion \(A\subseteq B\) is the same as
  for the groupoid~\(\check{A}\rtimes S\).
\end{proof}

The quasi-orbit space in Theorem~\ref{thm:inverse_semigroup_graded}
is homeomorphic to the quasi-orbit space
in~\cite{Boenicke-Li:Ideal}.

\begin{remark}
  By Corollary~\ref{cor:prime_primitive2} and
  Proposition~\ref{pro:graded_induced}, the quasi-orbit map
  \(\varrho\colon \check{B} \to \check{A}/{\sim}\)
  defined above is open and surjective if and only if intersections of
  \(S\)\nb-graded
  ideals are \(S\)\nb-graded
  and the intersection of any \(S\)\nb-graded
  ideal \(I\in \Ideals(B)\)
  with the smallest \(S\)\nb-graded
  ideal containing an ideal \(J\in \Ideals(B)\)
  is the smallest \(S\)\nb-graded ideal containing \(I\cap J\).
\end{remark}

Now we turn to regular \(\Cst\)\nb-inclusions.
In the context of von Neumann algebras such inclusions were introduced
by Dixmier in 1954.

\begin{definition}
  Let \(A\subseteq B\)
  be a \(\Cst\)\nb-subalgebra.
  An element \(b\in B\)
  \emph{normalises}~\(A\)
  if \(b A b^*\subseteq A\)
  and \(b^* A b\subseteq A\).
  Then we also call~\(b\)
  a \emph{normaliser} of~\(A\).
  Let~\(N(A)\)
  be the subset of normalisers (see~\cite{Kumjian:Diagonals}).  The
  inclusion \(A\subseteq B\)
  is \emph{regular} if it is non-degenerate, that is, \(AB=B\),
  and \(N(A)\)
  generates~\(B\)
  as a \(\Cst\)\nb-algebra (see~\cite{Renault:Cartan.Subalgebras}).
\end{definition}

\begin{lemma}
  \label{lem:graded_to_regular}
  Let \(A\defeq B_e\subseteq B\)
  be the unit fibre of an \(S\)\nb-grading
  \(\Bgr=(B_t)_{t\in S}\).
  The \(\Cst\)\nb-inclusion
  \(A\subseteq B\)
  is regular.  In fact, \(B_g\subseteq N(A)\) for every \(g\in S\).
\end{lemma}

\begin{proof}
  That \(A\subseteq B\)
  is non-degenerate follows from
  Proposition~\ref{prop:invariant_ideals_in_graded} applied to
  \(I=A\).
  For every \(g\in S\),
  \(g e g^* \le e\)
  and so \(B_g B_e B_g^* \subseteq B_{g e g^*} \subseteq B_e\).
  Thus \(B_g\subseteq N(A)\).
\end{proof}

The spaces \(B_g\subseteq N(A)\)
in the above lemma have the special feature that they are bimodules
over \(A\).
Exel calls such subspaces of~\(N(A)\)
\emph{slices} in~\cite{Exel:noncomm.cartan} and proves a number of
facts that allow us to show the converse to
Lemma~\ref{lem:graded_to_regular}:

\begin{proposition}
  \label{prop:regular_vs_inverse_semigroups}
  Let \(A\subseteq B\)
  be a non-degenerate \(\Cst\)\nb-subalgebra.  Let
  \[
  \SS(A)\defeq \setgiven{M \subseteq N(A)}
  {M \text{ is a closed linear space and } AM\subseteq M,\ MA\subseteq M}
  \]
  and define
  \[
  M\cdot N\defeq \overline{\operatorname{span}}
  \setgiven{m n}{m \in M, n\in N},\quad
  M^*\defeq \setgiven{m^*}{m \in M}
  \]
  for \(M,N\in \SS(A)\).
  These operations turn~\(\SS(A)\)
  into an inverse semigroup with unit \(A\in \SS(A)\).
  And \((M)_{M\in \SS(A)}\)
  is a saturated \(\SS(A)\)-grading
  on the \(\Cst\)\nb-algebra
  \(\overline{\sum_{M\in \SS(A)} M}\).
  The \(\Cst\)\nb-inclusion
  \(A\subseteq B\)
  is regular if and only if \(B=\overline{\sum_{M\in \SS(A)} M}\).
\end{proposition}

\begin{proof}
  Let \(M,N\in \SS(A)\).
  We have \(M\cdot N\in \SS(A)\)
  by \cite{Exel:noncomm.cartan}*{Proposition 13.1}.  Clearly, the
  multiplication~\(\cdot\)
  is associative and \(M^*\in \SS(A)\).
  \cite{Exel:noncomm.cartan}*{Proposition 10.2} implies
  \(M^*\cdot M \triangleleft A\)
  and \(M\cdot M^* \triangleleft A\).
  Hence~\(M\)
  is naturally a Hilbert \(A\)\nb-bimodule.
  So \(M\cdot M^*\cdot M=M\)
  and \(A\cdot M=M\cdot A= M\).
  Thus~\(\SS(A)\)
  is an inverse semigroup and~\(A\)
  is a unit in~\(\SS(A)\).
  This implies that \(\sum_{M\in \SS(A)} M\)
  is a \Star{}subalgebra of~\(B\).
  Then \(\overline{\sum_{M\in \SS(A)} M}\)
  is a \(\Cst\)\nb-subalgebra.
  If \(M \le N\)
  in~\(\SS(A)\),
  then \(M=M\cdot M^* N\subseteq A\cdot N \subseteq N\).
  So \((M)_{M\in \SS(A)}\)
  is a grading of \(\overline{\sum_{M\in \SS(A)} M}\).
  The Fell bundle \((M)_{M\in \SS(A)}\)
  is saturated by definition.  By
  \cite{Exel:noncomm.cartan}*{Proposition 10.5}, every \(a\in N(A)\)
  lies in some \(M\in \SS(A)\).
  Thus \(B=\overline{\sum_{M\in \SS(A)} M}\)
  if and only if~\(B\)
  is generated as a \(\Cst\)\nb-algebra
  by~\(N(A)\) if and only if \(A\subseteq B\) is regular.
\end{proof}

\begin{corollary}
  \label{cor:non_commutative_cartans}
  A \(\Cst\)\nb-inclusion
  \(A\subseteq B\)
  is regular if and only if
  there are an inverse semigroup~\(S\)
  and an \(S\)\nb-grading
  on~\(B\)
  with unit fibre~\(A\).
  Moreover, the grading may be chosen to be saturated.
\end{corollary}

\begin{proof}
  Combine Lemma~\ref{lem:graded_to_regular} and
  Proposition~\ref{prop:regular_vs_inverse_semigroups}.
\end{proof}

The above result has two advantages.  Firstly, by passing
to a larger inverse semigroup, every graded \(\Cst\)\nb-algebra
may be viewed as a \(\Cst\)\nb-algebra
with a saturated grading (see
also~\cite{BussExel:Regular.Fell.Bundle}).  Secondly, every regular
\(\Cst\)\nb-inclusion
\(A\subseteq B\)
may be studied as a graded \(\Cst\)\nb-algebra
by choosing any inverse subsemigroup \(S\subseteq \SS(A)\)
with \(\overline{\sum_{M\in S}M}=B\).
In fact, in certain cases we may even drop the assumption that~\(S\)
be a semigroup:

\begin{proposition}
  \label{prop:invariant_ideals_in_regular}
  Let \(A\subseteq B\)
  be a regular \(\Cst\)\nb-subalgebra
  and let \(S\subseteq \SS(A)\)
  be any subset such that
  \(\sum_{M\in S} M\) is dense in~\(B\).
  \begin{enumerate}
  \item \label{enu:invariant_ideals_in_regular_invariant}%
    \(I\in \Ideals(A)\)
    is restricted if and only if~\(I\)
    is \(M\)\nb-invariant, that is, \(IM=MI\) for all \(M\in S\);
  \item \label{enu:invariant_ideals_in_regular3}%
    \(J\in \Ideals(B)\)
    is induced if and only if \(J=\overline{\sum_{M\in S} J\cap M}\).
  \end{enumerate}
\end{proposition}

\begin{proof}
  Let~\(\overline{S}\)
  be the inverse semigroup generated by \(S\subseteq \SS(A)\).
  An ideal \(I\in \Ideals(A)\)
  is \(M\)\nb-invariant
  for every \(M\in S\)
  if and only if it
  is \(M\)\nb-invariant
  for every \(M\in \overline{S}\).
  Hence~\ref{enu:invariant_ideals_in_regular_invariant} follows from
  Proposition~\ref{prop:invariant_ideals_in_graded} applied
  to~\(\overline{S}\).
  Now let \(J\in \Ideals(B)\).
  Clearly, \(J=\overline{\sum_{M\in S} J\cap M}\)
  implies \(J=\overline{\sum_{M\in \overline{S}} J\cap M}\).
  Hence~\(J\)
  is induced by Proposition~\ref{pro:graded_induced}.  Conversely,
  if~\(J\)
  is induced, then \(J= I B\)
  for \(I\defeq r(J)=J\cap A\)
  because~\(I\)
  is symmetric by Proposition~\ref{prop:invariant_ideals_in_graded}.
  Hence~\(J\)
  is equal to the closed linear span of \(I\cdot M\subseteq J\cap M\)
  for \(M\in S\).
\end{proof}

\section{Further applications and examples}
\label{sec:applications}

We now apply our theory to a few more
situations, namely, commutative and skew-commutative tensor products
and \(\Cont_0(X)\)-\(\Cst\)-algebras
in Section~\ref{sec:induction_locale_morphism}, relative
Cuntz--Pimsner algebras of \(\Cst\)\nb-\hspace{0pt}correspondences in
Section~\ref{sec:product_systems}, crossed products for
groupoid actions in Section~\ref{sec:groupoid}, and crossed products for quantum
group coactions in Section~\ref{sec:quantum_group_crossed}.

\subsection{Tensor products and
  \texorpdfstring{$\Cont_0(X)$-$\Cst$}{C₀(X)-C*}-algebras}
\label{sec:induction_locale_morphism}

If all ideals in~\(A\)
are symmetric, then
Lemma~\ref{lem:symmetric_restricted_intersection} implies that~\(i\)
commutes with finite intersections.  So~\(i\) is a frame homomorphism
and induces a continuous map
\(\Prime(B) \to \Prime(A)\).
We will exhibit two well known cases where this happens.

\begin{example}
  \label{exa:tensor}
  Let \(B=A\otimes D\)
  be some \(\Cst\)\nb-tensor
  product with a \(\Cst\)\nb-algebra~\(D\)
  and let \(\varphi\colon A\to\Mult(B)\)
  be the canonical non-degenerate \Star{}homomorphism.  If
  \(I\in\Ideals(A)\),
  then \(i(I)\)
  is the closure of the algebraic tensor product \(I\odot D\)
  in \(A\otimes D\),
  which we also denote by \(I\otimes D\).
  Hence any ideal \(I\in\Ideals(A)\)
  is symmetric.
  So~\(i\)
  is a frame homomorphism by
  Lemma~\ref{lem:symmetric_restricted_intersection} and
  Proposition~\ref{pro:restrict_induce_ideals}.\ref{pro:restrict_induce_ideals4}.
  So it induces a continuous map
  \(\Prime(B) \to \Prime(A)\).

  We claim that any ideal \(I\in \Ideals(A)\)
  is restricted.
  Let \(J\in\Ideals(B)\)
  be the kernel of the canonical map from \(A\otimes D\)
  to the minimal \(\Cst\)\nb-tensor product
  \(A/I \otimes_\mathrm{min} D\).
  Then \(r(J)=I\)
  because the map from~\(A\)
  to the multiplier algebra of
  \(A/I \otimes_\mathrm{min} D\)
  vanishes exactly on~\(I\).
  Thus the quasi-orbit space is just the primitive ideal
  space~\(\check{A}\).
  The conditions \eqref{cond:C1}
  and~\eqref{cond:C2} are trivial in this case, and~\eqref{cond:C3}
  holds because~\(i\)
  commutes with finite meets.
  The quasi-orbit map \(\check{B} \to \check{A}\)
  may be constructed directly: it is the restriction
  of~\(r\)
  to \(\check{B}\subseteq \Ideals(B)\).
  Indeed, a representation~\(\pi\)
  of \(A\otimes D\)
  on a Hilbert space~\(\Hils\)
  is described by commuting representations \(\pi_A\)
  and~\(\pi_D\)
  of \(A\)
  and~\(D\)
  on~\(\Hils\),
  respectively.
  Then \(r(\ker \pi) = \ker \pi_A\).
  A subspace of~\(\Hils\)
  that is \(A\)\nb-invariant
  is \(B\)\nb-invariant
  as well.  So~\(\pi_A\)
  is irreducible if~\(\pi\) is.  If \(\prid\in \check{B}\)
  is a primitive ideal, that is, the kernel of an irreducible
  representation, then \(r(\prid) \in\Ideals(A)\)
  is primitive as well.  The restriction
  \(r\colon \check{B} \to \check{A}\)
  of~\(r\) is the quasi-orbit map.
  This map is continuous, and
  \(r^{-1}\colon \Open(\check{A}) \to \Open(\check{B})\)
  becomes the map~\(i\)
  when we identify \(\Open(\check{A}) \cong \Ideals(A)\)
  and \(\Open(\check{B}) \cong \Ideals(B)\).
\end{example}

\begin{example}
  \label{exa:Cstar_over_X}
  Let~\(X\)
  be a locally compact space.  A \emph{\(\Cst\)\nb-algebra
    over~\(X\)}
  or \(\Cont_0(X)\)-\(\Cst\)-algebra
  is a \(\Cst\)\nb-algebra~\(B\)
  with a non-degenerate \Star{}homomorphism from \(A\defeq \Cont_0(X)\)
  to the centre of the multiplier algebra of~\(B\).
  Ideals in~\(A\)
  are of the form \(\Cont_0(U)\)
  for open subsets \(U\subseteq X\),
  and \(i(\Cont_0(U)) = \Cont_0(U)\cdot B = B\cdot \Cont_0(U)\)
  because the image of~\(A\)
  in~\(\Mult(B)\)
  is central.  That is, all ideals in~\(A\)
  are symmetric.  So~\(i\)
  is a frame homomorphism by
  Lemma~\ref{lem:symmetric_restricted_intersection} and
  Proposition~\ref{pro:restrict_induce_ideals}.\ref{pro:restrict_induce_ideals4}.
  Hence~\(i\) induces a continuous map
  \(\Prime(B) \to \Prime(A) \cong X\).
  It restricts to a continuous map \(\pi\colon \check{B} \to X\)
  (see also \cite{Nilsen:Bundles}*{Proposition~2.1}).  Conversely, any
  continuous map \(\check{B} \to X\)
  comes from a \(\Cont_0(X)\)-\(\Cst\)-algebra
  structure on~\(B\) by the Dauns--Hofmann Theorem.

  Let us identify ideals in $A$ and~$B$ with open subsets in $X$
  and~$\check{B}$, respectively.  Then the induction and restriction
  maps have the form \(i(U) = \pi^{-1}(U)\) and
  \(r(V) = \big(X\setminus \pi(\check{B}\setminus V)\big)^\circ\), for
  \(U\in \Open(X)\), \(V\in\Open(\check{B})\), where~\(^\circ\) stands
  for the interior of a given set.  Since~\(r\) preserves meets by
  Proposition~\ref{pro:restrict_induce_ideals}.\ref{pro:restrict_induce_ideals4}
  and \(i\colon \Ideals^B(A)\to \Ideals^A(B)\) is an isomorphism, we
  see that~\eqref{cond:C3'} holds if and only if~\(i\) preserves meets.
  This happens if and only if~\(\pi\) is open (see
  \cite{Meyer-Nest:Bootstrap}*{Lemma~2.9}).  Thus~\eqref{cond:C3'} holds
  if and only if~\(A\) is a \emph{continuous
    \(\Cont_0(X)\)-\(\Cst\)\nb-algebra}.  In particular,
  if~\eqref{cond:C3'} holds, then
  \(\pi^{-1}\colon \Open(X) \to \Open(\check{B})\) is right adjoint to
  \(V\mapsto \pi(V)\), and condition~\eqref{eq:open_surjective} is
  satisfied.  The map~\(i\) is injective if and only if~\(\pi\) is
  surjective, which agrees with Theorem~\ref{the:open_surjective}.
\end{example}

As we have seen, condition~\eqref{cond:C3'} fails in
Example~\ref{exa:Cstar_over_X} for every
\(\Cont_0(X)\)-\(\Cst\)\nb-algebra which is not continuous.  Now we show
that~\eqref{cond:C3'} may fail also in the situation of
Example~\ref{exa:tensor}, for maximal tensor products.

\begin{example}
  \label{exa:MI_fails_tensor}
  Suppose first that~\(D\)
  is exact.  We claim that for every \(\Cst\)\nb-algebra~\(A\),
  the inclusion \(A\to \Mult(A\otimes_{\mathrm{min}} D)\)
  satisfies~\eqref{cond:C3'}.  Indeed, let
  \((I_\lambda)_{\lambda\in\Lambda}\)
  be a family of ideals in~\(A\).
  Clearly, \((\bigcap_{\lambda\in \Lambda}
  I_\lambda)\otimes_{\mathrm{min}} D
  \subseteq \bigcap_{\lambda\in \Lambda} (I_\lambda\otimes_{\mathrm{min}} D)\).
  Let \(R_\psi\colon A\otimes_{\mathrm{min}} D\to A\) denote the
  slice map corresponding to a functional~\(\psi\)
  on~\(D\).  Then
  \[
  R_\psi \biggl(\bigcap_{\lambda\in \Lambda}
  (I_\lambda\otimes_{\mathrm{min}} D)\biggr)
  \subseteq \bigcap_{\lambda\in \Lambda} R_\psi
  (I_\lambda\otimes_{\mathrm{min}} D)
  = \bigcap_{\lambda\in \Lambda} I_\lambda.
  \]
  Hence \(\bigcap_{\lambda\in \Lambda}
  (I_\lambda\otimes_{\mathrm{min}} D)  \subseteq
  (\bigcap_{\lambda\in \Lambda} I_\lambda)\otimes_{\mathrm{min}} D\)
  by \cite{Kirchberg:Fubini_exact}*{Theorem 1.1}.
  Thus
  \[
  \blank\otimes_{\mathrm{min}} D\colon \Ideals(A)\to
  \Ideals(A\otimes_{\mathrm{min}} D)
  \]
  commutes with intersections, that is, \eqref{cond:C3'} holds.

  Now let us characterise condition~\eqref{cond:C3'} for certain
  tensor products in terms of condition~\eqref{cond:C3'} for
  \(\Cont_0(X)\)-\(\Cst\)\nb-algebras.
  Let~\(A\) be a continuous \(\Cont_0(X)\)-\(\Cst\)\nb-algebra and let
  \(B=A\otimes D\)
  be a \(\Cst\)\nb-tensor
  product with some~\(D\).
  The canonical homomorphism \(\varphi\colon A\to\Mult(B)\) extends
  uniquely to a homomorphism \(\Mult(A)\to\Mult(B)\), which maps the
  centre of \(\Mult(A)\) to the centre of \(\Mult(B)\).  Thus the
  composite \(\Cont_0(X)\to \Mult(A)\to \Mult(B)\) gives a homomorphism
  \(\Cont_0(X)\to \Mult(B)\) that gives~\(B\) a
  \(\Cont_0(X)\)-\(\Cst\)\nb-algebra structure.
  Since~\(A\)
  is a continuous \(\Cont_0(X)\)-algebra,
  \eqref{cond:C3'} holds for \(\Cont_0(X)\to \Mult(A)\).
  If it holds for \(A\to \Mult(B)\)
  as well, then it holds also for \(\Cont_0(X)\to \Mult(B)\),
  that is, \(B\)
  is a continuous \(\Cont_0(X)\)-\(\Cst\)\nb-algebra.

  Now suppose that~\(D\)
  is a non-nuclear \(\Cst\)\nb-algebra.   Let \(X=\N^+\)
  be the one-point compactification of~\(\N\).
  By \cite{Kirchberg-Wassermann:Operations}*{Theorem~C},
  there is a separable continuous \(\Cont(\N^+)\)-\(\Cst\)-algebra~\(A\)
  such that the \(\Cont(\N^+)\)-\(\Cst\)-algebra \(A\otimes_{\mathrm{max}} D\)
  is not continuous.  The discussion above shows that
  \(A\to \Mult(A\otimes_{\mathrm{max}} D)\)
  does not satisfy~\eqref{cond:C3'}.  The quasi-orbit map
  \(\varrho\colon \check{B} \to \check{A}\)
  exists but is not open.  In particular, if \(D=\Cst(G)\)
  for a non-amenable (discrete) group, then
  \(B= A \otimes_{\mathrm{max}} D=A\rtimes_\mathrm{triv} G\),
  where \(\mathrm{triv}\) stands for the trivial action.
  Hence one cannot hope for counterparts of Theorems
  \ref{thm:Imai-Takai_duality} and~\ref{thm:open_orbit_map_bundles}
  for full crossed products.
\end{example}

We state two corollaries of Example~\ref{exa:MI_fails_tensor} (and
Theorem~\ref{thm:Imai-Takai_duality}).

\begin{corollary}
  \label{cor:nuclear_characterization}
  A \(\Cst\)-algebra~\(D\) is nuclear if and only if for every
  \(\Cst\)\nb-algebra \(\A\) the generalised \(\Cst\)\nb-inclusion
  \(A\to \Mult(A\otimes_{\mathrm{max}} D)\)
  satisfies~\eqref{cond:C3'}.
\end{corollary}

\begin{corollary}
  \label{cor:amenable_group_characterization}
  A discrete group~\(G\) is amenable if and only if for every
  \(G\)\nb-action on a separable \(\Cst\)\nb-algebra~\(A\) the
  quasi-orbit map \(\varrho\colon \check{B} \to \check{A}/{\sim}\)
  for the full crossed product \(B\defeq A\rtimes_\alpha G\) is open
  and surjective.
\end{corollary}

Tensor products of \(\Cst\)\nb-algebras may be modified so that the
tensor factors no longer commute.  A rather general such construction
using quantum group coactions on the tensor factors is introduced
in~\cite{Meyer-Roy-Woronowicz:Twisted_tensor}.  Here we examine the
simplest case -- the skew-commutative tensor product of
\(\Z/2\)\nb-graded \(\Cst\)\nb-algebras where odd elements in the
tensor factors anticommute (see \cite{Kasparov:Operator_K}*{§2.6}).
Our results are rather negative already in this case: nothing beyond
the Galois adjunction between \(i\) and~\(r\) seems to hold in
general.

\begin{example}
  \label{ex:skew-commutative}
  Let \(A\)
  and~\(D\)
  be \(\Z/2\)\nb-graded
  \(\Cst\)\nb-algebras
  with grading involutions \(\alpha,\delta\)
  and let \(A=A_+\oplus A_-\)
  and \(D=D_+\oplus D_-\)
  be the decompositions into even and odd elements, that is, the
  eigenspaces for \(\alpha\) and~\(\delta\).

  The skew-commutative tensor product \(A \hot D\)
  is a variant of the (minimal) \(\Cst\)\nb-tensor
  product where elements in \(A_-\)
  and~\(D_-\)
  anti-commute.  It is defined
  as a \(\Cst\)\nb-completion
  of the algebraic tensor product \(A\odot D\).
  It comes with injective morphisms
  \(A \hookrightarrow \Mult(A\hot D) \hookleftarrow D\).

  Let \(I\in\Ideals(A)\).
  Then \(I\cdot (A\hot D)\)
  is the closure of \(I\odot D\)
  in \(A\hot D\).
  When we multiply on the left with \(A\hot D\), we get
  \[
  (A\odot D) \cdot (I\odot D)
  = I\odot D + \alpha(I) \odot D_-\cdot D
  \]
  because \(d\cdot (i\otimes d_2) = \alpha(i)\otimes d d_2\)
  if \(d\in D_-\).
  So if~\(D_-\neq0\),
  then~\(I\) is symmetric if and only if \(\alpha(I)=I\).

  If the ideal in~\(D_+\)
  generated by \(D_-\cdot D_-\)
  is equal to~\(D_+\),
  then \((A\hot D)\cdot I\cdot (A\hot D)\)
  is equal to the closure of \((I+ \alpha(I))\odot D\).
  So \(i(I) = i(I+ \alpha(I))\).
  Then \(r\circ i(I) = I+\alpha(I)\),
  which is the smallest \(\Z/2\)\nb-invariant
  ideal containing~\(I\).
  Therefore, the restricted ideals are exactly the
  \(\Z/2\)\nb-invariant
  ideals, and these are also the same as the symmetric ideals.  So our
  theory applies in this case, regardless whether~\(D\)
  is unital or not.

  If the ideal in~\(D_+\)
  generated by \(D_-\cdot D_-\)
  is not equal to~\(D_+\),
  however, then all this breaks down.  Then \(r\circ i(I) = I\)
  for all ideals \(I\in\Ideals(A)\),
  that is, all ideals in~\(A\)
  are restricted.  But not all ideals are symmetric, unless \(D_-=0\).
  So if~\(D_-\)
  is non-zero but not full as a \(D_+\)\nb-module,
  then \(i\colon A\to A\hot D\)
  is not a frame homomorphism, unlike in the situation of commutative
  \(\Cst\)\nb-tensor products in Example~\ref{exa:tensor}.

  An elementary case where this happens is
  \(D = \C \oplus \widehat{\Mat_2}\),
  where~\(\C\)
  is trivially graded and \(\widehat{\Mat_2}\)
  carries the usual inner grading where the off-diagonal entries are
  odd.  Here the ideal generated by~\(D_-\cdot D_-\)
  misses the first summand~\(\C\).
  Now take \(A=\C\oplus\C\)
  with the flip grading.  Then \(A\hot D \cong D\rtimes_\delta \Z/2\)
  for any~\(D\).  In our case,
  \[
  A\hot D \cong \Cst(\Z/2)\otimes D \cong \C\oplus \C\oplus \Mat_2 \oplus \Mat_2
  \]
  because the \(\Z/2\)\nb-action
  on~\(D\)
  is inner.  Here each ideal in~\(A\)
  is restricted from \(A\hot D\).
  But only the invariant ideals \(0\) and~\(A\)
  are symmetric.  Here \(A\subseteq A\hot D\)
  because~\(D\)
  is unital.  In this example, the ideals \(i(\C\oplus0)\)
  and \(i(0\oplus\C)\)
  in \(A\hot D\)
  are \(\C\oplus 0\oplus \Mat_2 \oplus \Mat_2\)
  and \(0\oplus \C\oplus \Mat_2 \oplus \Mat_2\),
  respectively.  Their intersection
  \(0\oplus 0\oplus \Mat_2 \oplus \Mat_2\)
  is not induced, although it is the intersection of two induced
  ideals.
\end{example}

\subsection{Relative Cuntz--Pimsner algebras}
\label{sec:product_systems}

Let~\(X\)
be a \(\Cst\)\nb-correspondence
over a \(\Cst\)\nb-algebra~\(A\).
That is, \(X\)
is a right Hilbert \(A\)\nb-module
with a \Star{}homomorphism \(\varphi_X\colon A\to \Bound(X)\)
which defines a left action of~\(A\)
on~\(X\)
by adjointable operators.  We write \(a x\defeq \varphi_X(a)x\)
for \(a\in A\), \(x\in X\).

A \emph{representation} of the \(\Cst\)\nb-correspondence~\(X\)
in a \(\Cst\)\nb-algebra~\(B\)
is a pair of maps \((\psi_0, \psi_1)\)
where \(\psi_0\colon A\to B\)
is a \Star{}homomorphism and \(\psi_1\colon X\to B\)
is linear and \(\psi_1(x)^* \psi_1(y) = \psi_0(\braket{x}{y}_A)\)
and \(\psi_0(a)\psi_1(x)=\psi_1(ax)\)
for all \(a\in A\) and \(x,y\in X\).
The formula \(\psi^{(1)}\bigl(\ket{x}\bra{y}\bigr) = \psi_1(x)\psi_1(y)^*\)
for \(x,y \in X\)
defines a \Star{}homomorphism \(\psi^{(1)}\colon\Comp(X) \to B\)
on the \(\Cst\)\nb-algebra
of compact operators on~\(X\).
Let
\[
J(X)\defeq \psi^{-1}_0(\Comp(X))\quad \text{ and }\quad J_X\defeq J(X)\cap (\ker\varphi_X)^\bot.
\]
The representation \((\psi_0,\psi_1)\)
is called \emph{covariant on} \(J\idealin J(X)\)
if \(\psi^{(1)}(\psi_X(a))=\psi_0(a)\) for all \(a\in J\).
Let~\(J\)
be an ideal in~\(J(X)\).
There is a universal \(\Cst\)\nb-algebra
\(\OO(J,X)\)
generated by a representation \((j_0,j_1)\)
that is covariant on~\(J\).
We call~\(\OO(J,X)\)
the \emph{Cuntz--Pimsner algebra relative to~\(J\)}.
The homomorphism \(j_0\colon A\to \OO(J,X)\)
is injective if and only if
\(J\subseteq J_X\).
\emph{Katsura's Cuntz--Pimsner algebra} of \(X\)
is \(\OO_X\defeq  \OO(J_X,X)\).
The \emph{Toeplitz algebra} of~\(X\) is \(\TT_X\defeq \OO(0,X)\).
There is a gauge action \(\gamma\colon \T \to \Aut(\OO(J,X))\)
of the circle group, defined by
\(\gamma_z(j_0(a))= j_0(a)\) and \(\gamma_z(j_1(x) =z\cdot j_1(x)\)
for all \(z\in \T\), \(a\in A\), \(x\in X\).

Let~\(I\) be an ideal in~\(A\).
Define
\begin{align*}
  X(I) &\defeq \braket{X}{\varphi_X(I)X}_A =
  \overline{\operatorname{span}}
  \setgiven{\braket{x}{a\cdot y}_A}{a\in I,\ x,y\in X},\\
  X^{-1}(I) &\defeq \setgiven{a\in A}{\braket{x}{a\cdot y}_A \in I
  \text{ for all } x,y\in X}.
\end{align*}
We call~\(I\)
\emph{positively invariant} if \(X(I)\subseteq I\).  Given an
ideal~\(J\)
in~\(A\),
we call~\(I\)
\emph{\(J\)\nb-negatively invariant} or \(J\)\nb-saturated
if \(X^{-1}(I) \cap J \subseteq I\).
We call~\(I\)
\emph{\(J\)\nb-invariant}
if it is positively invariant and \(J\)\nb-negatively invariant.
The \(J_X\)\nb-invariant ideals are called just invariant ideals
in~\cite{Katsura:Ideal_structure_correspondences}.

\begin{proposition}
  \label{pro:restricted_Cuntz_Pismner}
  Let \(B\defeq \OO(J,X)\)
  for an ideal \(J\subseteq J(X)\).
  Let \(\varphi=j_0\colon A\to B\)
  be the canonical \Star{}homomorphism.
  \begin{enumerate}
  \item \label{pro:restricted_Cuntz_Pismner1}%
    \(\Ideals^B(A)\)
    is equal to the lattice~\(\Ideals^X_J(A)\)
    of \(J\)\nb-invariant ideals in~\(A\).
  \item \label{pro:restricted_Cuntz_Pismner2}%
    An induced ideal \(K\in \Ideals^A(B)\),
    generated by \(I\defeq r(K)\),
    is naturally isomorphic to the relative Doplicher--Roberts algebra
    \(\OO_X(I, J\cap I)\)
    introduced in
    \cite{Kwasniewski:Cuntz-Pimsner-Doplicher}*{Definition 7.17}.  The
    ideal~\(K\)
    is gauge-invariant and Morita equivalent to the Cuntz--Pimsner
    algebra \(\OO(J\cap I, XI)\).
    Moreover, \(\OO_X(I, J\cap I)=\OO(J\cap I, XI)\)
    if and only if~\(I\) is symmetric for~\(\varphi\).
  \item \label{pro:restricted_Cuntz_Pismner4}%
    Finite meets of induced ideals are induced, that is,
    condition~\eqref{cond:C3} holds.
  \item \label{pro:restricted_Cuntz_Pismner3}%
    If \(J+\ker(\varphi_X)=A\)
    or~\(X\)
    is a Hilbert bimodule and \(J= J_X\),
    then \(\Ideals^A(B)\)
    coincides with the lattice \(\Ideals^\gamma(B)\)
    of gauge-invariant ideals in~\(B\).
  \item \label{pro:restricted_Cuntz_Pismner5}%
    If \(\Ideals^A(B)=\Ideals^\gamma(B)\)
    then conditions \eqref{cond:C3'} and~\eqref{cond:C4} hold.
  \end{enumerate}
\end{proposition}

\begin{proof}
  Statements \ref{pro:restricted_Cuntz_Pismner1}
  and~\ref{pro:restricted_Cuntz_Pismner2} follow from
  \cite{Kwasniewski:Cuntz-Pimsner-Doplicher}*{Theorem 6.20}.  In
  particular, the positive invariance of a restricted ideal is
  straightforward and well known.  Moreover, if \(I\in \Ideals^B(A)\)
  is \(J\)\nb-invariant
  and we view \(\OO_X(I, J\cap I)\)
  and \(\OO(J\cap I, XI)\)
  as \(\Cst\)\nb-subalgebras
  of \(\OO(J,X)\),
  then
  \begin{align*}
    \OO_X(I, J\cap I) &= \overline{\operatorname{span}}
    \setgiven{j_{n}(X^{\otimes n})j_0(I)j_{m}(X^{\otimes m})^*}
    {n,m\in \N},\\
    \OO(J\cap I, XI) &= \overline{\operatorname{span}}
    \setgiven{j_0(I)j_{n}(X^{\otimes n})j_{m}(X^{\otimes m})^*j_0(I)}
    {n,m\in \N},
  \end{align*}
  where \((j_0,j_n)\)
  is the representation of \((A,X^{\otimes n})\)
  induced by \((j_0,j_1)\).
  This implies \(\OO(J\cap I, XI)\subseteq \OO_X(I, J\cap I)\),
  and equality holds here if and only if~\(j_0(I)\)
  is a non-degenerate subalgebra of \(\OO_X(I, J\cap I)\),
  if and only if~\(I\) is symmetric.

  Statement~\ref{pro:restricted_Cuntz_Pismner3} follows from
  \cite{Kwasniewski:Cuntz-Pimsner-Doplicher}*{Theorem~7.11}.  We
  prove~\ref{pro:restricted_Cuntz_Pismner5}.  Assume
  \(\Ideals^A(B)=\Ideals^\gamma(B)\).
  Let \(F(J)\defeq\overline{\sum_{z\in \T} \gamma_z(J)}\)
  for \(J\in \Ideals(B)\).
  Then \(F\colon \Ideals(B) \hookrightarrow \Ideals^A(B)\)
  satisfies~\eqref{eq:definition_of_F}.  Thus \eqref{cond:C3'} holds by
  Lemma~\ref{lem:galois_insertion}.  If \(I\in\Ideals^A(B)\)
  and \(J\in\Ideals(B)\),
  then \(I \gamma_z(J)= \gamma_z(I) \gamma_z(J)=\gamma_z(IJ)\)
  for all \(z\in \T\).
  Thus
  \(I\cap F(J) = I \cdot \overline{\sum_{z\in \T} \gamma_z(J)} =
  \overline{\sum_{z\in \T} \gamma_z(I J)} =F(I\cap J)\).
  That is, \eqref{cond:C4} holds.

  We prove~\ref{pro:restricted_Cuntz_Pismner4}.  Let
  \(I_1, I_2\in \Ideals^B(A)\).
  Then \(i(I_1\cap I_2)\subseteq i(I_1)\cap i(I_2)\)
  because~\(i\)
  is monotone.  To see the reverse inclusion, recall that
  \(i(I_1)\cap i(I_2)=i(I_1)i(I_2)\)
  is spanned by elements in the sets
  \[
  \Bigl(j_{n}(X^{\otimes n})j_0(I_1)j_{m}(X^{\otimes m})^*\Bigr)
  \Bigl(j_{l}(X^{\otimes l})j_0(I_2)j_{k}(X^{\otimes k})^*\Bigr),
  \]
  where \(n,m,l,k \in \N\).
  Assuming, for instance, that \(m\ge l\)
  and using properties of the maps \(j_m\)
  and~\(j_l\), we see that the above set is contained in
  \[
  j_{n}(X^{\otimes n})j_0(I_1)j_{m-l}(X^{\otimes m-l})^*j_0(I_2)j_{k}(X^{\otimes k})^*.
  \]
  This is contained in
  \(j_{n}(X^{\otimes n})j_0(I_1) j_0(I_2) j_{k+m-l}(X^{\otimes
    k+m-l})^* \subseteq i(I_1\cap I_2)\)
  because~\(I_2\)
  is positively invariant.  Thus
  \(i(I_1)\cap i(I_2)\subseteq i(I_1\cap I_2)\).
\end{proof}

\begin{remark}
  Since arbitrary meets of restricted ideals are restricted by
  Proposition~\ref{pro:restrict_induce_ideals}.\ref{pro:restrict_induce_ideals6},
  statement~\ref{pro:restricted_Cuntz_Pismner1} implies and
  strengthens \cite{Katsura:Ideal_structure_correspondences}*{first
    part of Proposition 4.10 and Corollary 4.11}.
\end{remark}

It seems that the homomorphism \(j_0\colon A\to \OO(J,X)\)
is never symmetric, unless~\(j_1(X)\)
is a Hilbert bimodule over~\(j_0(A)\),
when we are dealing with a Fell bundle over~\(\Z\).
Even if~\(A\)
separates ideals in~\(\OO_X\)
condition~\eqref{cond:C1} usually fails:

\begin{example}[\cite{Katsura:Ideal_structure_correspondences}*{Example 4.12}]
  \label{exa:Katsura_graph}
  Let~\(X\)
  be the \(\Cst\)\nb-correspondence
  over \(A=\C^3\)
  built from the directed graph
  \begin{tikzcd}
    \bullet &\bullet \arrow[l] \arrow[r] & \bullet
  \end{tikzcd}.
  Then \(\OO_X=M_2(\C)\oplus M_2(\C)\)
  and \(j_0(a,b,c)= \begin{pmatrix} a & 0 \\
    0 & c \end{pmatrix} \oplus
  \begin{pmatrix}
    b & 0 \\
    0 & c
  \end{pmatrix}\).
  We have already met this inclusion in
  Example~\ref{exa:join_not_restricted}.  The join of the
  restricted ideals \(I_1=\C\oplus 0 \oplus 0\)
  and \(I_2=0\oplus \C \oplus 0\)
  is not restricted.  Note that \(A\)
  separates ideals in \(B=\OO_X\)
  and, in agreement with
  Lemma~\ref{lem:separate_ideals_into_primes}, there is a
  continuous map
  \(r\colon \check{B}\congto
  \Prime(\Ideals^B(A))\subsetneq \check{A}\).
  There is, however, no surjective map from~\(\check{A}\)
  onto~\(\check{B}\).
  This is no contradiction with
  Corollary~\ref{cor:primitive_ideal_space_description}
  because~\eqref{cond:C1} is not satisfied.
\end{example}

The above discussion shows the following.  In the context of
Cuntz--Pimsner algebras, apart from the adjunction, the only
fact we get from our general theory is that
\(r\colon \Prime(\OO(J,X)) \to \Prime^B(A)=\Prime(\Ideals^X_J(A))\)
is a well defined continuous map (combine Proposition
\ref{pro:restricted_Cuntz_Pismner} and Lemma~\ref{lem:induced_map_rho}).
When we are not in the situation of
statement~\ref{pro:restricted_Cuntz_Pismner3} in
Proposition~\ref{pro:restricted_Cuntz_Pismner}, there are
gauge-invariant ideals in \(\OO(J,X)\)
which are not induced, and then we cannot even apply
Theorem~\ref{the:prime_primitive2} directly.
Nevertheless, we may overcome this issue by using the following
Lemma~\ref{lem:from_A_to_A_tilde}.

For any positively invariant ideal~\(I\)
there is a natural quotient \(\Cst\)\nb-correspondence
\(X_I\defeq X/XI\)
over~\(A/I\).  Let \(q_I\colon A\to A/I\) be the quotient map and let
\[
J_X(I)\defeq\setgiven{a\in A}{\varphi_{X_I}(q_I(a))\in \Comp(X_I), \ aX^{-1}(I) \subseteq I}.
\]
Let \(J\idealin J(X)\).
A \emph{\(J\)\nb-pair}
of~\(X\)
is a pair \((I,I')\)
of ideals \(I\),
\(I'\)
of~\(A\)
such that~\(I\)
is positively invariant and \(J+I\subseteq I'\subseteq J_X(I)\).
Let \(\Pair^X_J(A)\)
be the set of \(J\)\nb-pairs.
Equip it with the natural pre-order coming from inclusion.  Then the
map \(\Ideals^X_J(A) \hookrightarrow\Pair^X_J(A)\),
\(I \mapsto (I, I+J)\),
is an order-preserving embedding.  It is an isomorphism, for instance,
in the situation of statement~\ref{pro:restricted_Cuntz_Pismner3} in
Proposition~\ref{pro:restricted_Cuntz_Pismner}.

\begin{lemma}
  \label{lem:from_A_to_A_tilde}
  Let \(B\defeq \OO(J,X)\)
  for an ideal \(J\subseteq J(X)\).
  Put \(\widetilde{A}\defeq j_0(A)+ j^{(1)}(\Comp(X))\)
  and consider the inclusion \(\widetilde{A}\subseteq B\).
  \begin{enumerate}
  \item \label{lem:from_A_to_A_tilde1} We have a lattice isomorphism
    \(\Ideals^{B}(\widetilde{A})\cong \Pair^X_J(A)\),
    where a restricted ideal~\(I\)
    in~\(\widetilde{A}\)
    is mapped onto the \(J\)\nb-pair
    \(\left(j_0^{-1}(I), j_0^{-1} \bigl(I+
      j^{(1)}(\Comp(X))\bigr)\right)\) of~\(X\);
  \item \label{lem:from_A_to_A_tilde2}
    \(\Ideals^{\widetilde{A}}(B)=\Ideals^\gamma(B)\),
    that is, induced and gauge-invariant ideals coincide.
  \end{enumerate}
\end{lemma}

\begin{proof}
  The map \(\Ideals^\gamma(B) \ni I \mapsto \left(j_0^{-1}(I),
    j_0^{-1}\bigl(I+ j^{(1)}(\Comp(X))\bigr)\right) \in
  \Pair^X_J(A)\) is a lattice isomorphism by
  \cite{Katsura:Ideal_structure_correspondences}*{Proposition~11.9}
  (see also
  \cite{Kwasniewski:Cuntz-Pimsner-Doplicher}*{Theorem~7.17}).  This
  readily implies~\ref{lem:from_A_to_A_tilde2} and hence
  also~\ref{lem:from_A_to_A_tilde1}.
\end{proof}

We take the opportunity to use a recent result
of~\cite{Carlsen-Kwasniewski-Ortega:Free_correspondence}.
If \(J\idealin J(X)\) then the kernel~\(J_\infty\) of \(j_0\colon A\to \OO(J,X)\)
is the smallest \(J\)\nb-invariant ideal containing~\(0\).
There is a canonical isomorphism
\(\OO(J,X)\cong\OO(q_{J_{\infty}}(J), X_{J_{\infty}})\),
where \(q_{J_{\infty}}\colon A\to A/J_{\infty}\)
is the quotient map and \(X_{J_{\infty}}=X/XJ_{\infty}\)
is the quotient \(\Cst\)\nb-correspondence
over~\(A/J_{\infty}\) (see
\cite{Kwasniewski:Cuntz-Pimsner-Doplicher}*{Theorem~6.23}).

\begin{definition}
  Let~\(X\)
  be a \(\Cst\)\nb-correspondence
  over a liminal \(\Cst\)\nb-algebra~\(A\)
  and let \(J\idealin J(X)\).
  We call~\(X\) \emph{topologically free on~\(J\)}
  if the graph dual to~\(X_{J_{\infty}}\)
  is topologically free on the set corresponding to~\(q_{J_{\infty}}(J)\)
  (see
  \cite{Carlsen-Kwasniewski-Ortega:Free_correspondence}*{Defitinions
    5.1 and~3.2}).  We call~\(X\)
  \emph{residually topologically free on~\(J\)}
  if, for every \((I,I')\in \Pair^X_J(A)\),
  the \(\Cst\)\nb-correspondence~\(X_I\)
  is topologically free on~\(q_I(I')\).
\end{definition}

\begin{theorem}
  \label{thm:Cuntz-Pimsner_algebras}
  Let \(B\defeq \OO(J,X)\)
  for an ideal \(J\subseteq J(X)\).
  Suppose that~\(\check{B}\)
  is second countable
  or that \(\Prime(\Pair^X_J(A))\)
  is first countable.  The formula
  \[
  r(\prid)\defeq
  \left(j_0^{-1}(\prid), j_0^{-1}\bigl(\prid+ j^{(1)}(\Comp(X))\bigr)\right)
  \]
  defines a continuous open and surjective map
  \(r\colon \check{B}\to \Prime(\Pair^X_J(A))\).
  The following conditions are equivalent:
  \begin{enumerate}
  \item \label{thm:Cuntz-Pimsner_algebras1}%
    \(r\) is a homeomorphism;
  \item \label{thm:Cuntz-Pimsner_algebras2}%
    all ideals in~\(B\) are gauge-invariant;
  \item \label{thm:Cuntz-Pimsner_algebras3}%
    \(\widetilde{A}\defeq j_0(A)+ j^{(1)}(\Comp(X))\)
    separates ideals in \(\OO(J,X)\).
  \end{enumerate}
  If~\(A\) is liminal, the above are further equivalent to
  \begin{enumerate}[resume]
  \item \label{thm:Cuntz-Pimsner_algebras4}%
    \(X\)~is residually topologically free on~\(J\).
  \end{enumerate}
\end{theorem}

\begin{proof}
  Since gauge-invariant ideals are closed under arbitrary meets,
  Lemma~\ref{lem:from_A_to_A_tilde} implies that the
  \(\Cst\)\nb-inclusion
  \(\widetilde{A}\subseteq B\)
  satisfies the assumptions of the second part of
  Theorem~\ref{the:prime_primitive2}.  This theorem implies the first
  part of the present assertion.  Since
  \(\Ideals^{\widetilde{A}}(B)=\Ideals^\gamma(B)\),
  we readily see the equivalence between
  \ref{thm:Cuntz-Pimsner_algebras1}--\ref{thm:Cuntz-Pimsner_algebras3}.

  Assume now that~\(A\)
  is liminal.  By
  \cite{Carlsen-Kwasniewski-Ortega:Free_correspondence}*{Theorem~8.3}
  (see also
  \cite{Carlsen-Kwasniewski-Ortega:Free_correspondence}*{Lemma~6.2}),
  the algebra~\(\widetilde{A}\)
  detects ideals in~\(B\)
  if and only if~\(X\)
  is topologically free on~\(J\).
  For any \(J\in \Ideals^{\widetilde{A}}(B)=\Ideals^\gamma(B)\),
  the quotient~\(B/J\)
  is naturally isomorphic to the relative Cuntz--Pimsner algebra
  \(\OO(q_I(I'),X_I)\)
  (see \cite{Kwasniewski:Cuntz-Pimsner-Doplicher}*{Theorem~7.17}).
  Under this isomorphism, the algebra~\(\widetilde{A}\)
  is mapped onto its counterpart in \(\OO(q_I(I'),X_I)\).
  Hence Proposition~\ref{pro:separate_induced} implies the equivalence
  between \ref{thm:Cuntz-Pimsner_algebras3}
  and~\ref{thm:Cuntz-Pimsner_algebras4}.
\end{proof}

\begin{corollary}
  Let \(B\defeq \OO(J,X)\)
  for an ideal \(J\subseteq J(X)\)
  with \(J+\ker(\varphi_X)=A\).
  Suppose that \(\Prime(\Ideals^X_J(A))\)
  is first countable or~\(\check{B}\)
  is second countable.  Then \(r(\prid)\defeq j_0^{-1}(\prid)\)
  defines a continuous, open and surjective map
  \(r\colon \check{B}\to \Prime(\Ideals^X_J(A))\),
  and~\(r\)
  is a homeomorphism if and only if~\(A\) separates ideals in~\(B\).

  If~\(A\)
  is liminal, then~\(r\)
  induces a homeomorphism \(\check{B}\cong\Prime(\Ideals^X_J(A))\)
  if and only if for every \(J\)\nb-invariant
  ideal~\(I\)
  the graph dual to~\(X_I\)
  is topologically free on
  \(\widehat{q_I(J)}\cong\widehat{J}\setminus \widehat{I}\).
\end{corollary}

\begin{proof}
  If \(J+\ker(\varphi_X)=A\)
  we may replace \(\Pair^X_J(A)\) by
  \(\Ideals^X_J(A)\).
\end{proof}

\subsection{Groupoid actions}
\label{sec:groupoid}

Throughout this section, let~\(G\)
be a locally compact groupoid with Haar system.  Let~\(G^0\)
be its object space.  We are going to generalise
Proposition~\ref{pro:restricted_group_crossed} and
Theorem~\ref{thm:group_restricted_ideals_nice} to actions of~\(G\)
(see~\cite{Muhly-Williams:Renaults_equivalence}).  Such an action
requires a \(\Cont_0(G^0)\)-\(\Cst\)-algebra~\(A\).
Let~\(A_x\)
for \(x\in G^0\)
be its fibres.  The arrows in~\(G\)
act by isomorphisms \(\alpha_g\colon A_{s(g)} \congto A_{r(g)}\)
for all \(g\in G\),
which satisfy the usual algebraic condition
\(\alpha_g \alpha_h = \alpha_{g h}\)
for composable \(g,h\in G\)
and which are continuous in a suitable sense.  Namely, if
\(U\subseteq G\)
is a Hausdorff, open subset, then applying~\(\alpha_g\)
for \(g\in G\)
pointwise gives an isomorphism of \(\Cont_0(U)\)-\(\Cst\)-algebras
\(s|_U^*(A) \congto r|_U^*(A)\).

We fix a \(\Cont_0(G^0)\)-\(\Cst\)-algebra~\(A\)
with a continuous action~\(\alpha\) of~\(G\).
The full crossed product \(A\rtimes G\)
is defined in~\cite{Muhly-Williams:Renaults_equivalence},
following~\cite{Renault:Representations}.  The reduced crossed
product \(A\rtimes_\lambda G\)
is easier to define, using a particular family of ``regular''
representations of the convolution algebra that defines \(A\rtimes G\).
The morphism \(A\to \Mult(A\rtimes_\lambda G)\)
is well known to be injective.  A \emph{crossed product} for
\((A,G,\alpha)\)
is a \(\Cst\)\nb-algebra~\(B\)
with surjections \(A\rtimes G \onto B \onto A\rtimes_\lambda G\)
that compose to the regular representation.
The following results apply to any crossed product.

Let
\(I\in\Ideals(A)\).  Let \(I_x\in\Ideals(A_x)\) for \(x\in G^0\) be
the image of~\(I\) in the fibre~\(A_x\).  These ideals
determine~\(I\) uniquely.
We call~\(I\) \emph{\(\alpha\)\nb-invariant} if
\(\alpha_g(I_{s(g)}) = I_{r(g)}\) for all \(g\in G\).  The
ideal~\(I\) inherits a \(\Cont_0(G^0)\)-\(\Cst\)-algebra structure
through the canonical morphism \(\Mult(A) \to \Mult(I)\).  Its
fibres are canonically isomorphic to the ideals
\(I_x\in\Ideals(A_x)\).  Being invariant means that the action
on~\(A\) restricts to an action~\(\alpha|_I\) on~\(I\).  Moreover,
there is an induced action on the quotient~\(A/I\).  It inherits a
\(\Cont_0(G^0)\)-\(\Cst\)-algebra structure through the canonical
morphism \(\Mult(A) \to \Mult(A/I)\).  The fibre of~\(A/I\) at
\(x\in G^0\) is the quotient \(A_x/I_x\), and~\(\alpha_g\) for
\(g\in G\) indeed induces an isomorphism \(\dot\alpha_g\colon
A_{s(g)}/I_{s(g)} \congto A_{r(g)}/I_{r(g)}\) if~\(I\) is invariant.
The \(\Cont_0(G^0)\)-\(\Cst\)-algebra structure on~\(A\) is
equivalent to a continuous map \(p\colon \check{A} \to G^0\).
The isomorphism~\(\alpha_g\) above induces a homeomorphism
\(\alpha_g^*\colon p^{-1}(s(g)) \congto p^{-1}(r(g))\).  This
defines an action of the groupoid~\(G\) on~\(\check{A}\).  This
action is continuous, that is, the maps \(\alpha_g^*\) above piece
together to a homeomorphism \(G\times_{s,G^0,p} \check{A} \congto
G\times_{r,G^0,p} \check{A}\), regardless of whether~\(\check{A}\)
is Hausdorff or not.
An ideal~\(I\) is
\(G\)\nb-invariant if and only if the corresponding open subset
of~\(\check{A}\) is \(G\)\nb-invariant.

\begin{lemma}
  \label{lem:invariant_restricted_groupoid}
  Let \(I\in\Ideals(A)\) be invariant.  Then~\(I\) is symmetric and
  restricted, and the induced ideal \(i(I)\) in~\(B\)
  is the image in~\(B\) of the ideal \(I\rtimes G\idealin A\rtimes G\).
\end{lemma}

\begin{proof}
  Due to the surjective map \(A\rtimes G\onto B\),
  it suffices to prove that~\(I\)
  is symmetric for the inclusion \(A\subseteq \Mult(A\rtimes G)\).
  By definition, \(A\rtimes G\)
  is the \(\Cst\)\nb-completion
  of a certain convolution \Star{}algebra \(\mathfrak{S}(G,A)\)
  of compactly supported functions \(G\to A\).
  Namely, \(\mathfrak{S}(G,A)\)
  is the linear span of the spaces of compactly supported continuous
  sections of the \(\Cont_0(U)\)-\(\Cst\)-algebra
  \(r|_U^*(A)\)
  for Hausdorff, open, relatively compact subsets \(U\subseteq G\).
  The left multiplication with elements of~\(A\)
  is simply pointwise multiplication.  Therefore,
  \(I\cdot\mathfrak{S}(G,A) = \mathfrak{S}(G,I)\).
  This is a \Star{}subalgebra of~\(\mathfrak{S}(G,A)\)
  because~\(I\)
  is invariant.  So~\(\mathfrak{S}(G,I)\)
  carries its own \Star{}algebra structure, defined by the same
  formulas.  Thus
  \[
  \mathfrak{S}(G,A) \cdot I
  = (I^* \cdot \mathfrak{S}(G,A)^*)^*
  = (I \cdot \mathfrak{S}(G,A))^*
  = \mathfrak{S}(G,I)^*
  = \mathfrak{S}(G,I).
  \]
  Thus~\(I\) is symmetric.  Moreover, the ideal in
  \(A\rtimes G\) induced by~\(I\)
  is the closure of \(\mathfrak{S}(G,I)\).
  Passing to a quotient~\(B\) of \(A\rtimes G\), the ideal~\(i(I)\)
  becomes the image of~\(I\rtimes G\) in~\(B\) because it makes no
  difference whether we first close \(\mathfrak{S}(G,I)\) in
  \(A\rtimes G\) and then project to~\(B\) or the other way around.

  We are going to prove that \(r(i(I)) = I\).  We use the
  homomorphism \(B/i(I) \onto (A\rtimes_\lambda G)/i(I) \onto A/I
  \rtimes_\lambda G\), where we use the canonical map \(A
  \rtimes_\lambda G\to A/I \rtimes_\lambda G\), which clearly
  vanishes on \(I\rtimes_\lambda G\), the image of \(I\rtimes G\) in
  \(A\rtimes_\lambda G\).  The canonical map \(A/I \to \Mult(A/I
  \rtimes_\lambda G)\) is injective.  Hence the map \(A/I \to
  \Mult(B/i(I))\) is injective as well.  Since \(r(i(I))\) is the
  kernel of the map \(A \to \Mult(B/i(I))\), this implies \(I =
  r(i(I))\).
\end{proof}

\begin{lemma}
  \label{lem:invariant_ideals_properties_groupoid}
  Joins of invariant ideals are again invariant.
\end{lemma}

\begin{proof}
  The lattice \(\Ideals(A)\) is isomorphic to \(\Open(\check{A})\),
  where joins are simply unions.  Here invariant ideals correspond
  to \(G\)\nb-invariant open subsets of~\(\check{A}\).  A union of
  \(G\)\nb-invariant subsets of~\(\check{A}\) is again
  \(G\)\nb-invariant.
\end{proof}

\begin{proposition}
  \label{pro:restricted_ideal_groupoid}
  Let~\(\alpha\)
  be an action of a second countable groupoid~\(G\)
  on a separable \(\Cst\)\nb-algebra~\(A\).
  Let \(\varphi\colon A\to \Mult(B)\)
  be the canonical generalised \(\Cst\)\nb-inclusion
  of~\(A\)
  into a crossed product algebra~\(B\)
  of~\(\alpha\).
  \begin{enumerate}
  \item \label{pro:restricted_ideal_groupoid1}%
    \(\Ideals^B(A)\)
    is the subset \(\Ideals^\alpha(A)\) of
    \(\alpha\)\nb-invariant ideals in~\(A\);
  \item \label{pro:restricted_ideal_groupoid2}%
    an ideal \(J\idealin B\)
    is induced if and only if it is the image of
    \(I\rtimes G\)
    in~\(B\) for some \(\alpha\)\nb-invariant ideal \(I\idealin A\).
  \end{enumerate}
  The \(\Cst\)\nb-inclusion
  \(\varphi\colon A\to \Mult(B)\)
  is symmetric, and \eqref{cond:C1}, \eqref{cond:C2} and \eqref{cond:C3} hold.
\end{proposition}

\begin{proof}
  In view of Lemmas \ref{lem:invariant_restricted_groupoid}
  and~\ref{lem:invariant_ideals_properties_groupoid}
  and Corollary \ref{cor:symmetric_nice}, it only remains to show that
  every restricted ideal is \(\alpha\)\nb-invariant.
  To this end, let \(J\in\Ideals(B)\).
  We will prove that \(r(J)\in\Ideals(A)\) is \(\alpha\)\nb-invariant.

  The assumptions imply that~\(B\) is separable.  So there is a
  faithful representation \(B/J \hookrightarrow \Bound(\Hils)\) on a
  separable Hilbert space~\(\Hils\).  We use the quotient map
  \(A\rtimes G\onto B\) to view it as a representation \(\pi\colon
  A\rtimes G \to \Bound(\Hils)\).  There are canonical morphisms
  \(A\to \Mult(A\rtimes G)\) and \(\Cst(G) \to \Mult(A\rtimes G)\),
  which give us representations \(\pi^A\) and \(\pi^G\) of \(A\)
  and~\(\Cst(G)\) from~\(\pi\).  The kernel of the morphism \(A\to
  \Mult(B/J)\) is the kernel of~\(\pi^A\) because the extension of a
  faithful representation to the multiplier algebra remains
  faithful.  We must show that the ideal \(\ker \pi^A\) in~\(A\) is
  \(\alpha\)\nb-invariant.

  Our assumptions ensure that Renault's Disintegration Theorem
  applies to~\(\pi\)
  (see \cites{Muhly-Williams:Renaults_equivalence,
    Renault:Representations}).  This gives us the
  following structure: first, a quasi-invariant measure~\(\nu\)
  on~\(G^0\) and a Borel field of Hilbert spaces~\((\Hils_x)_{x\in
    G^0}\) over~\(G^0\) such that~\(\Hils\) is isomorphic to the
  Hilbert space \(L^2(G^0,\nu,(\Hils_x)_{x\in G^0})\) of
  square-integrable sections of the field~\((\Hils_x)_{x\in G^0}\)
  with respect to the measure~\(\nu\); secondly, a Borel
  representation~\(U\) of~\(G\) by unitary operators \(U_g\colon
  \Hils_{s(g)} \congto \Hils_{r(g)}\) for \(g\in G\), such that the
  representation~\(\pi^G\)
  is obtained by integrating the representation~\(U\)
  of~\(G\); thirdly, a Borel family of
  representations~\(\pi^A_x\) of~\(A_x\) on~\(\Hils_x\) for all
  \(x\in G^0\), which is covariant with respect to the
  representation~\(U\), that is, \(\pi^A_{r(g)}(\alpha_g(a)) = U_g
  \pi^A_{s(g)}(a) U_g^*\) for all \(a\in A\), \(g\in G\).  The
  representation~\(\pi^A\) of~\(A\) on~\(\Hils\) is the pointwise
  application of~\(\pi^A_x\).  So \(\pi^A(a)=0\) for \(a\in A\) if
  and only if the set of \(x\in G^0\) with \(\pi^A_x(a)\neq0\) is a
  \(\nu\)\nb-null set.

  We identify \(\ker(\pi^A)\) with an open subset~\(U\)
  of~\(\check{A}\) and \(\ker(\pi^A_x)\) with an open subset~\(U_x\)
  of \(\check{A}_x \defeq p^{-1}(x)\subseteq \check{A}\) for each
  \(x\in G^0\).  Since~\((\pi^A_x)\) is covariant with respect
  to~\(U\), the subset \(\tilde{U}\defeq \bigsqcup_{x\in G^0} U_x
  \subseteq \check{A}\) is \(G\)\nb-invariant.  Given \(a\in A\),
  let \(V_a \defeq \setgiven{\prid\in\check{A}}{a\notin\prid}\).
  Then \(\pi^A_x(a)=0\) if and only if \(V_a\cap
  \check{A}_x\subseteq U_x\).  So \(\pi^A(a)=0\) if and only if
  \(p(V_a\setminus \tilde{U})\subseteq G^0\) is a \(\nu\)\nb-null
  set.  The subsets \(V_a\subseteq \check{A}\) are open.  We have
  \(V_a\subseteq U\) if \(a\in\ker\pi^A\), and \(U=V_a\) if \(a\in
  \ker(\pi^A)\) is, say, strictly positive.  Thus~\(U\) is the
  largest open subset~\(V\) of~\(\check{A}\) with the property that
  \(\nu(p(V\setminus \tilde{U}))=0\).

  Let~\(\mu\) be the Haar system on~\(G\) and let
  \(\nu\circ\mu\) be the measure on~\(G\) that first integrates
  along~\(\mu\) and then along~\(\nu\).  Given a subset~\(T\)
  of~\(\check{A}\), let \(r^*(T)\) and \(s^*(T)\) be their
  pre-images in the arrow space \(G\times_{r,G^0,p} \check{A}\) of
  the transformation groupoid \(G\ltimes \check{A}\) under the range
  and source maps, respectively.  Since the range map of~\(G\) is
  open and each fibre of~\(\mu\) has full support, the largest
  open subset \(V\subseteq G\ltimes\check{A}\) such that
  \(\nu\circ\mu\bigl((\Id\times_{r,G^0} p)(V \setminus
  r^*(\tilde{U}))\bigr) =0\) is \(r^*(U)\).  Since~\(\nu\) is
  quasi-invariant, the inversion in~\(G\) preserves the property of
  being a \(\nu\circ\mu\)-null set.  Therefore, \(r^*(U)\) and
  \(s^*(U)\) are both the largest open subsets in
  \(G\ltimes\check{A}\) with the same property.  Hence they are
  equal.  And this says that~\(U\) is \(G\)\nb-invariant.
\end{proof}

\begin{theorem}
  \label{thm:groupoid_quasi_orbit}
  Let~\(\alpha\)
  be an action of a second countable groupoid~\(G\)
  on a separable \(\Cst\)\nb-algebra~\(A\).
  Let \(\varphi\colon A\to \Mult(B)\)
  be the canonical generalised \(\Cst\)\nb-inclusion
  of~\(A\)
  into a  crossed product algebra~\(B\)
  of \(\alpha\).
  The quasi-orbit space~\(\check{A}/{\sim}\)
  and the quasi-orbit map \(\varrho\)
  for~\(\varphi\) exist, and
  \begin{enumerate}
  \item \(\check{A}/{\sim}\)
    is a quotient of~\(\check{A}\)
    by the open equivalence relation where \(\prid_1\sim \prid_2\)
    if and only if the orbit closures \(\overline{G\cdot \prid_i}\)
    for \(i=1,2\)
    are equal.
  \item \(\varrho\colon \check{B} \to \check{A}/{\sim}\)
    is given by \(\prid\mapsto \widetilde{\pi}^{-1}(r(\prid))\),
    where \(\widetilde{\pi}\colon
    \check{A}/{\sim}\to\Prime(\Ideals^\alpha(A))\)
    is the homeomorphism induced by the continuous, open and
    surjective map
    \(\pi\colon \check{A}\to\Prime(\Ideals^\alpha(A))\),
    \(\prid\mapsto \bigcap_{g\in s^{-1}( p(\prid))} \alpha_g^*(\prid)\).
  \end{enumerate}
\end{theorem}

\begin{proof}
  By Proposition~\ref{pro:restricted_ideal_groupoid}, we may apply
  Theorem~\ref{the:prime_primitive}, and the quasi-oribt map
  \(\varrho\colon \check{B} \to \check{A}/{\sim}\)
  exists.  As in the group case, the
  map \(\pi\colon \Prime(A) \to \Prime(\Ideals^B(A))\)
  maps \(\prid\in\Prime(A)\)
  to the intersection
  \(\bigcap_{\prid'\in \overline{G\cdot \prid}} \prid'
  = \bigcap_{g\in s^{-1}(p(\prid))} \alpha_g^*(\prid)\).
  Hence \(\prid_1 \sim \prid_2\)
  if and only if \(\overline{G\cdot \prid_1}=\overline{G\cdot \prid_2}\).
\end{proof}

It is quite easy to generalise the result above to saturated Fell
bundles over locally compact Hausdorff groupoids.  The
Packer--Raeburn Stabilisation Trick shows that any such Fell
bundle is equivariantly Morita equivalent to a groupoid action in
the usual sense.  This equivariant Morita equivalence preserves all
structure that we are interested in, that is, the reduced and full
crossed products, the ideal lattices of the \(\Cst\)\nb-algebras
involved, and the restriction and induction maps between them (compare
Proposition~\ref{pro:equivalent_Fell_bundles} about the group case).
Thus all our results generalise to this situation.  Invariant ideals for
Fell bundles over groupoids are described as in the group case: an
ideal \(I\in\Ideals(A)\) is invariant with respect to a Fell bundle
\((A_g)_{g\in G}\) over~\(G\) if and only if \(I_{r(g)}\cdot A_g =
A_g \cdot I_{s(g)}\) for all \(g\in G\).

The case of Fell bundles over non-Hausdorff locally compact
groupoids is a different matter: the Packer--Raeburn Stabilisation
Trick fails in this case, even for rather important Fell bundles
(see~\cite{Buss-Meyer:Actions_groupoids}).  We cannot treat this
case in general because the Disintegration Theorem has not yet been
shown for Fell bundles over non-Hausdorff groupoids.  Muhly and
Williams~\cite{Muhly-Williams:Equivalence.FellBundles} only treat
Fell bundles over Hausdorff locally compact groupoids, which is the
case when the Packer--Raeburn Stabilisation Trick allows to replace
them by ordinary actions.  Renault's original proof of the
Disintegration Theorem in~\cite{Renault:Representations} covers
Green twisted actions of non-Hausdorff groupoids on continuous
fields of \(\Cst\)\nb-algebras
over~\(G^0\).
So the results above also hold in this case.  The results also hold
for arbitrary Fell bundles over \'etale groupoids because these may
be rewritten as Fell bundles over inverse semigroups.

\subsection{Quantum group crossed products}
\label{sec:quantum_group_crossed}

Now we consider crossed products for \(\Cst\)\nb-quantum groups.
Here our general results are incomplete because all implications
that we proved in our study of group and groupoid crossed products
now require some technical assumptions.

Let~\((C,\Delta)\)
be a \(\Cst\)\nb-quantum
group as in \cites{Woronowicz:Mult_unit_to_Qgrp,
  Soltan-Woronowicz:Multiplicative_unitaries}, that is, it is
generated by a manageable multiplicative unitary
\(W\in\U(\Hils\otimes\Hils)\)
on some separable Hilbert space~\(\Hils\).
We are going to study restricted and induced ideals for reduced
crossed products with~\((C,\Delta)\).
Throughout this section, a \emph{morphism} from~\(A\)
to~\(B\)
is a non-degenerate \Star{}homomorphism from~\(A\)
to the multiplier algebra~\(\Mult(B)\).
These extend uniquely to strictly continuous unital homomorphisms
\(\Mult(A)\to\Mult(B)\) and thus form a category.

A (right) \emph{coaction} of~\(C\)
on a \(\Cst\)\nb-algebra~\(A\)
(this is also called an action) is a faithful morphism
\(\alpha\colon A\to\Mult(A\otimes C)\)
such that \(\alpha(A)\cdot (1\otimes C) = A\otimes C\)
and the following coassociativity diagram commutes:
\[
\begin{tikzcd}
  A \arrow[r,"\alpha"] \arrow[d,"\alpha"]&
  \Mult(A\otimes C) \arrow[d,"\alpha\otimes \Id_C"] \\
  \Mult(A\otimes C)  \arrow[r,"\Id_A \otimes \Delta"] &
  \Mult( A\otimes C \otimes C)
\end{tikzcd}
\]
Any \(\Cst\)\nb-quantum
group has a dual quantum group \((\hat{C},\hat{\Delta})\).
The multiplicative unitary~\(W\)
generates faithful representations of \(C\)
and~\(\hat{C}\)
on~\(\Hils\), which we write down as morphisms to \(\Comp(\Hils)\).

The \emph{reduced crossed product}
\(B\defeq A\rtimes_{\alpha,\red} \hat{C}\)
is defined as the \(\Cst\)\nb-subalgebra
of \(\Mult(A\otimes\Comp(\Hils))\)
generated by \(\alpha(A)\cdot (1\otimes \hat{C})\),
where we view \(A\otimes C\)
and \(1\otimes\hat{C}\)
as non-degenerate \(\Cst\)\nb-subalgebras
of \(\Mult(A\otimes\Comp(\Hils))\).
In fact, the product \(\alpha(A)\cdot (1\otimes\hat{C})\)
is already a \(\Cst\)\nb-algebra.  This follows if
\begin{equation}
  \label{eq:crossed_product_identity}
  \alpha(A)\cdot (1\otimes\hat{C}) = (1\otimes\hat{C}) \cdot \alpha(A).
\end{equation}
It implies that there are canonical morphisms
\(A\to A\rtimes_{\alpha,\red} \hat{C} \leftarrow \hat{C}\).
In the generality in which we are working,
\eqref{eq:crossed_product_identity} is proved
in~\cite{Meyer-Roy-Woronowicz:Twisted_tensor}: the crossed product is
an instance of the twisted tensor product \(A\boxtimes \hat{C}\)
introduced there using right coactions of \(C\)
and~\(\hat{C}\)
on \(A\)
and~\(\hat{C}\).
The assumption that \(\alpha(A)\cdot (1\otimes C) = A\otimes C\),
not just \(\alpha(A)\cdot (1\otimes C) \subseteq A\otimes C\),
is crucial for this proof.

\begin{definition}
  \label{def:invariant_qg}
  An ideal \(I\idealin A\)
  is called \emph{\(\alpha\)\nb-invariant}
  if \(\alpha(I)\cdot (1\otimes C) \subseteq I\otimes C\).
  Let~\(I\) be \(\alpha\)\nb-invariant.
  We say that the action~\(\alpha\)
  \emph{restricts to an action on~\(I\)}
  if \(\alpha(I)\cdot (1\otimes C) = I\otimes C\).
  It \emph{descends to an action on~\(A/I\)}
  if the map \(\dot\alpha\colon A/I \to \Mult(A/I \otimes C)\)
  induced by \(\alpha\colon A \to \Mult(A \otimes C)\)
  is injective.
\end{definition}

The conditions written in Definition~\ref{def:invariant_qg} are
exactly what is needed to get induced coactions on~\(I\)
and~\(A/I\),
respectively.  Namely, \(I\)
is in the kernel of the composite map
\[
A\xrightarrow{\alpha} \Mult(A\otimes C) \to \Mult(A/I \otimes C),
\]
so that~\(\dot\alpha\) makes sense.  And
\(\dot\alpha(A/I)\cdot (1\otimes C) = A/I\otimes C\)
and the coassociativity of \(\alpha|_I\) and~\(\dot\alpha\)
follow from the corresponding properties of~\(\alpha\).
If \(I\idealin A\) is an \(\alpha\)\nb-invariant ideal, then it is,
in general, unclear whether~\(\alpha\)
restricts to~\(I\)
or descends to~\(A/I\).
Our general theory only works well when this is the case for all
invariant ideals.

\begin{lemma}
  \label{lem:invariant_symmetric}
  Let \(I\idealin A\) be an invariant ideal.
  \begin{enumerate}
  \item If~\(\alpha\) restricts to~\(I\), then~\(I\) is symmetric
    and \(i(I) = I\rtimes_{\alpha,\red} \hat{C}\).
  \item If~\(\alpha\)
    descends to~\(A/I\)
    or if~\(\hat{C}\) is discrete and~\(\alpha\)
    restricts to~\(I\),
    then~\(I\) is restricted.
  \end{enumerate}
\end{lemma}

\begin{proof}
  Assume that~\(\alpha\)
  restricts to~\(I\).  Then
  \begin{align*}
    I\cdot B
    &= \alpha(I) \cdot \alpha(A) \cdot (1\otimes \hat{C})
    = \alpha(I) \cdot (1\otimes \hat{C})
    = I\rtimes_{\alpha,\red} \hat{C},\\
    B\cdot I
    &= (1\otimes \hat{C}) \cdot \alpha(A) \cdot \alpha(I)
    =  (1\otimes \hat{C}) \cdot \alpha(I)
    = I\rtimes_{\alpha,\red} \hat{C}.
  \end{align*}
  Thus
  \(I\cdot B = I\rtimes_{\alpha,\red} \hat{C} = I\cdot B\),
  that is, the ideal~\(I\)
  is symmetric and \(i(I) = I\rtimes_{\alpha,\red} \hat{C}\).
  If~\(\hat{C}\)
  is discrete, then \(A\subseteq B\).
  Then symmetric ideals are restricted by
  Lemma~\ref{lem:symmetric_induced_ideal}.

  Now assume that~\(\alpha\)
  descends to~\(A/I\).
  Let \(p\colon B \to
  A/I\rtimes_{\dot\alpha,\red} \hat{C}\)
  be the canonical quotient map and let \(J\defeq \ker p\).
  We claim that \(r(J) = I\).
  The inclusion \(I\subseteq r(J)\)
  follows because \(p|_{i(I)}=0\).
  The canonical morphism
  \(A/I \to \Mult(A/I\rtimes_{\dot\alpha,\red} \hat{C})\)
  is injective.  Hence the kernel of the map
  \(A\to \Mult(B/J)\)
  is contained in~\(I\).  That is,
  \(r(J)\subseteq I\).
\end{proof}

\begin{lemma}
  \label{lem:qg_restricted_invariant}
  If~\(C\)
  is an exact \(\Cst\)\nb-algebra,
  then restricted ideals are invariant.
\end{lemma}

\begin{proof}
  Let \(J\in\Ideals(B)\).
  We want to argue as in the group case, using that the coaction
  on~\(A\)
  extends to an inner coaction on~\(B\).
  The coaction~\(\Delta\)
  on~\(C\)
  is implemented by the multiplicative unitary~\(W\),
  that is, \(W(c\otimes 1) W^* = \Delta(c)\)
  holds as an operator on \(\Hils\otimes\Hils\)
  for all \(c\in C\).  Hence the inner automorphism
  \[
  \Ad_{1_A\otimes W}\colon
  A\otimes \Comp(\Hils) \otimes C \to A\otimes \Comp(\Hils) \otimes C
  \]
  maps \(\alpha(a)_{12} \in \Mult(A\otimes C\otimes C)\)
  to
  \((\Id_A\otimes \Delta)(\alpha(a)) = (\alpha\otimes
  \Id_C)(\alpha(a))\)
  by the coassociativity of~\(\alpha\).
  Since
  \(W\in\U\Mult(\hat{C}\otimes C) \subseteq
  \U\Mult(B \otimes C)\),
  the map~\(\Ad_W\)
  is an inner automorphism
  of \(B \otimes C\).
  Hence we get a morphism
  \[
  \beta\colon B\to \Mult(B\otimes C),\qquad
  b\mapsto W(b\otimes 1)W^*.
  \]
  The computations above show that the diagram
  \[
  \begin{tikzcd}
    A \arrow[d,"\alpha"] \arrow[r,hookrightarrow]&
    \Mult(B) \arrow[d,"\beta"] \\
    \Mult(A\otimes C)  \arrow[r,hookrightarrow] &
    \Mult(B \otimes C)
  \end{tikzcd}
  \]
  commutes.  (We do not claim and do not need that~\(\beta\)
  is a coaction as defined above.  It is indeed an injective and
  coassociative map with
  \(\beta(B)\cdot (1\otimes C) \subseteq B\otimes C\),
  but equality here is unclear.)
  Since ideals are invariant under inner automorphisms,
  \(\beta(J\otimes C) \subseteq J\otimes C\).
  Hence we get an induced morphism
  \[
  \dot\beta\colon B/J\to \Mult(B/J\otimes C),\qquad
  [b]\mapsto W([b]\otimes 1)W^*.
  \]
  Let \(\varphi'\colon A\to \Mult(B)\to \Mult(B/J)\) be the
  canonical map.  Its kernel is \(r(J)\),
  compare~\eqref{eq:induced_hom_on_quotient}.
  Let \(D \subseteq \Mult(B/J)\)
  be its image.  So there is a \(\Cst\)\nb-algebra
  extension \(r(J)\into A \onto D\).
  It induces another extension
  \(r(J) \otimes C\into A\otimes C \onto D \otimes C\)
  because~\(C\)
  is exact.

  We are going to prove
  \(\alpha(a)\cdot (1\otimes c) \in r(J) \otimes C\)
  for all \(a\in r(J)\), \(c\in C\).
  This means that the ideal \(r(J)\) in~\(A\) is invariant.
  By the exact sequence above, it suffices to show that
  \(\alpha(a)\cdot (1\otimes c)\)
  is mapped to~\(0\)
  in \(D\otimes C\).
  The canonical map \(D\otimes C \to \Mult(B/J\otimes C)\)
  is injective because \(D\subseteq \Mult(B/J)\):
  this is a known property of the minimal tensor
  product~\(\otimes\).
  Finally, the commuting diagram
  \[
  \begin{tikzcd}
    A \arrow[d,"\alpha"] \arrow[r,hookrightarrow]&
    \Mult(B) \arrow[d,"\beta"] \arrow[r]&
    \Mult(B/J) \arrow[d,"\dot\beta"] \\
    \Mult(A\otimes C)  \arrow[r,hookrightarrow] &
    \Mult(B \otimes C) \arrow[r]&
    \Mult(B/J\otimes C)
  \end{tikzcd}
  \]
  shows that
  \[
  (\varphi'\otimes \Id_C)(\alpha(a)\cdot (1\otimes c))
  = \dot\beta\circ\varphi'(a)\cdot (1\otimes c)
  = 0
  \]
  because \(a\in r(J) = \ker \varphi'\).
\end{proof}

\begin{theorem}
  \label{the:qg_ideals}
  Let~\((C,\Delta_C)\)
  be a \(\Cst\)\nb-quantum
  group.  Let \((A,\alpha)\)
  be a \(\Cst\)\nb-algebra
  with a coaction of~\(C\),
  and let \(B\defeq A\rtimes_{\red,\alpha} \hat{C}\)
  be the reduced crossed product.  Assume~\(\Prime^B(A)\)
  to be first countable, \(C\)
  to be an exact \(\Cst\)\nb-algebra, and one of the following:
  \begin{enumerate}
  \item \label{the:qg_ideals1}%
    the quantum group~\(\hat{C}\)
    is discrete and~\(\alpha\)
    restricts to all invariant ideals;
  \item \label{the:qg_ideals2}%
    for any invariant ideal \(I\idealin A\),
    \(\alpha\) restricts to~\(I\) and descends to~\(A/I\).
  \end{enumerate}
  Then an ideal in~\(A\)
  is restricted if and only if it is invariant, and all restricted
  ideals are symmetric.  There is a continuous, open surjection
  \(\pi\colon \check{A} \onto \Prime^B(A)\),
  mapping \(\prid\in\check{A}\)
  to the largest invariant ideal contained in~\(\prid\).
  Define the equivalence relation~\(\sim\)
  by \(\prid_1 \sim \prid_2\)
  if and only if the largest invariant ideals contained in
  \(\prid_1\)
  and~\(\prid_2\)
  are equal.  Then~\(\pi\)
  induces a homeomorphism
  \(\check{A}/{\sim} \cong \Prime^B(A)\).
\end{theorem}

\begin{proof}
  Since~\(C\)
  is exact by assumption, Lemma~\ref{lem:qg_restricted_invariant}
  shows that restricted ideals are invariant.  The converse holds by
  Lemma~\ref{lem:invariant_symmetric}, which also shows that
  invariant ideals are symmetric.  Since intersections of invariant
  ideals are again invariant, condition~\eqref{cond:C1} holds here.
  Since all restricted ideals are symmetric,
  Corollary~\ref{cor:symmetric_nice} implies
  condition~\eqref{cond:C2}.  Now Theorem~\ref{the:prime_primitive}
  gives most of the remaining assertions.  The description
  of~\(\pi\)
  is equivalent to~\eqref{eq:definition_of_pi} because an ideal is
  restricted if and only if it is invariant.
\end{proof}

Next we show that induced ideals in the crossed product
\(B\defeq A\rtimes_{\alpha,\red} \hat{C}\)
are invariant for the dual coaction of \((\hat{C},\hat{\Delta})\).
In the generality of \(\Cst\)\nb-quantum groups
generated by manageable multiplicative unitaries, the dual coaction
is defined in~\cite{Meyer-Roy-Woronowicz:Twisted_tensor} through the
functoriality of~\(\boxtimes\)
for covariant homomorphisms.
It is the unique left coaction
\(\gamma\colon B \to \hat{C} \otimes B\)
with the following property.
Let \(a\in A\), \(c\in \hat{C}\).
Then
\(\gamma(\alpha(a)\cdot (1\otimes c)) = \alpha(a)_{23}
\hat{\Delta}(c)_{13}\)
in \(\Mult(\hat{C} \otimes B) \subseteq
\Mult(\hat{C} \otimes A\otimes \Comp(\Hils))\),
where the subscripts are the leg numbering notation.

\begin{lemma}
  \label{lem:qg_induced_invariant}
  Let \(I\in\Ideals(A)\).
  Then the induced ideal \(i(I)\)
  in~\(B\)
  is invariant under the dual coaction, and the dual coaction
  restricts to~\(I\).
\end{lemma}

\begin{proof}
  In leg numbering notation,
  \[
  i(I)
  = \hat{C}_2 \cdot \alpha(A) \cdot \alpha(I) \cdot \alpha(A) \cdot
  \hat{C}_2
  = \hat{C}_2 \cdot \alpha(I) \cdot \hat{C}_2.
  \]
  The dual coaction maps this to
  \(\gamma(i(I)) = \hat{\Delta}(\hat{C})_{13}\cdot \alpha(I)_{23}
  \cdot \hat{\Delta}(\hat{C})_{13}\).  We must prove
  \[
  \gamma(i(I)) \cdot \hat{C}_1 = \hat{C} \otimes B.
  \]
  Notice that we claim
  equality here, not just an inclusion.  The proof uses that all
  \(\Cst\)\nb-quantum groups are bisimplifiable, that is,
  \[
  \hat{\Delta}(\hat{C}) \cdot (\hat{C}\otimes 1)
  = \hat{C} \otimes \hat{C} =\hat{\Delta}(\hat{C}) \cdot (1\otimes\hat{C}).
  \]
  So
  \begin{multline*}
    \gamma(i(I)) \cdot \hat{C}_1
    = \hat{\Delta}(\hat{C})_{13}\cdot \alpha(I)_{23} \cdot
    \hat{\Delta}(\hat{C})_{13} \cdot \hat{C}_1
    = \hat{\Delta}(\hat{C})_{13}\cdot \alpha(I)_{23} \cdot \hat{C}_1 \cdot \hat{C}_3
    \\= \hat{\Delta}(\hat{C})_{13} \cdot \hat{C}_1\cdot \alpha(I)_{23} \cdot \hat{C}_3
    = \hat{C}_1 \cdot \hat{C}_3 \cdot \alpha(I)_{23} \cdot \hat{C}_3
    = \hat{C}_1 \otimes i(I).
  \end{multline*}
  This says that~\(i(I)\) is an invariant ideal and that the dual
  coaction restricts to it.
\end{proof}

\begin{example}
  \label{ex:group_coaction}
  Let \(C=\Cred(G)\)
  for a locally compact group~\(G\)
  with the usual comultiplication
  \(\Delta(\lambda_g) \defeq \lambda_g \otimes \lambda_g\).
  This is a quantum group.
  If~\(G\) is amenable, then any injective coaction
  \(\delta\colon A\to A\otimes C\)
  of~\((C,\Delta)\)
  satisfies \(\delta(A)\cdot (1\otimes C) = A\otimes C\)
  by \cite{Katayama:Takesaki_Duality}*{Proposition~6}.
  Thus~\(\delta\)
  restricts to any \(\delta\)\nb-invariant
  ideal \(I\idealin A\).
  Moreover, \(\delta\) descends to an action on~\(A/I\)
  by
  \cite{GootmanLazar.DualityCrossedProduct}*{Proposition~3.14.(iii)}.
  (Note that \(\operatorname{Ind}^{\wedge}(I)\)
  in~\cite{GootmanLazar.DualityCrossedProduct}
  coincides with the kernel of
  \(A\xrightarrow{\alpha} \Mult(A\otimes C) \to \Mult(A/I \otimes
  C)\).)
  Then we are in the situation of
  Theorem~\ref{the:qg_ideals}.\ref{the:qg_ideals2}.
  Its conclusions imply
  \cite{GootmanLazar.DualityCrossedProduct}*{Proposition~4.6}.
\end{example}

\end{document}